\documentclass[leqno, a4paper, 10pt]{amsart}
\usepackage{amsmath, color, a4wide, shuffle}
\usepackage{amssymb, amssymb, latexsym, mathrsfs, slashed, tikz-cd, mathtools, mathabx, shuffle}

 \newtheorem{theorem}{Theorem}[section]
 \newtheorem{lemma}[theorem]{Lemma}
 \newtheorem{proposition}[theorem]{Proposition}
 \newtheorem{corollary}[theorem]{Corollary}

 \newtheorem*{theorem*}{Theorem}
\newtheorem*{proposition*}{Proposition}
\newtheorem*{lemma*}{Lemma}

\theoremstyle{definition}
\newtheorem{definition}[theorem]{Definition}

 \theoremstyle{remark}
 \newtheorem{example}[theorem]{Example}
 \newtheorem{remark}[theorem]{Remark}

\newcommand{\op}[1]{\operatorname{#1}}

\newcommand{\acou}[2]{\ensuremath{\left\langle #1 , #2 \right\rangle}} 
\newcommand{\acoup}[2]{\ensuremath{\left(#1,#2\right)}}

\newcommand{\brak}[1]{\ensuremath{\langle #1\rangle}}

\newcommand{\Tr}{\ensuremath{\op{Tr}}}
\newcommand{\tr}{\op{tr}}

\def\XXint#1#2#3{{\setbox0=\hbox{$#1{#2#3}{\int}$}
\vcenter{\hbox{$#2#3$}}\kern-.5\wd0}}

\newcommand{\Str}{\op{Str}}

\newcommand{\ind}{\op{ind}}
\newcommand{\Ch}{\op{Ch}}
\newcommand{\bCh}{\op{\mathbf{Ch}}}

\newcommand{\C}{\ensuremath{\mathbb{C}}} 
 
\newcommand{\K}{\ensuremath{\mathbf{k}}} 
\newcommand{\N}{\ensuremath{\mathbb{N}}} 
\newcommand{\Q}{\ensuremath{\mathbb{Q}}} 
\newcommand{\R}{\ensuremath{\mathbb{R}}} 
 
\newcommand{\Z}{\ensuremath{\mathbb{Z}}}

\newcommand{\fS}{\ensuremath{\mathfrak{S}}}

\newcommand{\Ca}[1]{\ensuremath{\mathcal{#1}}}
\newcommand{\cA}{\Ca{A}}

\newcommand{\cE}{\Ca{E}}

\newcommand{\cH}{\ensuremath{\mathcal{H}}}
\newcommand{\cI}{\ensuremath{\mathcal{I}}}

\newcommand{\cL}{\ensuremath{\mathcal{L}}}

\newcommand{\cN}{\Ca{N}}

\newcommand{\cU}{\ensuremath{\mathcal{U}}}
\newcommand{\cV}{\ensuremath{\mathcal{V}}}

\newcommand{\sC}{\mathscr{C}}
\newcommand{\sE}{\mathscr{E}}

\newcommand{\sS}{\ensuremath{\slashed{S}}}

\newcommand{\sD}{\ensuremath{\slashed{D}}}

\newcommand{\ev}{{\textup{ev}}}
\newcommand{\odd}{{\textup{odd}}}

\newcommand{\coker}{\op{coker}} 

\newcommand{\Hom}{\op{Hom}}

\newcommand{\End}{\ensuremath{\op{End}}}

\newcommand{\dom}{\op{dom}}

\newcommand{\bt}{\bullet}
\newcommand{\hotimes}{\hat\otimes}

\newcommand{\CM}{\textup{CM}}

\newcommand{\CS}{\textup{CS}}

\newcommand{\bC}{\mathbf{C}}
\newcommand{\nC}{\overline{C}}

\newcommand{\HC}{\op{HC}}

\newcommand{\HP}{\op{HP}}
\newcommand{\bHC}{\op{\mathbf{HC}}}

\newcommand{\bHP}{\op{\mathbf{HP}}}
\newcommand{\bfH}{\op{\mathbf{H}}}
\newcommand{\Tot}{\op{Tot}}
\newcommand{\Diag}{\op{Diag}}

\newcommand{\AW}{\op{AW}}

\newcommand{\wb}{\widebar{b}}
\newcommand{\wB}{\widebar{B}}
\newcommand{\wC}{\widebar{C}}
\newcommand{\wbd}{\widebar{d}}
\newcommand{\ws}{\widebar{s}}
\newcommand{\wt}{\widebar{t}}
\newcommand{\wT}{\widebar{T}}

\newcommand{\wpi}{\widebar{\pi}}
 
\newcommand{\wepsn}{\widebar{\varepsilon}^\natural}

\newcommand{\OG}{{\overline{\Gamma}}}

\newcommand{\restr}[2]{\left. #1\right|_{#2}}

\renewcommand{\frown}{\smallfrown}
\renewcommand{\smile}{\smallsmile}

\numberwithin{equation}{section}

\begin{document}

\title{NONCOMMUTATIVE GEOMETRY AND CONFORMAL GEOMETRY. I. LOCAL INDEX FORMULA AND CONFORMAL INVARIANTS}
 \author{Rapha\"el Ponge}
 \address{Department of Mathematical Sciences, Seoul National University, Seoul, South Korea}
 \email{ponge.snu@gmail.com}
 \author{Hang Wang}
 \address{School of Mathematical Sciences, University of Adelaide, Adelaide, Australia}
 \email{hang.wang01@adelaide.edu.au}

 \thanks{R.P.\ was partially supported by Research Resettlement Fund and Foreign Faculty Research Fund of Seoul National University, and by 
 grants 2013R1A1A2008802 and 2016R1D1A1B01015971 of National Research Foundation of Korea. H.W.\ was partially supported by Discovery Early Career Researcher Award DE160100525 of Australian Research Council}
  
\begin{abstract}
    This paper is part of a series of articles on noncommutative geometry and conformal geometry. In this paper, we reformulate the local index formula in conformal geometry in such a way to take into account of the action of conformal diffeomorphisms.  We also construct and compute a whole new family of geometric conformal invariants associated with conformal diffeomorphisms. This includes conformal invariants associated with equivariant characteristic classes. The approach of this paper involves using various tools from noncommutative geometry, such as twisted spectral triples and cyclic theory. An important step is to establish the conformal invariance of the Connes-Chern character of the conformal Dirac spectral triple of Connes-Moscovici. Ultimately, however, the main results of the paper are stated in a purely differential-geometric fashion. 
 \end{abstract}

\maketitle

\section{Introduction}
This paper is part of a series of papers applying techniques of noncommutative geometry in the settings of conformal geometry and noncommutative versions of conformal geometry (\emph{cf}.~\cite{PW:JNCG16, PW:AIM15} for other papers of this series). The present paper has two main results. 

The first main result is the reformulation of the local index formula in conformal geometry in such a way to take into account the action of conformal diffeomorphisms (Theorem~\ref{thm:LIF-Conformal-Geometry}). This is in the line of research initiated by Alain Connes and Henri Moscovici (see~\cite{CM:GAFA95, CM:TGNTQF, Mo:LIFTST}). In particular, the existence of such an index formula was hinted to by Moscovici (see~\cite[Remark~3.8]{Mo:LIFTST}). 

The second main result is the construction and explicit computation of a whole new families of geometric conformal invariants attached with conformal diffeomorphisms (Theorem~\ref{thm:Confinv.main}). This includes conformal invariants associated with equivariant characteristic classes (Theorem~\ref{thm:Confinv.characteristic-classes}). These conformal invariants are not of the same type of the conformal invariants considered by Alexakis~\cite{Al:DGCI} in his proof of the conjecture of Deser-Schwimmer~\cite{DS:GCCAAD} on the characterization of global conformal invariants. Our conformal invariants take into account the action of the group of conformal diffeomorphisms, i.e., the conformal gauge group. This should of be of interest in the context of string theory and conformal gravity. In addition, our conformal invariants also shed new light on the conformal indices of Branson-{\O}rsted~\cite{BO:CM86, BO:DGA91} (see end of Section~\ref{sec:Conformal-Invariants}).

To put things into context, recall that, given an action of a group $G$ on a manifold $M$, the orbit space $M/G$ has a natural manifold structure when the action is free and proper. In general, however, $M/G$ need not even be Hausdorff, and so it is difficult to use $M/G$ to extract much geometric or even topological information for general group actions. The solution proposed by noncommutative geometry (in the sense of Connes~\cite{Co:NCG}) is to trade the space $M/G$ for the crossed-product algebra $C^\infty(M)\rtimes G$. This algebra always makes sense independently how wild the action might be; the non-Hausdorffness of the quotient only pertains through the noncommutativity of $C^\infty(M)\rtimes G$. 

More generally, noncommutative geometry provides us with a general  framework that allows us to deal with a variety of geometric problems whose noncommutative natures prevent us from using classical techniques. In this framework spaces are traded for algebras which formally play the roles of the algebras of functions on (ghost) noncommutative spaces. From this point of view, a noncommutative manifold is usually represented by a spectral triple $(\cA, \cH, D)$, where $\cA$ is an algebra represented by bounded operators on the Hilbert space $\cH$ and $D$ is a selfadjoint unbounded operator that commutes with $\cA$ up to bounded operators. This operator accounts for the geometric properties of the spectral triple. The paradigm of a spectral triple is provided by a Dirac spectral triple associated with a Dirac operator on a spin manifold.  

Much like with Dirac operators, the datum of a spectral triple $(\cA,\cH, D)$ gives rise to an index problem $\cE \rightarrow \ind D_{\nabla^\cE}$ in terms of the Fredholm indices of operators obtained by twisting the operator $D$ with connections on noncommutative vector bundles $\cE$ (i.e., finitely generated projective modules over $\cA$). These indices are $K$-theory invariants.  The relevant noncommutative analogue of de Rham theory is provided by the cyclic homology and cyclic cohomology of Connes~\cite{Co:Ober81, Co:CRAS83, Co:NCDG} (see also Tsygan~\cite{Ts:UMN83}). In particular, Connes~\cite{Co:NCDG} established the following index formula, 
\begin{equation}
 \ind D_{\nabla^\cE}= \acou{\Ch(D)}{\Ch(\cE)} \qquad \forall (\cE, \nabla^\cE),
 \label{eq:Intro.CC-character-index}
\end{equation}
where $\Ch(\cE)$ is the Chern character in cyclic homology and $\Ch(D)$ is a class in cyclic cohomology, called the Connes-Chern character. Under further assumptions, 
the Connes-Chern character is represented by the CM cocycle~\cite{CM:CMP93, CM:GAFA95}. The components of the CM cocycle are given 
by formulas that are local in the sense they involve a version for spectral triples of the noncommutative residue trace 
of Guillemin~\cite{Gu:NPWF} and Wodzicki~\cite{Wo:LISA}. Together with~(\ref{eq:Intro.CC-character-index}) this provides us with the local index formula 
in noncommutative geometry. In the case of a Dirac spectral triple this enables us to recover the local index 
formula of Atiyah-Singer~\cite{AS:IEO1, AS:IEO3} (see \cite{CM:GAFA95, Po:CMP1}). 

Given an arbitrary group of diffeomorphisms $G$ of a manifold $M$, there is a technical difficulty for constructing a spectral triple over the crossed-product algebra $C^\infty(M)\rtimes G$ due to the lack of geometric structures that are invariant under the full diffeomorphism group (at the notable exception of the manifold structure itself). In particular, there need not exist a $G$-invariant metric. 
As observed by Connes~\cite{Co:Kyoto83}, up to a Connes-Thom isomorphism, the problem can be solved by passing to the total space of the metric bundle $P\rightarrow M$. Following up this idea Connes-Moscovici~\cite{CM:GAFA95} built a spectral triple over the algebra $C^\infty_c(P)\rtimes G$. The computation of the CM cocycle of that spectral triple involves Hopf cyclic cohomology and expresses the Connes-Chern character in terms of Gelf'and-Fuchs cohomology classes (see~\cite{CM:CMP98, CM:EM01}). 

As observed by Connes-Moscovici~\cite{CM:TGNTQF},  when $G$ is a group of conformal diffeomorphisms, we actually can have a spectral triple over $C^\infty(M)\rtimes G$ instead of $C^\infty_c(P)\rtimes G$ at the expense of twisting the definition of a spectral triple. More precisely, a \emph{twisted} 
spectral $(\cA,\cH,D)_{\sigma}$ is like an ordinary spectral triple at the exception that the boundedness of commutators $[D,a]$, $a \in \cA$, is replaced by that 
of twisted commutators, 
\begin{equation*}
   [D,a]_{\sigma}:=Da-\sigma(a)D, \qquad a \in \cA,  
\end{equation*}
where $\sigma$ is a given automorphism of the algebra $\cA$. A natural example is given by conformal deformations of ordinary spectral triples, 
\begin{equation*}
    (\cA,\cH,D) \longrightarrow (\cA,\cH,kDk)_{\sigma}, \qquad \sigma(a)=k^{2}ak^{-2},
\end{equation*}where $k$ ranges over positive invertible elements of $\cA$ (see~\cite{CM:TGNTQF}). Another example due to Connes-Moscovici~\cite{CM:TGNTQF} is 
the conformal Dirac spectral triple $(C^{\infty}(M)\rtimes G, L^{2}_{g}(M,\sS),\sD_{g})_{\sigma_{g}}$ associated with the 
Dirac operator $\sD_{g}$ on a compact spin Riemannian manifold $(M,g)$ and a group $G$ of diffeomorphisms preserving a given 
conformal structure $\sC$ (see also Section~\ref{sec:Conformal-CC-character}). 

As for ordinary spectral triples, the datum of a twisted spectral triple $(\cA,\cH,D)_{\sigma}$ gives rise to an index 
problem by twisting the operator $D$ with $\sigma$-connections (see~\cite{PW:KJM16}). The resulting Fredholm indices are 
computed by a Connes-Chern character $\Ch(D)_{\sigma}$ defined as a class in the (ordinary) cyclic cohomology of $\cA$ 
(see~\cite{CM:TGNTQF, PW:KJM16}). However, although Moscovici~\cite{Mo:LIFTST} produced an Ansatz for a version of the CM cocycle for 
twisted spectral triples, this Ansatz has been verified only in a special class of examples (see~\cite{Mo:LIFTST}). In 
particular, to date it is still not known if Moscovici's Asantz even holds for conformal deformations of ordinary 
spectral triples.    

For our purpose it is important to define the Connes-Chern character in 
the cyclic cohomology of \emph{continuous} cochains, which is smaller than the cyclic cohomology of arbitrary 
cochains. We show that for a natural class of twisted spectral triples 
$(\cA,\cH,D)_{\sigma}$ over locally convex algebras, which we call \emph{smooth} twisted spectral triples, the Connes-Chern 
character descends to a class $\bCh(D)_{\sigma}\in \op{\mathbf{HP}}^{0}(\cA)$, where $\op{\mathbf{HP}}^{\bt}(\cA)$ is the periodic cyclic 
cohomology of continuous cochains (Proposition~\ref{prop:CCC-index-formula2}). We also show the invariance of this class under 
conformal perturbations of twisted spectral triples (Proposition~\ref{prop:conf.Dirac.same.Ch.LCA}). 

The construction of the conformal Dirac spectral triple $(C^{\infty}(M)\rtimes G, L^{2}_{g}(M,\sS),\sD_{g})_{\sigma_{g}}$ 
associated with a conformal class $\sC$ on a compact spin manifold $M$ alluded to above \emph{a priori} depends on a choice of the metric 
$g\in \sC$. As it turns out, up to equivalences of twisted spectral triples, changing a metric within $\sC$ 
only amounts to make a conformal deformation of twisted spectral triples (see~\cite{Mo:LIFTST}). Combining this with the invariance under conformal 
deformations of the Connes-Chern character mentioned above then shows that the Connes-Chern character 
$\bCh(\sD_{g})_{\sigma_{g}}\in \op{\mathbf{HP}}^{0}(C^{\infty}(M)\rtimes G)$ is an invariant of the conformal class $\sC$ (Theorem~\ref{thm:CC.conf.inv}). 

The next step is the computation of the Connes-Chern character $\bCh(\sD_{g})_{\sigma_{g}}$. As mentioned above, we do not know if Moscovici's Ansatz holds for the conformal Dirac spectral triple $(C^{\infty}(M)\rtimes G, L^{2}_{g}(M,\sS),\sD_{g})_{\sigma_{g}}$. Therefore, we cannot make use of the CM cocycle given by that Ansatz. However, we can observe that the conformal invariance of $\bCh(\sD_{g})_{\sigma_{g}}$ allows us to use \emph{any} metric in the conformal class $\sC$. In particular, we may use a $G$-invariant metric. This is a convenient choice of metric, since in that case the automorphism $\sigma_g$ is the identity, and so the associated conformal Dirac spectral triple becomes an \emph{ordinary} spectral triple. The existence of such metrics is ensured by Ferrand-Obata theorem, provided that the conformal structure $\sC$ is non-flat,  i.e., it is not equivalent to the conformal structure of the round sphere. 

Assuming the non-flatness of $\sC$ we are reduced to the computation of the Connes-Chern character of an ordinary equivariant Dirac spectral triple $(C^{\infty}(M)\rtimes 
G, L^{2}_{g}(M,\sS), \sD_{g})$, where $G$ is a group of isometries. In this case, the Connes-Chern character is computed in~\cite{PW:JNCG16} by making use of a CM cocycle representative and by computing this cocycle by means of heat kernel techniques. This allows us to express the Connes-Chern  character $\bCh(\sD_{g})_{\sigma_{g}}$ in terms of explicit polynomials in curvatures and normal curvatures of the fixed-point manifolds of diffeomorphisms in $G$ (see Theorem~\ref{thm:LIF-Conformal-Geometry} for for the precise statement). These formulas are reminiscent of the local equivariant index theorem for Dirac operators of Atiyah-Segal-Singer~\cite{AS:IEO2,AS:IEO3}. This provides us with a reformulation of the local index formula in conformal-diffeomorphism invariant geometry. 

As mentioned above, pairing the Connes-Chern character $\Ch(\sD_g)_{\sigma_g}$ with any (even) periodic cyclic cycle  over the algebra $\cA_G:=C^\infty(M)\rtimes G$ produces an invariant of the conformal class $\sC$. This potentially produces a large family of conformal invariants. To understand and compute these invariants we need an explicit construction of (periodic) cyclic cycles over $\cA_G$. There is a large amount of work on the cyclic homology of crossed-product algebras, especially in the case of group actions on manifolds or varieties (see, e.g., \cite{BC:CCDG, Br:AIF87, Br:Preprint87, BN:KT94, BDN:AIM17, Co:Kyoto83, Co:NCG, Cr:KT99, FT:LNM87, GJ:Crelle93, NPPT:Crelle06, Ni:InvM90, Po:CRAS4, Po:CRAS5}). In his seminal work on cyclic cohomology and foliations in the early 80s, Connes~\cite{Co:Kyoto83} constructed an explicit cochain map and quasi-isomorphism from the equivariant cohomology into the homogeneous component of the periodic cyclic cohomology of $\cA_G$. However, since then it was not until the recent notes~\cite{Po:CRAS4, Po:CRAS5} that further explicit quasi-isomorphisms were exhibited. 

We refer to Section~\ref{sec:cyclic-homology-crossed-product} for a review of the results of~\cite{Po:CRAS4, Po:CRAS5} in the case of group actions on manifolds. This involves a large amount of homological algebra. Combining these results with the explicit computation of the Connes-Chern character enables us to define and compute explicitly a whole family of geometric conformal invariants (Theorem~\ref{thm:Confinv.main}). These conformal invariants are associated with conformal diffeomorphisms and mixed equivariant cycles. The mixed equivariant homology was introduced in~\cite{Po:CRAS5}. It mixes group homology and de Rham homology. It is also the natural receptacle for the cap product of group homology with equivariant cohomology. As a result, we obtain conformal invariants associated with equivariant characteristic classes (Theorem~\ref{thm:Confinv.characteristic-classes}). We also relate these conformal invariants to the conformal indices of Branson-{\O}rsted~\cite{BO:CM86, BO:DGA91} (see end of Section~\ref{sec:Conformal-Invariants}).

The paper is organized as follows. In Section~\ref{sec:TwistedST}, we review the main definitions and examples regarding twisted 
spectral triples and the construction of their index maps. In Section~\ref{sec:CyclicCohomChernChar}, we review the main facts about cyclic 
cohomology, cyclic homology and the Chern character in cyclic homology. In Section~\ref{sec:CC-character}, we review the construction 
of the Connes-Chern character of a twisted spectral triple. In Section~\ref{sec:smooth-twisted-ST}, we show that for smooth twisted spectral 
triples the Connes-Chern character descends to the cyclic cohomology of continuous cochains. In Section~\ref{sec:Invariance-CC-character}, we 
establish the invariance of the Connes-Chern character under conformal deformations. In Section~\ref{sec:Conformal-CC-character}, after reviewing 
the construction of the conformal Dirac spectral triple, we prove that its Connes-Chern Character is a conformal invariant. In 
Section~\ref{sec:LIF-Conformal-Geometry}, we compute this Connes-Chern character by using the results of~\cite{PW:JNCG16}. This provides us with a local 
index formula in conformal-diffeomorphism invariant geometry. In Section~\ref{sec:cyclic-homology-crossed-product}, we review the results of~\cite{Po:CRAS4, Po:CRAS5} on the construction of explicit quasi-isomorphisms computing the cyclic homology of crossed-product algebras in the case of group actions on manifolds. 
Finally, in Section~\ref{sec:Conformal-Invariants} we combine the results of the previous sections to construct and compute our conformal invariants. 

Throughout this paper we focus exclusively on the even-dimensional case. The main results of this paper have analogues in odd dimension. However, a complete treatment would require a comprehensive account on the index theory for odd twisted spectral triples. This will be dealt with in forthcoming papers. 
In addition, the conformal invariants that we construct in this paper are associated with conformal diffeomorphisms of finite order (i.e., periodic diffeomorphisms). It is also possible to define and compute conformal invariants associated with non-periodic conformal diffeomorphisms. We also postpone this to a forthcoming paper.
Finally, the non-flatness leaves out the important case of the action of the projective Poincar\'e group $\op{PO}(n,1)$ on the round sphere $\mathbb{S}^n$. The main issue is the computation of the Connes-Chern character in this case. As pointed out in~\cite{Mo:LIFTST} this requires a refinement of the analytical considerations of~\cite{PW:JNCG16} in order to take into account the contribution of the non-elliptic components of $\op{PO}(n,1)$. 

\section*{Acknowledgements}
The authors would like to thank Alain Connes, Sasha Gorokhovsky, Xiaonan Ma, Henri Moscovici, Bent {\O}rsted, Hessel Posthuma, Xiang Tang and Bai-Ling Wang for helpful discussions related to the subject matter of this paper. They also would like to thank the following institutions for their hospitality during the 
preparation of this manuscript: Seoul National University (HW); University of Adelaide, University of California at 
Berkeley, Kyoto University (Research Institute of Mathematical Sciences and  Department of Mathematics), University Paris-Diderot (Paris 7), and Mathematical 
Science Center of Tsinghua University (RP); Australian National University, Chern 
Institute of Mathematics of  Nankai University, and Fudan University (RP+HW).

\section{Index Theory for Twisted Spectral Triples.}\label{sec:TwistedST}
In this section, we recall how the datum of a twisted spectral triple naturally gives rise to an index 
problem~(\cite{CM:TGNTQF, PW:KJM16}). The exposition closely follows that of~\cite{PW:KJM16} (see also~\cite{Mo:EIPDNCG} 
for the case of ordinary spectral triples). 

In the setting of noncommutative geometry, the role of manifolds is played by spectral triples.  

\begin{definition}
A spectral triple $(\cA, \cH, D)$ is given by 
\begin{enumerate}
\item A $\Z_2$-graded Hilbert space $\mathcal{H}=\mathcal{H}^+\oplus \mathcal{H}^-$.
\item A unital $*$-algebra $\mathcal{A}$ represented by bounded operators on $\cH$ preserving its $\Z_{2}$-grading.
\item A selfadjoint unbounded operator $D$ on $\mathcal{H}$ such that
\begin{enumerate}
    \item $D$ maps $\dom (D)\cap \cH^{\pm}$ to $\cH^{\mp}$. 
    \item The resolvent $(D+i)^{-1}$ is a compact operator.
    \item $a \dom (D) \subset \dom (D)$ and $[D, a]$ is bounded for all $a \in \cA$. 
\end{enumerate}
\end{enumerate}   
\end{definition}

\begin{remark}
    The condition (3)(a) implies that with respect to the splitting $\mathcal{H}=\mathcal{H}^+\oplus \mathcal{H}^-$ the 
    operator $D$ takes the form,
    \begin{equation*}
        D=
        \begin{pmatrix}
            0 & D^{-} \\
            D^{+} & 0
        \end{pmatrix}, \qquad D^{\pm}:\dom(D)\cap \cH^{\pm}\rightarrow \cH^{\mp}.
    \end{equation*}
\end{remark}

\begin{example}\label{ex:TwistedST.Dirac-ST}
The paradigm of a spectral triple is given by a Dirac spectral triple, 
\begin{equation*}
( C^{\infty}(M), L^{2}_{g}(M,\sS), \sD_{g}),
\end{equation*}
where $(M^{n},g)$ is a compact spin Riemannian manifold of even dimension $n$ and $\sD_{g}$ is its Dirac operator 
acting on the spinor bundle $\sS=\sS^{+}\oplus \sS^{-}$.     
\end{example}

The definition of a twisted spectral triple is  similar to that of an ordinary spectral triple, 
except for some ``twist'' given by the conditions (3) and (4)(b) below.

 \begin{definition}[\cite{CM:TGNTQF}]\label{TwistedSpectralTriple}
A twisted spectral triple $(\cA, \cH, D)_{\sigma}$ is given by 
\begin{enumerate}
\item A $\Z_2$-graded Hilbert space $\mathcal{H}=\mathcal{H}^{+}\oplus \mathcal{H}^{-}$.
\item A unital $*$-algebra $\mathcal{A}$ represented by even bounded operators on $\cH$.

\item An automorphism $\sigma:\cA\rightarrow \cA$ such that $\sigma(a)^{*}=\sigma^{-1}(a^{*})$ for all $a\in 
\cA$. 
\item An odd selfadjoint unbounded operator $D$ on $\mathcal{H}$ such that 
\begin{enumerate}
    \item The resolvent $(D+i)^{-1}$ is compact.
    \item $a (\dom D) \subset \dom D$ and $[D, a]_{\sigma}$ is bounded for all $a \in \cA$, where the twisted 
    commutator $[D,a]_{\sigma}$ is defined by
   \begin{equation}
       [D, a]_{\sigma}:=Da-\sigma(a)D \qquad \forall a \in \cA.
       \label{eq:TwistedST.twisted-commutators}
   \end{equation}
\end{enumerate}
\end{enumerate}   
\end{definition}

The relevance of the notion of twisted spectral triples in the setting of conformal geometry stems from the following 
observation.  Let $( C^{\infty}(M), L^{2}_{g}(M,\sS), \sD_{g})$ be a Dirac spectral triple as in Example~\ref{ex:TwistedST.Dirac-ST}, and 
consider a conformal change of metrics, 
\begin{equation*}
    \hat{g}=k^{-2}g, \qquad k\in C^{\infty}(M), \ k>0.
\end{equation*}We then can form a Dirac spectral triple $(C^{\infty}(M), L^{2}_{\hat{g}}(M,\sS), \sD_{\hat{g}})$ 
associated with the new metric $\hat{g}$. 
As it turns out (see~\cite{PW:KJM16}) this spectral triple is equivalent to the following spectral triple, 
\begin{equation*}
   \left ( C^{\infty}(M), L^{2}_{g}(M,\sS), \sqrt{k}\;\sD_{g}\!\sqrt{k}\right).  
\end{equation*}
We note that the above spectral triple continues to make sense if we only assume $k$ to be a positive Lipschitz 
function on $M$. 

More generally, let $(\cA,\cH,D)$ be an ordinary spectral triple and $k$ a positive element of $\cA$. If we replace $D$ by its 
\emph{conformal deformation} $kDk$ then, when $\cA$ is noncommutative, the triple $(\cA,\cH,kD{k})$ needs not be an ordinary spectral triple. However, as the 
following result shows, it always gives rise to a twisted spectral triple. 

\begin{proposition}[\cite{CM:TGNTQF}] Let $\sigma:\cA\rightarrow \cA$ be the automorphism defined by
    \begin{equation}
        \sigma(a):= k^{2}a k^{-2} \qquad \forall a \in \cA.
        \label{eq:TwistedST.sigmah}
    \end{equation}Then $\left(\mathcal{A}, \mathcal{H}, kDk\right)_{\sigma}$ is a twisted spectral triple. 
\end{proposition}

\begin{remark}
A more elaborate version of the above example, and the main focus of this paper, is the conformal Dirac spectral triple 
of Connes-Moscovici~\cite{CM:TGNTQF}. This is a twisted spectral triple taking into account of the action of 
the group of diffeomorphisms preserving a  given conformal structure. We refer to Section~\ref{sec:ConformalDiracST} 
for a review of this example.  
\end{remark}

\begin{remark}
 We refer to~\cite{CM:TGNTQF, GMT:T3sSTQSS, CT:GBTNC2T, IM:CPEST, Mo:LIFTST, PW:AIM15} for the constructions of 
various other examples of twisted spectral triples.
\end{remark}

In what folllows, we let $(\cA, \cH, D)_{\sigma}$ be a twisted spectral triple. In addition, we let $\cE$ be a finitely generated projective right module over $\cA$, i.e., 
 $\cE$ is a direct summand of a free module $\cE_{0}\simeq \cA^{N}$.

\begin{definition}[\cite{PW:KJM16}]
\label{def:sigmaTranslate}
A $\sigma$-translation of  $\cE$ is given by a pair $(\cE, \sigma^{\cE})$, 
 equipped with the following data:
\begin{itemize}
    \item[(i)]  $\cE^{\sigma}$ is a finitely generated projective right module over $\cA$. 

    \item[(ii)] $\sigma^{\cE}$ is a $\C$-linear isormophism from $\cE$ onto $\cE^{\sigma}$ such that 
\begin{equation}
    \sigma^{\cE}(\xi a)=\sigma^{\cE}(\xi)\sigma(a)  \qquad \text{for all $\xi\in \cE$ and $a\in \cA$}.
    \label{eq:sigma^E(xi a)}
\end{equation}    
\end{itemize}
\end{definition}

\begin{remark}
    When $\sigma=\op{id}$ the condition~(\ref{eq:sigma^E(xi a)}) simply means that $\sigma^{\cE}$ is a right-module 
    isomorphism, and so in this case we can take $(\cE, \op{id})$ as a $\sigma$-translation of $\cE$. 
\end{remark}

\begin{remark}
   Suppose that $\cE=e\cA^{N}$ with 
    $e=e^{2}\in M_{N}(\cA)$, $N\geq 1$.  Note that the automorphism $\sigma$ lifts to a $\C$-linear isomorphism $\sigma:\cA^{N}\rightarrow 
   \cA^{N}$ given by 
   \begin{equation*}
       \sigma(\xi)=(\sigma(\xi_{j})) \qquad \text{for all $\xi=(\xi_{j})\in \cA^{N}$}.
   \end{equation*}
   This induces a $\C$-linear isomorphism $\sigma_{e}:e\cA^{N}\rightarrow \sigma(e)\cA^{N}$ 
   satisfying~(\ref{eq:sigma^E(xi a)}). Therefore, in this case we may take $(\sigma(e)\cA^{N},\sigma_{e})$ as 
   a $\sigma$-translation of $\cE=e\cA^{N}$.  
\end{remark}

Throughout the rest of the section we let $(\cE^{\sigma},\sigma^{\cE})$ be a $\sigma$-translation of $\cE$. In addition, we consider the 
space of twisted 1-forms, 
 \begin{equation*}
    \Omega^{1}_{D,\sigma}(\cA)=\left\{\Sigma a^{i}[D, b^{i}]_{\sigma}: a^{i}, b^{i} \in\cA \right\},
\end{equation*}where $[D ,b]_{\sigma}$, $b \in \cA$, is the twisted 
commutator~(\ref{eq:TwistedST.twisted-commutators}). We note that $ \Omega^{1}_{D,\sigma}(\cA)$ is a subspace of 
$\cL(\cH)$. 
The \emph{twisted} differential $d_{\sigma}:\cA \rightarrow \Omega^{1}_{D,\sigma}(\cA)$ is given by 
\begin{equation}
\label{eq:TwistedDifferential}
    d_{\sigma}a:= [D,a]_{\sigma} \qquad \forall a \in \cA.
\end{equation}
This is a $\sigma$-derivation, in the sense that
\begin{equation}
\label{eq:SigmaDerivation}
    d_{\sigma}(ab)=(d_{\sigma}a)b+\sigma(a)d_{\sigma}b\qquad \forall a, b \in \cA. 
\end{equation}
In particular, this implies that $  \Omega^{1}_{D,\sigma}(\cA)$ is a sub-$(\cA,\cA)$-bimodule of $\cL(\cH)$. 

\begin{definition}
\label{def:SigmaConnection}
A $\sigma$-connection on $\cE$ is a $\C$-linear map $\nabla^{\cE}: \cE\rightarrow 
\cE^{\sigma}\otimes_{\cA}\Omega^1_{D,\sigma}(\cA)$ such that 
\begin{equation}
\label{eq:SigmaConnectionModuleMulti}
 \nabla^{\cE}(\xi a)=(\nabla^{\cE}\xi) a+\sigma^{\cE}(\xi)\otimes d_{\sigma}a \qquad \forall \xi\in\cE \ \forall a\in\cA.   
\end{equation}    
 \end{definition}

\begin{example}
 Suppose that $\cE=e\cA^{N}$ with $e=e^{2}\in M_{N}(\cA)$. Then a natural $\sigma$-connection on $\cE$ is the \emph{Grassmannian $\sigma$-connection} $\nabla_{0}^{\cE}$ 
 defined by 
 \begin{equation}
 \label{eq:GrassmannianSigmaConnection}
     \nabla_{0}^{\cE}\xi= \sigma(e)(d_{\sigma}\xi_{j}) \qquad \text{for all $\xi=(\xi_{j})$ in $\cE$}.
 \end{equation}    
\end{example}

\begin{definition}
     A Hermitian metric on $\cE$ is a map $\acoup{\cdot}{\cdot}:\cE\times \cE \rightarrow \cA$ such that
     \begin{enumerate}
        \item $\acoup{\cdot}{\cdot}$ is $\cA$-sesquilinear, i.e., it is $\cA$-antilinear with respect to the first 
        variable and $\cA$-linear with respect to the second variable. 
        
         \item  $(\cdot, \cdot)$ is positive, i.e., $\acoup{\xi}{\xi}\geq 0$ for all $\xi \in \cE$.
     
         \item   $(\cdot, \cdot)$ is nondegenerate, i.e., $\xi \rightarrow \acoup{\xi}{\cdot}$ 
         is an $\cA$-antilinear isomorphism from $\cE$ onto its $\cA$-dual $\Hom_{\cA}(\cE,\cA)$. 
     \end{enumerate}
 \end{definition}

 \begin{example}\label{ex:CanonicalHermitianStructure}
  The canonical Hermitian structure on the free module $\cA^{N}$ is given by
  \begin{equation}
  \label{eq:CanonicalHermitianStructure}
      \acoup{\xi}{\eta}_{0}=\xi_{1}^{*}\eta_{1}+\cdots + \xi_{q}^{*}\eta_{q} \qquad \text{for all $\xi=(\xi_{j})$ and 
      $\eta=(\eta_{j})$ in $\cA^{N}$}.
  \end{equation}
  It induces a Hermitian metric on any direct summand $\cE=e\cA^{N}$, $e=e^{2}\in M_{N}(\cA)$, (see, e.g.,~\cite{PW:KJM16}).
\end{example}

From now on we assume that $\cE$ and its $\sigma$-translation $\cE^{\sigma}$ carry Hermitian metrics.  We denote by $\cH(\cE)$ the pre-Hilbert space 
consisting of $\cE\otimes_\cA \cH$ equipped with the Hermitian inner product, 
 \begin{equation}
 \label{eq:HermitianInnerProductH(E)}
     \acou{\xi_{1}\otimes \zeta_{1}}{\xi_{2}\otimes \zeta_{2}}:= \acou{\zeta_{1}}{(\xi_{1},\xi_{2})\zeta_{2}}, \qquad 
     \xi_{j}\in \cE,  \zeta_{j} \in \cH, 
 \end{equation}where $\acoup{\cdot}{\cdot}$ is the Hermitian metric of $\cE$. It can be shown that $\cH(\cE)$ is actually a 
 Hilbert space and its topology is independent of the choice of the Hermitian inner product of $\cE$ (see, e.g.,~\cite{PW:KJM16}). 
 We also note there is a natural $\Z_{2}$-grading on $\cH(\cE)$ given by
 \begin{equation}
 \label{eq:Z_2GradingH(E)}
     \cH(\cE)=\cH^{+}(\cE)\oplus \cH^{-}(\cE), \qquad \cH^{\pm}(\cE):=\cE\otimes_{\cA}\cH^{\pm}.
 \end{equation}
We denote by $\cH(\cE^{\sigma})$ the similar $\Z_{2}$-graded Hilbert space associated with $\cE^{\sigma}$ and its 
 Hermitian metric. 
 
 Let $\nabla^{\cE}$ be a $\sigma$-connection on $\cE$. 
 Regarding $\Omega_{D,\sigma}^{1}(\cA)$ as a subalgebra of 
$\cL(\cH)$ we have a natural left-action $c:\Omega_{D,\sigma}^{1}(\cA)\otimes_{\cA}\cH \rightarrow \cH$ given by 
\begin{equation*}
    c(\omega\otimes \zeta)=\omega(\zeta) \qquad \text{for all $\omega\in \Omega_{D,\sigma}^{1}(\cA)$ and $\zeta \in 
    \cH$}.
\end{equation*}
We then denote by $c\left(\nabla^{\cE}\right)$ the composition $(1_{\cE^{\sigma}}\otimes c)\circ (\nabla^{\cE}\otimes 
1_{\cH}):\cE\otimes\cH \rightarrow \cE^{\sigma}\otimes\cH$.  Thus, for $\xi\in \cE$ and $\zeta \in \cH$, and upon writing
 $\nabla^{\cE}\xi=\sum \xi_{\alpha} \otimes\omega_{\alpha}$ with $\xi_{\alpha}\in \cE^{\sigma}$ and 
 $\omega_{\alpha}\in  \Omega^{1}_{D,\sigma}(\cA)$, we have
 \begin{equation}
 \label{eq:CliffordNabla}
        c\left(\nabla^{\cE}\right)(\xi\otimes\zeta)=\sum \xi_{\alpha}\otimes\omega_{\alpha}(\zeta). 
\end{equation}

In what follows we regard the domain of $D$ as a left $\cA$-module, which is possible since the action of $\cA$ on 
$\cH$ preserves the domain $\dom D$. 

\begin{definition}
 \label{def:DNablaE}
   The operator $D_{\nabla^{\cE}}:\cE\otimes_{\cA} \dom (D) \rightarrow \cH(\cE^{\sigma})$ is  defined by
\begin{equation}
\label{eq:Index.Dnabla}
    D_{\nabla^{\cE}}(\xi\otimes \zeta):=\sigma^{\cE}(\xi)\otimes D\zeta + c(\nabla^{\cE})(\xi\otimes \zeta) \qquad 
    \text{for all $\xi\in \cE$ and $\zeta \in \dom D$}.
\end{equation}   
\end{definition}

\begin{remark} 
Although the operators $\sigma^{\cE}$, $D$ and $\nabla^{\cE}$ are not module maps, the operator is well defined as a 
linear map with domain $\cE\otimes_{\cA} \dom (D)$ (see~\cite{PW:KJM16}). 
\end{remark}

\begin{remark}
    With respect to the $\Z_2$-gradings~(\ref{eq:Z_2GradingH(E)}) for $\cH(\cE)$ and $\cH(\cE^{\sigma})$ the operator $D_{\nabla^{\cE}}$ takes the form, 
\begin{equation}
\label{eq:Index.Dnabla+-}
 D_{\nabla^{\cE}}=   \begin{pmatrix}
    0    & D_{\nabla^{\cE}}^{-} \\
        D_{\nabla^{\cE}}^{+} & 0
    \end{pmatrix}, \qquad D_{\nabla^{\cE}}^{\pm}:\cE\otimes_{\cA}\dom D^{\pm}\longrightarrow \cH^{\mp}(\cE^{\sigma}).
\end{equation}That is, $D_{\nabla^\cE}$ is an odd operator.
\end{remark}

\begin{example}[See \cite{PW:KJM16}]
\label{ex:D_grassDesigma}
Suppose that $\cE=e\cA^{N}$ with $e=e^{2}\in M_{N}(\cA)$ and let $\nabla_{0}^{\cE}$ be the Grassmanian $\sigma$-connection 
of $\cE$. Then up to the canonical unitary identifications $\cH(\cE)\simeq e\cH^{N}$ and $\cH(\cE^{\sigma})\simeq 
\sigma(e)\cH^{N}$ the operators $D_{\nabla^{\cE}_{0}}$ agrees with
\begin{equation*}
    \sigma(e)(D\otimes 1_{N}):e(\dom D)^{N}\longrightarrow \sigma(e)\cH^{N}. 
\end{equation*}
\end{example}

\begin{example}
In the case of a Dirac spectral triple $( C^{\infty}(M), L^{2}_{g}(M,\sS), \sD_{g})$, we may take $\cE$ to be the module 
$C^{\infty}(M,E)$ of smooth sections of a vector bundle $E$ over $M$. Any Hermitian metric and connection on $E$ give 
rise to a Hermitian metric and a connection $\nabla^{\cE}$ on $\cE$. Furthermore, if we set $\cH=L^{2}_{g}(M,E)$, then, under the natural identification $\cH(\cE)\simeq 
L^{2}(M,\sS\otimes E)$, the operator $(\sD_{g})_{\nabla^{\cE}}$ agrees with the usual twisted 
Dirac operator $\sD_{\nabla^{E}}$ as defined, e.g., in~\cite{BGV:HKDO}. 
\end{example}

\begin{proposition}[\cite{PW:KJM16}]
\label{lem:Dnabla.Fred}
    The operator $D_{\nabla^{\cE}}$ is closed and Fredholm. 
\end{proposition}

Note that the above result implies that the operators $D_{\nabla^{\cE}}^{\pm}$ in~(\ref{eq:Index.Dnabla+-}) are 
Fredholm.  This leads us to the following definition.

\begin{definition}
    The index of  the operator $D_{\nabla^{\cE}}$ is
\begin{equation*}
    \ind D_{\nabla^{\cE}}= \frac{1}{2}\left( \ind D_{\nabla^{\cE}}^{+}- \ind D_{\nabla^{\cE}}^{-}\right),
\end{equation*}where $\ind D_{\nabla^{\cE}}^{\pm}$ is the usual Fredholm index of $D_{\nabla^{\cE}}^{\pm}$ (i.e., 
$\ind D_{\nabla^{\cE}}^{\pm}= \dim \ker D_{\nabla^{\cE}}^{\pm}-\dim \coker D_{\nabla^{\cE}}^{\pm}$). 
\end{definition}

\begin{remark}\label{rmk:Twisted-ST.ribbon-condition}
 In general the indices $\pm \ind D_{\nabla^{\cE}}^{\pm}$ do not agree, so that it is natural take their mean to define the 
 index of $D_{\nabla^{\cE}}$. However, as shown in~\cite{PW:KJM16}, the index $\ind D_{\nabla^{\cE}}$ is an integer when the automorphism $\sigma$ is \emph{ribbon}, in the sense 
 there is another automorphism $\tau:\cA\rightarrow \cA$ such that $\sigma=\tau\circ \tau$ and $\tau(a)^{*}=\tau^{-1}(a^{*})$ for all $a\in 
\cA$. The ribbon condition is satisfied in all main examples of twisted spectral triples (see~\cite{PW:KJM16}). When this 
condition holds, it is shown in~\cite{PW:AIM15} that we always can endow $\cE$ with a ``$\sigma$-Hermitian 
structure'' (see~\cite{PW:AIM15} for the precise definition). In that case, for any  connection
$\nabla^{\cE}$ compatible with  the $\sigma$-Hermitian structure, 
\begin{equation*}
   \ind D_{\nabla^{\cE}}=\dim \ker  D_{\nabla^{\cE}}^{+}-\dim \ker  D_{\nabla^{\cE}}^{-} \in \Z.
\end{equation*}
The above formula is the generalization of the usual formula for the index of a Dirac operator twisted by a Hermitian connection on a 
Hermitian vector bundle.
\end{remark}

As it turns out the index $ \ind D_{\nabla^{\cE}}$ only depends on the $K$-theory class of $\cE$ (see~\cite{PW:KJM16}). More precisely, we 
have the following result. 

\begin{proposition}[\cite{CM:TGNTQF, PW:KJM16}]
\label{prop.IndexTwisted-connection}
There is a unique additive map $\ind_{D,\sigma}:K_{0}(\cA)\rightarrow \frac{1}{2}\Z$ such that, for any finitely 
generated projective module $\cE$ over $\cA$ and any $\sigma$-connection $\nabla^{\cE}$ on $\cE$, we have
    \begin{equation*}
        \ind_{D,\sigma}[\cE]= \ind D_{\nabla^{\cE}}. 
    \end{equation*}
\end{proposition}

\section{Cyclic Homology and Cyclic Cohomology}
\label{sec:CyclicCohomChernChar}
Cyclic homology and cyclic cohomology were introduced by 
Connes~\cite{Co:Ober81, Co:CRAS83, Co:NCDG} as the relevant noncommutative analogue of de Rham theory (see also Tsygan~\cite{Ts:UMN83}). In this section, we review the main background on cyclic homology and cyclic cohomology that is needed for this paper. We refer to the monographs~\cite{Co:NCDG, Co:NCG, Lo:CH} for 
more comprehensive accounts on these topics. 

\subsection{Preamble: Mixed complexes, cyclic modules, and $S$-maps} In what follows we let $\K$ be an arbitrary unital ring. According to Burghelea~\cite{Bu:CM86} and Kassel~\cite{Ka:JAlg87} a mixed complex of $\K$-modules is given by a datum $(C_\bt,b,B)$, where $C_m$, $m\geq 0$, are (left) $\K$-modules and $b:C_\bt \rightarrow C_{\bt-1}$ and $B:C_\bt\rightarrow C_{\bt+1}$ are $\K$-module maps such that $b^2=B^2=bB+Bb=0$. The \emph{cyclic complex} of a mixed complex $C=(C_\bt,b,B)$ is the chain complex $C^\natural=(C^\natural_\bt, b+SB)$, where $C^\natural_m=C_m\oplus C_{m-2}\oplus \cdots$, $m\geq 0$, and $S:C_\bt^\natural \rightarrow C_{\bt-2}^\natural$ is the canonical projection obtained by factorizing out $C_m$. The homology of this chain complex  is called the \emph{cyclic homology} of the mixed complex $C$ and is denoted by $\HC_\bt(C)$. 

The operator $S:C_\bt^\natural \rightarrow C_{\bt-2}^{\natural}$ is called the \emph{periodicity operator}.  The \emph{periodic cyclic complex} is the $\Z_2$-graded chain complex $C^\sharp=(C^\sharp_\bt, b+B)$, where $C_i^\sharp = \prod_{q \geq 0}C_{2q+i}$, $i=0,1$. Equivalently, we have $C_\bt^\sharp=\varprojlim_S C_{2q+\bt}$, where $\varprojlim_S$ is the projective limit of the projective system defined by the operators 
$S:C_{2q+\bt}\rightarrow C_{2q+\bt-2}$, $q\geq 1$. The homology of the chain complex $C^\sharp$ is called  the \emph{periodic cyclic homology} of the mixed complex $C$. It is denoted by $\HP_\bt(C)$. 

Following Connes~\cite{Co:CRAS83}, a \emph{cyclic $\K$-module} is given by a datum $(C_\bt,d,s,t)$, where $C_m$, $m\geq 0$, are $\K$-modules and the $\K$-module maps $d:C_\bt\rightarrow C_{\bt-1}$, $s:C_\bt\rightarrow C_{\bt+1}$ and $t:C_\bt \rightarrow C_\bt$ define a simplicial module structure with faces $d_j=t^{j}dt^{-(j+1)}$ and degeneracies $s_j=t^{j+1}dt^{-(j+1)}$. In addition, the operator $t$ is required to be cyclic, in the sense that $t^{m+1}=1$ on $C_m$. Any cyclic $\K$-module $C=(C_\bt, d,s,t)$ gives rise to a mixed complex $C=(C_\bt, b,B)$, where $b=\sum_{j=0}^m(-1)^j d_j$ and $B=(1-\tau)sN$ on $C_m$ with $\tau =(-1)^m t$ and $N=1+\tau +\cdots + \tau^m$. We then define the cyclic  and periodic cyclic homologies of $C$ as the the cyclic  and periodic cyclic homologies of that mixed complex.

Given mixed complexes $C=(C_\bt, b,B)$ and $\wC=(C_\bt, b,B)$, an $S$-map in the sense of Kassel~\cite{Ka:JAlg87} is a chain map $f:C^\natural_\bt \rightarrow \wC^\natural_\bt$ which is compatible with the $S$-operator. Any  such map uniquely decomposes as $f=\sum f^{(j)} S^j$, where $f^{(j)}:C_\bt \rightarrow \wC_{\bt+2j}$ is a $\K$-module map of degree $2j$ such that $[b,f^{(0)}]=0$ and $[B,f^{(j)}]+[b,f^{(j+1)}]=0$ for $j\geq 0$. In particular, it uniquely extends to a chain map $f^\sharp: C^\sharp_\bt \rightarrow \wC_\bt^\sharp$ between the periodic cyclic complexes of $C$ and $\wC$. 

\subsection{Cyclic homology of algebras} From now on we let $\cA$ be a unital algebra over $\C$. Connes~\cite{Co:CRAS83} associated with any such algebra a cyclic space (i.e., a cyclic $\C$-module). This is $C(\cA)=(C_\bt(\cA),d,s,t)$, where $C_m(\cA)=\cA^{\otimes (m+1)}$, $m\geq 0$, and the structural operators $d:C_\bt(\cA)\rightarrow C_{\bt-1}(\cA)$, $s:C_\bt(\cA)\rightarrow C_{\bt+1}(\cA)$, and $t:C_\bt(\cA)\rightarrow C_{\bt}(\cA)$ are given by
\begin{gather}
d(a^{0}\otimes \cdots \otimes a^{m}) = (a^{m} a^0)\otimes a^1  \cdots \otimes a^{m-1}, 
\label{eq:cyclic.end-face-chain}  \\
 s(a^{0}\otimes \cdots \otimes a^{m}) = 1\otimes a^0  \cdots \otimes a^{m},
\label{eq:cyclic.extra-degeneracy-chain} \\
t(a^{0}\otimes \cdots \otimes a^{m}) = a^m\otimes a^0  \cdots \otimes a^{m-1}, \qquad a^j \in \cA. 
\label{eq:cyclic.cyclic-operator-chain}
\end{gather}

We get a mixed complex $(C_\bt(\cA),b,B)$, where the differential $b:C_\bt(\cA)\rightarrow C_{\bt-1}$ is the Hochschild boundary, 
 \begin{align*}
    b(a^{0}\otimes\cdots\otimes a^{m})  = & \sum_{j=0}^{m-1}(-1)^{j}a^{0}\otimes\cdots\otimes 
    a^{j}a^{j+1}\otimes\cdots\otimes a^{m}\\ &  + 
    (-1)^{m}a^{m}a^{0}\otimes\cdots\otimes a^{m-1}, \qquad a^{j}\in \cA.
\end{align*}
This gives rise to the cyclic complex $C^\natural(\cA)=(C^\natural_\bt(\cA), b+SB)$ and the periodic cyclic complex $C^\sharp(\cA)=(C^\sharp_\bt(\cA), b+B)$. Their respective homologies are called the \emph{cyclic homology} and \emph{periodic cyclic homology} of the algebra $\cA$. They are denoted by $\HC_\bt(\cA)$ and $\HP_\bt(\cA)$, respectively.

The mixed complex $(C_\bt(\cA),b,B)$ is normalized as follows. The degeneracies of the cyclic space $C(\cA)$ are $s_j=t^{j+1}st^{-(j+1)}$, $j=0,\ldots, m$. The space of degenerate $m$-chains is $D_m(\cA)=s_0(C_{m-1}(\cA))\oplus \cdots \oplus s_{m-1}(C_{m-1}(\cA))$. Here $s_j(a^0\otimes \cdots \otimes a^{m-1})=a^0 \otimes \cdots \otimes a^j \otimes 1 \otimes a^{j+1} \otimes \cdots \otimes a^{m-1}$, $a^l \in \cA$. Therefore, we see that $D_m(\cA)$ is spanned by elementary tensor products $a^0\otimes \cdots \otimes a^m$ where $a^j=1$ for some $j\in \{1,\ldots,m\}$. The differentials $b$ and $B$ preserve $D_\bt(\cA)$, and so the mixed complex $(C_\bt(\cA), b,B)$ descends to a \emph{normalized} mixed complex $\nC(\cA)=(\nC_\bt(\cA),b,B)$, with $\nC_\bt(\cA)=C_\bt(\cA)\slash D_\bt(\cA)$. We let $\nC^\natural(\cA)=(\nC^\natural_\bt(\cA), b+SB)$ and $\nC^\sharp(\cA)=(\nC^\sharp_\bt(\cA), b+B)$ be the corresponding cyclic and periodic cyclic complexes. 
The canonical projection $\pi: C_\bt(\cA) \rightarrow \nC_\bt(\cA)$ is a map of mixed complexes. At the level of Hochschild homology we obtain a quasi-isomorphism. General homological algebra arguments then imply that we also obtain quasi-isomorphisms at the level of cyclic homology and periodic cyclic homology. This enables us to describe $\HC_\bt(\cA)$ and $\HP_\bt(\cA)$ in terms of \emph{normalized} cycles. 

When $\cA$ is a locally convex algebra, the space of chains $C_m(\cA)$ is often completed into $\bC_m(\cA):=\cA^{\hotimes (m+1)}$, where $\hotimes$ is the projective topological tensor product~\cite{Gr:TTPNS, Tr:TVSDK}. In particular, given any locally convex space $X$, a linear map $\Phi:\bC_m(\cA)\rightarrow X$ is continuous iff the $(m+1)$-linear map $\cA^{m+1}\ni (a^0, \ldots, a^m)\rightarrow \Phi(a^0\otimes \cdots \otimes a^m)$ is jointly continuous. The structural operators $(d,s,t)$ in~(\ref{eq:cyclic.end-face-chain})--(\ref{eq:cyclic.cyclic-operator-chain}) uniquely extend to continuous operators on $\bC_\bt(\cA)$, so that we obtain a cyclic space  $\bC(\cA):=(\bC_\bt(\cA), d,s,t)$. 
In the same way as above, this cyclic space gives rise to cyclic and periodic cyclic complexes $\bC^\natural(\cA):=(\bC_\bt^\natural(\cA),b+SB)$ and $\bC^\sharp(\cA):=(\bC_\bt^\sharp(\cA),b+B)$. Their respective homologies are denoted by $\bHC_\bt(\cA)$ and $\bHP_\bt(\cA)$. 

\begin{example}
\label{ex:HKR.smooth.functions}
    Let $\cA=C^{\infty}(M)$, where $M$ is a closed manifold. We endow $\cA$ with its standard Fr\'echet algebra topology. In addition, we let $\Omega(M)=(\Omega^\bt(M),d)$ be the de Rham complex of (smooth) differential forms on $M$. We shall regard $\Omega(M)$ as a mixed complex $\Omega(M)=(\Omega^\bt(M), 0,d)$. The Hochschild-Kostant-Rosenberg map $\alpha: \Omega^\bt(M)\rightarrow \bC_\bt(\cA)$ is given by 
    \begin{equation}
    \label{eq:cyclic.HKR-map-chain}
        \alpha(f^{0}\otimes \cdots \otimes f^{m})=\frac{1}{m!}f^{0}df^{1}\wedge \cdots \wedge df^{m}, \qquad  f^{j}\in \cA. 
    \end{equation}
 This is a map of mixed complexes. Moreover, Connes~\cite{Co:NCDG} proved that this is a quasi-isomorphism, and so we can describe $\bHC_\bt(\cA)$ and $\bHP_\bt(\cA)$ in terms of de Rham cohomology $H^\bt(M,\C)$. Namely, 
\begin{gather*}
\bHC_m(\cA) \simeq \left( \Omega^m(M)\slash d(\Omega^{m-1}(M))\right) \oplus H^{m-2}(M,\C) \oplus \cdots, \quad m\geq 0,\\
\bHP_\bt(\cA)\simeq \bigoplus_{q\geq 0} H^{\ev/\odd}(M,\C), \quad \text{where}\ H^{\ev/\odd}(M,\C):= \bigoplus_{\text{$m$ even/odd}} H^m(M,\C). 
\end{gather*}
\end{example}

\subsection{Cyclic cohomology of algebras}
 Cyclic cohomology is the dual version of cyclic homology. There is a natural duality between $C_m(\cA)=\cA^{\otimes (m+1)}$ and the space $C^m(\cA)$ that consists of $(m+1)$-linear maps $\varphi: \cA^{m+1}\rightarrow \C$. We then get a ``cyclic co-space'' $(C^\bt(\cA),d,s,t)$, where the operators $d:C^\bt(\cA)\rightarrow C^{\bt+1}(\cA)$, $s:C^\bt(\cA)\rightarrow C^{\bt-1}(\cA)$, and $t:C^\bt(\cA)\rightarrow C^{\bt}(\cA)$ are given by  
\begin{gather}
 (d\varphi)(a^0,\ldots, a^{m+1})=  \varphi\left((a^{m+1} a^0),a^1, \ldots, a^{m}\right), 
\label{eq:cyclic.end-face-cochain}\\
 (s\varphi)(a^0,\ldots, a^{m-1}) = \varphi(1, a^0, \ldots, a^{m-1}) , 
\label{eq:cyclic.extra-degeneracy-cochain}\\
(t\varphi)(a^0,\ldots, a^m) = \varphi(a^m, a^0, \ldots, a^{m-1}), \qquad \varphi\in C^m(\cA), \ a^j \in \cA. 
\label{eq:cyclic.cyclic-operator-cochain} 
\end{gather}
We obtain the Hochschild cochain complex $(C^\bt(\cA), b)$, where $b:C^{\bt}(\cA)\rightarrow C^{\bt+1}(\cA)$ is given by
\begin{align*}
    b\varphi(a^{0},\ldots,a^{m+1})= & \sum_{j=0}^{m}(-1)^{j}\varphi(a^{0},\ldots,a^{j}a^{j+1},\ldots,a^{m+1})\\ &+ 
    (-1)^{m+1}\varphi(a^{m+1}a^{0},\ldots, a^{m}), \qquad a^{j}\in \cA.
\end{align*}

The dual $B$-operator $B:C^{\bt}(\cA)\rightarrow C^{\bt-1}(\cA)$ is $B=Ns(1-T)$, where $T=(-1)^mt$ and $N=1+T+\cdots +T^m$ on $C^m(\cA)$. We then have a duality between the cyclic complex $C^\natural(\cA)$ and the cochain complex $C_\natural(\cA):=(C^\bt_\natural(\cA),b+SB)$, where $C^m_\natural(\cA)=C^m(\cA)\oplus C^{m-2}(\cA)\oplus \cdots$ and $S:C^m_\natural(\cA)\rightarrow C^{m+2}_\natural(\cA)$ is the inclusion of $C^m_\natural(\cA)$ into $C^{m+2}_\natural(\cA)$. We also have a natural duality between the periodic cyclic complex $C^\sharp(\cA)$ and the periodic cyclic cochain complex $C_\sharp(\cA)=(C_\sharp^\bt(\cA), b+B)$, where ${ C_\sharp^i(\cA)= \bigoplus_{q\geq 0} C^{2q+i}(\cA)}$, $i=0,1$. The cohomology of $C_\sharp(\cA)$ is the \emph{periodic cyclic cohomology} of $\cA$ and is denoted by $\HP^\bt(\cA)$. 

The cyclic cohomology of $\cA$ can be defined in terms of the cochain complex $C_\natural(\cA)$. For our purpose it is more convenient to use the original definition of Connes~\cite{Co:NCDG}  in terms of cyclic cochains. A cochain $\varphi \in C^m(\cA)$ is called \emph{cyclic} when $t\varphi =(-1)^m\varphi$. We denote by $C^m_\lambda(\cA)$ the space of cyclic $m$-cochains. As the operator $b$ preserves $C_\lambda(\cA)$, we obtain a subcomplex $C_\lambda(\cA):=(C^\bt_\lambda(\cA),b)$ of the Hochschild cochain complex of $\cA$. The \emph{cyclic cohomology} is defined as the cohomology of $C_\lambda(\cA)$ and is denoted by $H_\lambda (\cA)$. 

As the operator $B$ is annihilated by cyclic cochains, the natural inclusion of $C^\bt_\lambda(\cA)\subset C_\natural^\bt(\cA)$ is a cochain map. This is actually a quasi-isomorphism (this is even a chain homotopy equivalence; see~\cite{Ka:Crelle90}). Under this quasi-isomorphism the operator $S$ is chain homotopic to Connes' periodicity operator $S_\lambda: C^\bt_\lambda(\cA)\rightarrow C^{\bt+2}_\lambda(\cA)$, which is defined as follows. On $C^m(\cA)$, $m\geq 0$, it is given by $S_\lambda= (m+2)^{-1}(m+1)^{-1} \sum_{j=1}^{m+1}(-1)^{j}S_{j}$, where 
 \begin{align}
     S_{j}\varphi(a^{0},\ldots,a^{m+2}) = & \sum_{0\leq l\leq 
     j-2}(-1)^{l}\varphi(a^{0},\ldots,a^{l}a^{l+1},\ldots,a^{j}a^{j+1},\ldots,a^{m+2})\nonumber \\ & + 
     (-1)^{j+1}\varphi(a^{0},\ldots,a^{j-1}a^{j}a^{j+1},\ldots,a^{m+2}).
     \label{eq:CC.operatorSj}
 \end{align}
If $m=2q+i$ with $q\geq 0$ and $i\in \{0,1\}$, then the inclusion of $C^m_\lambda(\cA)$ into $C^m_\natural(\cA)$ gives an inclusion into $C^i_\sharp(\cA)$. This gives a cochain map from $C^\bt_\lambda(\cA)$ into $C^\bt_\sharp(\cA)$ which is compatible with the periodicity operator $S_\lambda$ in the sense that, for any cocycle $\varphi\in C^m_\lambda(\cA)$, the cocycles $\varphi$ and $S_\lambda \varphi$ define the same class in $\HP^\bt(\cA)$. 

A cochain $\varphi\in C^m(\cA)$ is called \emph{normalized} when  $\varphi(a^{0},\ldots,a^{m})=0$ whenever $a^{j}=1$ for some $j\geq 1$.    
We denote by $\nC^{m}(\cA)$ the space of normalized $m$-cochains. The duality between $C_m(\cA)$ and $C^m(\cA)$ descends to a duality between normalized chains and normalized cochains. The operators $b$ and $B$ preserve the space of normalized cochains. Therefore, we get normalized cochain complexes $\nC_\lambda(\cA):=(\nC_\lambda^\bt(\cA), b)$ and $\nC_\sharp(\cA):=(\nC_\sharp^\bt(\cA), b+B)$, where $\nC_\lambda^m(\cA)=C_\lambda^m(\cA)\cap \nC^m(\cA)$ and $\nC^i_\sharp(\cA)=\bigoplus_{q \geq 0} \nC^{2q+i}(\cA)$. Furthermore, the inclusion of $\nC^\bt(\cA)$ into $C^\bt(\cA)$ gives rise to a quasi-isomorphism of cochain complexes from $\nC^\bt_\lambda(\cA)$ (resp., $\nC^\bt_\sharp(\cA)$) into $C^\bt_\lambda(\cA)$ (resp., $C^\bt_\sharp(\cA)$). This enables us to represent classes in $H_\lambda^\bt(\cA)$ and $\HP^\bt(\cA)$ by normalized cochains. 

When $\cA$ is a locally convex algebra we denote by $\bC^m(\cA)$ the space of (jointly) continuous $m$-cochains $\varphi:\cA^{m+1}\rightarrow \C$. The duality between $C_m(\cA)$ and $C^m(\cA)$ induces a duality between $\bC_m(\cA)$ and $\bC^m(\cA)$. The structural operators $(d,s,t)$
 in~(\ref{eq:cyclic.end-face-cochain})--(\ref{eq:cyclic.cyclic-operator-cochain}) preserve the space of continuous cochains $\bC^\bt(\cA)$. Therefore, we obtain subcomplexes $\bC_\lambda(\cA):=(\bC_\lambda^\bt(\cA), b)$ and $\bC_\sharp(\cA):=(\bC_\sharp^\bt(\cA), b+B)$, where $\bC_\lambda^m(\cA)=C_\lambda^m(\cA)\cap \bC^m(\cA)$ and $\bC^i_\sharp(\cA)=\bigoplus_{q \geq 0} \bC^{2q+i}(\cA)$. The cohomologies of these cochain complexes are denoted by $\mathbf{H}_\lambda^\bt(\cA)$ and $\bHP^\bt(\cA)$, respectively. 
We also obtain a cochain complex $\bC_\natural(\cA):=(\bC_\natural^\bt(\cA), b+SB)$, where $\bC_\natural^m(\cA)=\bC^m(\cA)\oplus \bC^{m-2}(\cA)\oplus \cdots$. The inclusion of $\bC_\lambda^\bt(\cA)$ into $\bC_\natural^\bt(\cA)$ is a quasi-isomorphism of cochain complexes. 

\begin{example}
 As in Example~\ref{ex:HKR.smooth.functions}, let $\cA=C^\infty(M)$, where $M$ is a closed manifold. In addition, let $(\Omega_\bt(M),d)$ be the de Rham complex of currents on $M$, so that $\Omega_m(M)$ is the topological dual of $\Omega^m(M)$. By duality the HKR map~(\ref{eq:cyclic.HKR-map-chain}) gives rise to a cochain map $\alpha: (\Omega_\bt,0) \rightarrow (\bC^\bt(\cA),b)$ given by
\begin{equation*}
\alpha(C)(f^0, \ldots , f^m)= \frac{1}{m!} \acou{C}{f^0 df^1\wedge \cdots \wedge d^m}, \qquad C\in \Omega_m(M), \ f^j\in \cA. 
\end{equation*}
This maps intertwines the de Rham boundary $d$ and the operator $B$, and so it gives rise to cochain maps $\alpha: \Omega_\bt^\natural(M) \rightarrow \bC_\natural^\bt(\cA)$ and $\alpha: \Omega_{\ev/\odd}^\bt(M)\rightarrow \bC^\bt_\sharp(\cA)$, where $\Omega_m^\natural(M)=\Omega_m(M)\oplus \Omega_{m-2}(M)\oplus \cdots$ and $ \Omega_{\ev/\odd}^\bt(M)=\bigoplus_{\text{$m$ even/odd}} \Omega^m(M)$. As shown by Connes~\cite{Co:NCDG} these cochain maps are quasi-isomorphisms, and so  this allows us to express the cyclic and periodic cyclic homologies of $\cA=C^\infty(M)$ in terms of the de Rham homology $H_\bt(M,\C)$. More precisely, we obtain isomorphisms, 
\begin{gather*}
 \bfH_\lambda^m(\cA)\simeq H_m(M,\C)\oplus H_{m-2}(M,\C) \oplus \cdots, \qquad m\geq 0,\\
 \bHP^\bt(\cA) \simeq H_{\ev/\odd}(M,\C), \quad  \text{where}\ H_{\ev/\odd}(M,\C):=\bigoplus_{\text{$m$ even/odd}} H_{m}(M,\C).
\end{gather*}
\end{example}

\subsection{The Chern character in cyclic homology} Given $N\in \N$, let $M_{N}(\cA)=\cA\otimes M_{N}(\C)$ be 
the algebra of $N\times N$-matrices with coefficients in $\cA$. The standard trace $\tr:M_{N}(\cA)\rightarrow \cA$ 
gives rise to a linear map $\tr:C_{\bt}(M_{N}(\cA))\rightarrow C_{\bt}(\cA)$ defined by
 \begin{equation*}
     \tr\left[ (a^{0}\otimes \mu^{0})\otimes \cdots \otimes (a^{m}\otimes \mu^{m})\right] =\tr\left[ \mu^{0}\cdots 
     \mu^{m}\right]a^{0}\otimes \cdots \otimes a^{m},  \qquad a^{j}\in \cA, \ \mu^{j}\in M_{N}(\C).
 \end{equation*}
 This map is compatible with the structural operators $(d,s,t)$ in~(\ref{eq:cyclic.end-face-chain})--(\ref{eq:cyclic.cyclic-operator-chain}), and so it gives rise to a map of cyclic spaces $\tr: C_\bt(M_N(\cA))\rightarrow C_\bt(\cA)$. 
 
 The Chern character in cyclic homology~\cite{Co:NCDG, GS:OCCTSM} is defined as follows.  Let $\cE$ be a finitely 
 generated projective module over $\cA$, so that $\cE\simeq e\cA^{N}$ for some idempotent $e\in M_{N}(\cA)$. The Chern 
 character of $e$ is the  (normalized) even chain $\Ch(e)=\left( \Ch_{2q}(e)\right)_{q\geq 0}\in 
 \nC_{0}^\sharp(\cA)$ defined by 
  \begin{equation*}
     \Ch_{0}(e)=\tr[e], \qquad \Ch_{2q}(e)=(-1)^{q}\frac{(2q)!}{q!}\tr\left[\left(e-\frac{1}{2}\right)\otimes  
     \overbrace{e\otimes \cdots  \otimes e}^{\textrm{$2q$ times}} \right], \quad q\geq 1.
 \end{equation*}
This chain is a cycle in the normalized periodic complex $\nC^\sharp(\cA)$, and so this defines a class in $\HP_0(\cA)$. This class only depends on the $K$-theory class of $e$. Therefore, we get an additive map  $\Ch: K_{0}(\cA)\rightarrow \HP_{0}(\cA)$, which is called the \emph{Chern character} in cyclic homology. 

Composing the above Chern character with the duality pairing between $\HP_0(\cA)$ and $\HP^0(\cA)$ provides us with a bilinear pairing 
$\acou{\cdot}{\cdot}:\HP^{0}(\cA)\times K_{0}(\cA)\rightarrow \C$. Given any cyclic $2q$-cocycle $\varphi$ and any idempotent $e\in M_{N}(\cA)$, it can be shown (see, e.g., \cite[Remark~6.4]{PW:KJM16}) that
\begin{equation}
\label{eq:pairing.cyclic.cocycle.Chern.char}
 \acou{[\varphi]}{[\Ch(e)]}=(-1)^{q}\frac{(2q)!}{q!}\acou{\varphi}{\tr\left[e^{\otimes 
(2q+1)}\right]}.
\end{equation}Therefore, we recover the original pairing of Connes~\cite{Co:NCDG, Co:NCG}.

When $\cA$ is a locally convex algebra, by using the inclusion of chain complexes of $C^\sharp_\bt(\cA)$ into $\bC^\sharp_\bt(\cA)$ we obtain a Chern character 
$\bCh:K_{0}(\cA) \rightarrow \bHP_{0}(\cA)$. We then obtain a  bilinear pairing $ \acou{\cdot}{\cdot}: \bHP^{0}(\cA)\times K_{0}(\cA)\rightarrow \C$. If we let $\cA=C^\infty(M)$, where $M$ is a closed manifold, then, under the CHKR isomorphism $\bHP_0(\cA)\simeq H^\ev(M,\C)$ and the Serre-Swan isomorphism $K_0(\cA)\simeq K^0(M)$, we recover the usual Chern character $\Ch:K^0(M)\rightarrow H^\ev(M,\C)$.

\section{The Connes-Chern Character of a Twisted Spectral Triple}\label{sec:CC-character}
In this section, we shall recall the construction the Connes-Chern character of a twisted spectral triple and how this 
enables us to compute the associated index map~\cite{CM:TGNTQF, PW:KJM16}. This extends to twisted spectral triples the construction of the Connes-Chern character of an ordinary 
spectral triple by Connes~\cite{Co:NCDG}. The exposition follows closely that of~\cite{PW:KJM16}. 

Let $(\cA,\cH, D)_{\sigma}$ be a twisted spectral triple. We assume that $(\cA,\cH, D)_{\sigma}$ is \emph{$p$-summable} for some $p\geq 1$, that is, 
\begin{equation}
\label{eq:p-summable}
  \Tr   |D|^{-p} <\infty. 
\end{equation}
In what follows, letting $\cL^{1}(\cH)$ be the ideal of trace-class operators on $\cH$, we denote by $\Str$ its
supertrace, i.e., $\Str [T]=\Tr [\gamma T]$, where $\gamma:=\op{id}_{\cH^{+}}-\op{id}_{\cH^{-}}$ is 
the $\Z_{2}$-grading of $\cH$. 

\begin{definition}\label{eq:CC.tauD}
Assume $D$ is invertible and let $q$ be an integer $\ge\frac{1}{2}(p-1)$. Then $\tau^{D,\sigma}_{2q}$ is the $2q$-cochain on $\cA$ defined by 
\begin{equation}
\label{eq:tau_2k}
\tau^{D, \sigma}_{2q}(a^0, \ldots, a^{2q})=c_{q}\Str\big(D^{-1}[D, a^0]_{\sigma}\cdots D^{-1}[D, a^{2q}]_{\sigma}\big) \qquad \forall a^j\in\cA,
\end{equation}
where we have set $c_{q}=\frac12(-1)^q\frac{q!}{(2q)!}$ .  
\end{definition}

\begin{remark}
    The right-hand side of~(\ref{eq:tau_2k}) is well defined since the $p$-summability 
    condition~(\ref{eq:p-summable}) and the fact that $q\ge\frac12 (p-1)$ imply that 
    \begin{equation*}
          D^{-1}[D,a^{0}]_{\sigma}\cdots D^{-1}[D,a^{q}]_{\sigma}\in \cL^{1}(\cH) \qquad \forall a^{j}\in \cA.
    \end{equation*}
\end{remark}

\begin{proposition}[\cite{CM:TGNTQF, PW:KJM16}]\label{prop:Cochian.ConnesChernChar}
Assume $D$ is invertible and let $q$ be an integer $\geq \frac12(p-1)$. Then 
\begin{enumerate}
\item The cochain $\tau^{D,\sigma}_{2q}$ is a normalized cyclic cocycle whose class in $\op{HP}^0(\cA)$ is independent of the value of $q$. 

\item For any finitely generated projective module $\cE$ over $\cA$ and $\sigma$-connection $\nabla^{\cE}$ on $\cE$, 
\begin{equation}
    \ind D_{\nabla^{\cE}}=\acou{\tau_{2q}^{D,\sigma}}{\Ch(\cE)},
    \label{eq:CC-character.Index-Formula1}
\end{equation}where $\Ch(\cE)$ is the Chern character of $\cE$ in cyclic homology. 
\end{enumerate}
\end{proposition}

When $D$ is not invertible, we can reduce to the invertible case by passing to the unital 
invertible double of $(\cA,\cH,D)_{\sigma}$  as follows.  

Consider the Hilbert space $\tilde{\cH}=\cH\oplus \cH$, which we equip with the $\Z_{2}$-grading given by
\begin{equation*}
    \tilde{\gamma}=\begin{pmatrix}\gamma & 0 \\ 0 & -\gamma \end{pmatrix},
\end{equation*}where $\gamma$ is the grading operator of $\cH$. On $\tilde{\cH}$ consider the selfadjoint operator,
\begin{equation*}
    \tilde{D}=\begin{pmatrix} D & 1 \\ 1 &  -D \end{pmatrix}, \qquad \dom(\tilde{D}):=\dom(D) \oplus \dom(D). 
\end{equation*}Noting that 
\begin{equation*}
    \tilde{D}^{2}=\begin{pmatrix}
        D^{2}+1 & 0 \\
        0 & D^{2}+1
    \end{pmatrix},
\end{equation*}we see that $\tilde{D}$ is invertible and $|\tilde{D}|^{-p}$ is a trace-class operator. Let $\tilde{\cA}=\cA\oplus \C$ be the unitalization of $\cA$ whose
product and involution are given by
\begin{equation*}
    (a,\lambda)(b,\mu)=(ab+\lambda b+\mu a,\lambda\mu), \qquad (a,\lambda)^{*}=(a^{*},\overline{\lambda}), \qquad 
    a,b\in \cA, \quad \lambda,\mu\in \C.
\end{equation*}The unit of $\tilde{\cA}$ is $1_{\tilde{\cA}}=(0,1)$. Thus, identifying any element $a\in \cA$ with 
$(a,0)$, any element $\tilde{a}=(a,\lambda)\in \tilde{\cA}$ can be uniquely written as $(a,\lambda)=a+\lambda 
1_{\tilde{\cA}}$. We represent $\tilde{\cA}$ in $\cH$ using the representation $\pi:\tilde{\cA}\rightarrow \cL(\cH)$ 
given by
\begin{equation*}
    \pi(1_{\tilde{\cA}})=1, \qquad \pi(a)=\begin{pmatrix} 
    a & 0 \\ 
    0 & 0
\end{pmatrix} \qquad \forall a \in \cA. 
\end{equation*}In addition, we extend the automorphism $\sigma$ into the 
automorphism $\tilde{\sigma}:\tilde{\cA}\rightarrow \tilde{\cA}$ given by
\begin{equation*}
    \tilde{\sigma}(a+\lambda 1_{\tilde{\cA}})=\sigma(a)+\lambda 1_{\tilde{\cA}} \qquad \forall (a,\lambda)\in \cA\times \C.
\end{equation*}
It can be verified that any twisted commutator $[\tilde{D},\pi(\tilde{a})]_{\tilde{\sigma}}$, $\tilde{a}\in 
\tilde{\cA}$, is bounded. We then deduce that $(\tilde{\cA},\tilde{\cH},\tilde{D})_{\tilde{\sigma}}$ is a 
$p$-summable twisted spectral triple. Moreover, as $\tilde{D}$ is invertible we may define the normalized cyclic 
cocycles $\tau_{2q}^{\tilde{D},\tilde{\sigma}}$, $q\geq \frac{1}{2}(p-1)$. 

\begin{definition}
    Let $q\geq \frac{1}{2}(p-1)$. Then $\overline{\tau}_{2q}^{D,\sigma}$ is the $2q$-cochain on $\cA$ defined by 
    \begin{align*}
      \overline{\tau}_{2q}^{\tilde{D}, \sigma}(a^{0},\ldots,a^{2q})   & = \tau_{2q}^{\tilde{D}, \sigma}(a^{0},\ldots,a^{2q})\\
         & = c_q\Str\big(\tilde{D}^{-1}[\tilde{D}, \pi(a^0)]_{\sigma}\cdots \tilde{D}^{-1}[\tilde{D}, 
         \pi(a^{2q})]_{\sigma}\big) \qquad \forall a^j\in\cA. 
    \end{align*}
\end{definition}

\begin{remark}
The cochain $\overline{\tau}_{2q}^{\tilde{D}}$ is the restriction to $\cA^{2q+1}$ of the cochain 
$\tau_{2q}^{\tilde{D},\tilde{\sigma}}$. Note that, as the 
restriction to $\cA$ of the representation $\pi$ is not unital, unlike in the invertible case, we don't have a 
normalized cochain.     
\end{remark}

\begin{proposition}[\cite{PW:KJM16}]\label{lem:Cochian.ConnesChernChar}
Let $q$ be an integer $\geq \frac12(p-1)$. 
\begin{enumerate}
\item The cochain $\overline{\tau}^{D,\sigma}_{2q}$ is a cyclic cocycle whose class in $\op{HP}^0(\cA)$ is independent of $q$. 

\item If $D$ is invertible, then the cocycles $\tau_{2q}^{D,\sigma}$ and 
    $\overline{\tau}^{D,\sigma}_{2q}$ are cohomologous in $\op{HP}^{0}(\cA)$.
    
\item  For any finitely generated projective module $\cE$ over $\cA$ and $\sigma$-connection $\nabla^{\cE}$ on $\cE$,
\begin{equation}
    \ind D_{\nabla^{\cE}}=\acou{\tau_{2q}^{D,\sigma}}{\Ch(\cE)} .
        \label{eq:CC-character.Index-Formula2}
    \end{equation}
\end{enumerate}   
\end{proposition}

All this leads us to the following definition. 

\begin{definition} 
The Connes-Chern character of $(\cA,\cH,D)_{\sigma}$, denoted by $\Ch(D)_{\sigma}$, is defined 
as follows:
\begin{itemize}
    \item  If $D$ is invertible, then $\Ch(D)_{\sigma}$ is the common class in $\op{HP}^{0}(\cA)$ of the 
    cyclic cocycles $\tau_{2q}^{D,\sigma}$ and $\overline{\tau}_{2q}^{D,\sigma}$, with $q\geq \frac{1}{2}(p-1)$. 

    \item  If $D$ is not invertible, then $\Ch(D)_{\sigma}$ is the common class in $\op{HP}^{0}(\cA)$ of the
    cyclic cocycles $\overline{\tau}_{2q}^{D,\sigma}$, $q\geq \frac{1}{2}(p-1)$. 
\end{itemize}
\end{definition}

\begin{remark}
    When $\sigma=\op{id}$ the Connes-Chern character is simply denoted $\Ch(D)$; this is the usual Connes-Chern Character of 
    an ordinary spectral triple constructed by Connes~\cite{Co:NCDG}. 
\end{remark}

With this definition in hand, the index formulas~(\ref{eq:CC-character.Index-Formula1}) and~(\ref{eq:CC-character.Index-Formula2}) can be merged onto the following 
result. 

\begin{proposition}[\cite{CM:TGNTQF, PW:KJM16}]\label{thm:CCC-index-formula}
 For any Hermitian finitely generated projective module $\cE$ over $\cA$ and any $\sigma$-connection $\nabla^{\cE}$ on $\cE$, 
\begin{equation}
\label{eq:IndexAsPairing}
\ind D_{\nabla^{\cE}}=\acou{\Ch(D)_{\sigma}}{\Ch(\cE)},
\end{equation}where $\Ch(\cE)$ is the Chern character of $\cE$ in cyclic homology. 
\end{proposition}

\begin{example}[\cite{Co:NCDG, BF:APDOIT, Po:CMP1}]
\label{ex:local.index.formula.DST}
    Let $(M^n, g)$ be a compact Riemannian manifold. The Connes-Chern character $\Ch(\sD_g)$ of the Dirac spectral triple 
    $(C^{\infty}(M), L^2(M, \sS), \sD_g)$ is cohomologous to the even periodic cocycle $\varphi=(\varphi_{2q})_{q\geq 
    0}$ given by
    \begin{equation}
    \label{eq:CMcocycle.Dirac}
    \varphi_{2q}(f^0, \ldots, f^{2q})=\frac{(2i\pi)^{-\frac{n}{2}}}{(2q)!}\int_M\hat{A}(R^M)\wedge 
    f^0df^1\wedge\ldots\wedge df^{2q}, \quad f^j\in C^\infty(M).
    \end{equation}
   This allows us to recover the index theorem of Atiyah-Singer~\cite{AS:IEO1, AS:IEO3} for Dirac operators. 
\end{example}

\begin{remark}
 The definitions of the cocycles $\tau_{2q}^{D,\sigma}$ and 
 $\overline{\tau}_{2q}^{D,\sigma}$ involve the usual (super)trace on trace-class 
 operators, but this is not a local functional since it does not vanish on finite rank operators. As 
 a result this cocycle is difficult to compute in practice (see, e.g.,~\cite{BF:APDOIT}). Therefore, it stands for 
 reason to seek for a representatative of the Connes-Chern character which is easier to compute. For ordinary spectral 
 triples, and under further assumptions, such a representative is provided by the CM cocycle of 
 Connes-Moscovici~\cite{CM:GAFA95}. This cocycle is an even periodic cycle whose components are given by formulas which 
 are \emph{local} in the sense of noncommuative geometry. More precisely, they 
 involve a version for spectral triples of the noncommutative residue trace of Guillemin~\cite{Gu:NPWF} and 
 Wodzicki~\cite{Wo:LISA}. This provides us with the local index formula in noncommutative geometry. 
 In Example~\ref{ex:local.index.formula.DST} the cocycle $\varphi=(\varphi_{2q})$ given 
 by~(\ref{eq:CMcocycle.Dirac}) is precisely the CM cocycle of the Dirac spectral triple $(C^{\infty}(M), L^2(M, \sS), 
 \sD_g)$ (see~\cite[Remark~II.1]{CM:GAFA95} and~\cite{Po:CMP1}). 
\end{remark}

\begin{remark}\label{rem:Moscivici.ansatz}
 In the case of twisted spectral triples, Moscovici~\cite{Mo:LIFTST} attempted to extend the local index formula to the setting of twisted spectral triples. 
 He devised an ansatz for such a local index formula  and verified it in the special case of a ordinary spectral triples twisted by scaling 
 automorphisms (see~\cite{Mo:LIFTST} for the precise definition). Whether Moscovici's ansatz holds for a larger class 
of twisted spectral triples still remains an open question to date.  For instance, it is not known if Moscovici ansatz 
holds for conformal deformations of ordinary spectral triples satisfying the local index formula in noncommutative 
geometry. 
\end{remark}

\section{Twisted Spectral Triples over Locally Convex Algebras}\label{sec:smooth-twisted-ST} 
In this section, we shall explain how to refine the 
construction of the Connes-Chern character for twisted spectral triples over locally convex algebras. 

In what follows by \emph{locally convex $*$-algebra} we shall mean a $*$-algebra $\cA$ equipped with a locally convex 
space topology with respect to which  its product is a jointly continuous bilinear map from $\cA^{2}\rightarrow \cA$ 
and its involution is a continuous anti-linear map.  

\begin{definition}
\label{def:TST.smooth}
    A twisted spectral triple $(\cA,\cH, D)_{\sigma}$ over a locally convex  $*$-algebra $\cA$ is called smooth when the 
    following conditions hold:
    \begin{enumerate}
        \item  The representation of $\cA$ in $\cL(\cH)$ is continuous. 
    
        \item  The map $a\rightarrow [D,a]_{\sigma}$ is continuous from $\cA$ to $\cL(\cH)$. 

        \item  The automorphism $\sigma:\cA\rightarrow \cA$ is a homeomorphism. 
    \end{enumerate}
\end{definition}

\begin{example}
    Let $(M^{n},g)$ be a closed spin Riemannian manifold of even dimension. Then the associated Dirac spectral triple 
    $(C^{\infty}(M), L^{2}_{g}(M,\sS),\sD_{g})$ is smooth.
\end{example}

\begin{example}
    Any conformal deformation of a smooth ordinary spectral triple yields a smooth 
    twisted spectral triple. 
\end{example}

\begin{remark}
 As we shall see in Section~\ref{sec:Conformal-CC-character} the conformal Dirac spectral triple of~\cite{CM:TGNTQF} is smooth. 
\end{remark}

Throughout the rest of this section we let $(\cA,\cH,D)_{\sigma}$ be a \emph{smooth} twisted spectral triple which is 
$p$-summable for some $p\geq 1$. We shall now show that the Connes-Chern character of $(\cA,\cH,D)_{\sigma}$, which is 
originally defined as a class in $\op{HP}^{0}(\cA)$, actually descends to a class in $\op{\mathbf{HP}}^{0}(\cA)$. 

\begin{lemma}
\label{lem:tauD.indep.q}
    Let $q$ be any integer~$\geq \frac{1}{2}(p-1)$.
    \begin{enumerate}
        \item  If $D$ is invertible, then the cyclic cocycle  $\tau_{2q}^{D, \sigma}$ is continuous and its class in 
        $\op{\mathbf{HP}}^{0}(\cA)$ is independent of $q$.
    
        \item   The cyclic cocycle  $\overline{\tau}_{2q}^{D, \sigma}$ is continuous and its class in 
        $\op{\mathbf{HP}}^{0}(\cA)$ is independent of $q$.
    \end{enumerate}
\end{lemma}
\begin{proof}
   Assume that $D$ is invertible and let $q$ be an integer~$\geq \frac{1}{2}(p-1)$. By assumption the map $a\rightarrow 
   [D,a]_{\sigma}$ is continuous from $\cA$ to $\cL(\cH)$. Combining this with H\"older's inequality for Schatten 
   ideals we deduce that the map $(a^{0},\ldots,a^{2q})\rightarrow \gamma D^{-1}[D,a^{0}]\cdots D^{-1}[D,a^{2q}]$ is 
   continuous from $\cA^{2q+1}$ to $\cL^{1}(\cH)$. As $\tau_{2q}^{D,\sigma}$ is (up to a multiple constant) the 
   composition of this map with the operator trace, we deduce that $\tau_{2q}^{D,\sigma}$ is a continuous cochain. 
   
  Moreover, by Lemma~7.4 and Lemma~7.5 of~\cite{PW:KJM16} we have
   \begin{equation}
   \label{eq:tauD.CCcocycle.image.b.B}
       \tau_{2q+2}^{D,\sigma}-\tau_{2q}^{D,\sigma}=(b+B)(\varphi_{2q+1}-\psi_{2q+1}), 
   \end{equation}where, up to normalization constants, the cochains $\varphi_{2q+1}$ and $\psi_{2q+1}$ are given by
    \begin{gather*}
         \varphi_{2q+1} (a^0, \ldots, a^{2q+1})=
        \Str(a^0 D^{-1}[D, a^1]_{\sigma}\cdots D^{-1}[D, a^{2q+1}]_{\sigma}),\\ 
 \psi_{2q+1}(a^0, \ldots, a^{2q+1})
 =\Str(\sigma(a^0)[D, a^1]_{\sigma}D^{-1}\cdots[D, a^{2q+1}]_{\sigma}D^{-1}), \qquad a^j\in\cA.
\end{gather*}
Note that $ \varphi_{2q+1}$ and $ \psi_{2q+1}$ are normalized cochains. Moreover, in the same way as with the cocycle 
$\tau_{2q}^{D,\sigma}$ we can show that these cochains are continuous.  Therefore~(\ref{eq:tauD.CCcocycle.image.b.B}) 
shows that the images in $\bC^{2q+2}_{\natural}(\cA)$ of $\tau_{2q+2}^{D,\sigma}$ and $\tau_{2q}^{D,\sigma}$ are 
cohomologous, and, hence, they define the same class in $\bHP^{0}(\cA)$.  It then follows that the class of 
$\tau_{2q}^{D,\sigma}$ in $\bHP^{0}(\cA)$ is independent of $q$. This also implies that 
\begin{equation}
 \label{eq:tauD.CCcocycle.S}
    S_{\lambda}\tau_{2q}^{D,\sigma}=\tau_{2q+2}^{D, \sigma} \quad \text{in $\mathbf{H}_{\lambda}^{2q+2}(\cA)$}.
\end{equation}

   When $D$ is not invertible, we observe that the unitalization $\tilde{\cA}=\cA\oplus \C$ is a locally convex $*$-algebra with respect to the direct sum 
   topology. It then can be checked that the invertible double $(\tilde{\cA},\tilde{\cH},\tilde{D})_{\tilde{\sigma}}$ 
   considered in Section~\ref{sec:CC-character}
   is a smooth twisted spectral triple. Therefore, the first part of the proof shows that the cocycle 
   $\tau_{2q}^{\tilde{D},\sigma}$ is continuous.  
   As the inclusion of $\cA$ into $\tilde{\cA}$ is continuous and $\overline{\tau}_{2q}^{D,\sigma}$ is the restriction 
   to $\cA$ of $\tau_{2q}^{\tilde{D},\sigma}$,  we then see that  $\overline{\tau}_{2q}^{D,\sigma}$ is a continuous 
   cocycle on $\cA$. Moreover, using~(\ref{eq:CC.operatorSj}) it can be checked that $S_{\lambda}\overline{\tau}_{2q}^{D,\sigma}$ is the restriction to $\cA$ of 
   $S_{\lambda}\tau_{2q}^{\tilde{D},\sigma}$. Therefore using~(\ref{eq:tauD.CCcocycle.S}) we deduce that the cocycles
   $S_{\lambda}\overline{\tau}_{2q}^{D,\sigma}$ and $\overline{\tau}_{2q+2}^{D,\sigma}$ are cohomologous in 
   $\mathbf{H}^{2q+2}_{\lambda}(\cA)$. It then follows that $\overline{\tau}_{2q}^{D,\sigma}$ and 
   $\overline{\tau}_{2q+2}^{D,\sigma}$ define the same class in $\bHP^{0}(\cA)$, and so the class of 
   $\overline{\tau}_{2q}^{D,\sigma}$ in $\op{\mathbf{HP}}^{0}(\cA)$ is independent of $q$. The proof is complete. 
 \end{proof}

In addition, we will also need the following version  for smooth twisted spectral triples of~\cite[Propositon~C.1]{PW:KJM16} on the homotopy invariance of 
the Connes-Chern character of a twisted spectral triple. 

\begin{lemma}
\label{lem:path.saDirac.cyclic.cocycle}
    Assume $D$ is invertible and consider an operator homotopy of the form, 
    \begin{equation*}
        D_{t}=D+V_{t}, \qquad 0\leq t \leq 1,
    \end{equation*}where $(V_{t})_{0\leq t \leq 1}$ is a $C^{1}$-family of selfadjoint operators in  $\cL(\cH)$ 
    anticommuting with the grading $\gamma$ such that $D_{t}$ is 
    invertible for all $t\in [0,1]$ and $(D_{t}^{-1})_{0\leq t \leq 1}$ is a bounded family in $\cL^{p}(\cH)$. Then
    \begin{enumerate}
        \item  $(\cA,\cH,D_{t})_{\sigma}$ is a smooth $p$-summable twisted spectral triple for all $t\in [0,1]$. 
    
        \item  For any $q\geq \frac{1}{2}(p+1)$, the cocycles $\tau_{2q}^{D_{0},\sigma}$ and $\tau_{2q}^{D_{1},\sigma}$ are cohomologous 
        in $\mathbf{H}^{2q}_{\lambda}(\cA)$. 
    \end{enumerate}
\end{lemma}
\begin{proof}
We know from~\cite[Propositon~C.1]{PW:KJM16} that $(\cA,\cH,D_{t})_{\sigma}$ is a $p$-summable twisted spectral triple for 
all $t\in [0,1]$ and, for any $q\geq \frac{1}{2}(p+1)$, the cocycles $\tau_{2q}^{D_{0},\sigma}$ and $\tau_{2q}^{D_{1},\sigma}$ are cohomologous 
        in $H_{\lambda}^{2q}(\cA)$. Therefore, we only need to show that the twisted spectral triples $(\cA,\cH,D_{t})_{\sigma}$, $t\in [0,1]$, are 
        smooth and the cocycles $\tau_{2q}^{D_{0},\sigma}$ and $\tau_{2q}^{D_{1},\sigma}$ differ by the coboundary of a \emph{continuous} cyclic cochain.  
        
        By assumption $(\cA,\cH,D)_\sigma$ is a smooth twisted spectral triple. In particular, the representation of 
        $\cA$ in $\cL(\cH)$ is continuous and the automorphism $\sigma$ is a homeomorphism. Moreover, for all $t\in [0,1]$ and $a\in \cA$, 
        \begin{equation*}
            [D_{t},a]_{\sigma}=[D,a]_{\sigma}+V_{t}a-\sigma(a)V_{t}. 
        \end{equation*}
We then see that the map $(t,a)\rightarrow [D_{t},a]_{\sigma}$  is continuous from $[0,1]\times \cA$ to 
$\cL(\cH)$. It then follows that $(\cA,\cH,D_{t})_{\sigma}$ is a smooth twisted spectral triple for all $t\in[0,1]$. 
    
Let $q$ be an integer~$\geq \frac{1}{2}(p-1)$. In~\cite{PW:KJM16} the explicit homotopy between  $\tau_{2q}^{D_{0},\sigma}$ 
and $\tau_{2q}^{D_{1},\sigma}$ is realized as follows. For $t\in[0,1]$ and $a\in \cA$ set
\begin{equation*}
    \delta_{t}(a)=D_{t}^{-1}[\dot{V}_{t}D_{t}^{-1},\sigma(a)]D_{t}.
\end{equation*}In addition, for $j=0,\ldots,2q+1$ we set 
\begin{equation*}
    \alpha_{j}^{t}(a)=a \ \text{if $j$ is even} \quad  \text{and } \quad 
\alpha_{j}^{t}(a)=D_{t}^{-1}\sigma(a)D_{t}\ \text{if $j$ is odd}. 
\end{equation*}
We note that 
\begin{equation*}
    \delta_{t}(a)=D_{t}^{-1}(\dot{V}_{t} \alpha_{1}^{t}(a)-\sigma(a)\dot{V}_{t})\qquad \text{and} \qquad  \alpha_{1}^{t}(a)=a-D_{t}^{-1}[D_{t},a]_{\sigma}
\end{equation*}
In particular we see that the maps $(t,a)\rightarrow \alpha_{j}^{t}(a)$ are continuous from $[0,1]\times \cA$ to 
$\cL(\cH)$. Moreover, it is shown in~\cite[Appendix~C]{PW:KJM16} that $(D_{t}^{-1})_{0\leq t \leq t}$ is a $C^{1}$-family in 
$\cL^{p}(\cH)$. Therefore, we also see that the map $(t,a)\rightarrow \delta_{t}(a)$ is continuous from $[0,1]\times \cA$ to 
$\cL^{p}(\cH)$.

Bearing this mind it is shown in~\cite{PW:KJM16} that
\begin{equation}
\label{eq:tauD_1-D_0.B}
    \tau_{2q}^{D_{1},\sigma}-\tau_{2q}^{D_{0},\sigma}=B\eta,
\end{equation}where $\eta$ is the Hochschild $(2q+1)$-cocycle given by
\begin{equation*}
\eta(a^{0},\ldots,a^{2q+1})= c_{q}\sum_{j=0}^{2q+1}\int_{0}^{1}    \Str\left( \alpha_{j}(a^{0})
   D_{t}^{-1}[D_{t},a^{1}]_{\sigma}\cdots \delta_{t}(a^{j}) \cdots 
   D_{t}^{-1}[D_{t},a^{2q+1}]_{\sigma}\right)dt,
\end{equation*}where $c_{q}$ is some normalization constant. It follows from all
the previous observations and the fact that $q\geq \frac{1}{2}(p-1)$ that all the maps
\begin{equation*}
    (t,a^{0},\ldots,a^{2q+1})\rightarrow \alpha_{j}(a^{0})
   D_{t}^{-1}[D_{t},a^{1}]_{\sigma}\cdots \delta_{t}(a^{j}) \cdots 
   D_{t}^{-1}[D_{t},a^{2q+1}]_{\sigma}
\end{equation*}are continuous from $[0,1]\times \cA^{2q+2}$ to $\cL^{1}(\cH)$. It follows from the cochain $\eta$ is 
cochain. Therefore, using~(\ref{eq:tauD.CCcocycle.S}) and the fact that $\eta$ is a Hochschild cocycle we see that
\begin{equation*}
    \tau_{2q}^{D_{1},\sigma}-\tau_{2q}^{D_{0},\sigma}=(b+B)\eta \qquad \text{in $\bC_{\natural}^{2q+2}(\cA)$}.
\end{equation*}
This implies that $S(\tau_{2q}^{D_{1},\sigma}-\tau_{2q}^{D_{0},\sigma})$ is a coboundary in 
$\mathbf{H}_{\lambda}^{2q+2}(\cA)$. Combining this with~(\ref{eq:tauD.CCcocycle.S}) we then deduce that $\tau_{2q+2}^{D_{0},\sigma}$ and 
$\tau_{2q+2}^{D_{0},\sigma}$ are cohomologous in $\mathbf{H}_{\lambda}^{2q+2}(\cA)$. This proves the 2nd part of the lemma and completes the proof. 
\end{proof}

The homotopy lemma is the main ingredient in the proof of the following result. 

\begin{lemma}
\label{lem:tauD.indep.q.LCA}
Assume $D$ is invertible and let $q$ be an integer~$\geq \frac{1}{2}(p+1)$. Then the cyclic cocycles  
$\tau_{2q}^{D,\sigma}$ and $\overline{\tau}_{2q}^{D,\sigma}$ are cohomologous in $\mathbf{H}_{\lambda}^{2q}(\cA)$, and hence 
define the the same class in  $ \bHP^{0}(\cA)$.     
\end{lemma}
\begin{proof}
 We know from~\cite[\S 7]{PW:KJM16} that, for $q \geq \frac{1}{2}(p+1)$, the cocycles $\tau_{2q}^{D,\sigma}$ and 
 $\overline{\tau}_{2q}^{D,\sigma}$ define the same class in $H_{\lambda}^{2q}(\cA)$. The argument relies on the 
 homotopy of unbounded operators on $\tilde{\cH}$ given by
 \begin{equation*}
     \tilde{D}_{t}=\tilde{D}_{0}+tJ, \qquad 0\leq t \leq 1,
 \end{equation*}where we have set
 \begin{equation*}
  \tilde{D}_{0}=
  \begin{pmatrix}
      D & 0 \\
      0 & -D
  \end{pmatrix} \qquad \text{and} \qquad J=
  \begin{pmatrix}
      0 & 1 \\
      1 & 0
  \end{pmatrix}.   
 \end{equation*}
 Note that $\tilde{D}_{1}=\tilde{D}$. It is shown in~\cite[\S 7]{PW:KJM16} that this homotopy satisfies the assumptions of 
 Lemma~\ref{lem:path.saDirac.cyclic.cocycle}. Therefore, $(\tilde{\cA},\tilde{H},\tilde{D}_{t})_{\tilde{\sigma}}$ is a smooth spectral triple for all 
 $t\in [0,1]$ and, for $q \geq \frac{1}{2}(p+1)$, the cocycles $\tau^{\tilde{D}_{0},\tilde{\sigma}}_{2q}$ and 
 $\tau^{\tilde{D}_{1},\tilde{\sigma}}_{2q}$ are cohomologous in $\mathbf{H}_{\lambda}^{2q}(\tilde{\cA})$.  
 Incidentally, the  restrictions to $\cA$ of $\tau^{\tilde{D}_{0},\tilde{\sigma}}_{2q}$ and 
 $\tau^{\tilde{D}_{1},\tilde{\sigma}}_{2q}$ are cohomologous cocycles in $\mathbf{H}_{\lambda}^{2q}(\cA)$. As 
 $\tilde{D}_{1}=\tilde{D}$ the restriction of $\tau^{\tilde{D}_{1},\tilde{\sigma}}_{2q}$ is 
 $\overline{\tau}_{2q}^{D,\sigma}$. Moreover, it is shown in~\cite[\S 7]{PW:KJM16} that $\tau_{2q}^{D,\sigma}$  is the 
 restriction of $\tau^{\tilde{D}_{0},\tilde{\sigma}}_{2q}$. Therefore, we see that $\tau_{2q}^{D,\sigma}$ and 
 $\overline{\tau}_{2q}^{D,\sigma}$ are cohomologous in $\mathbf{H}_{\lambda}^{2q}(\cA)$, and hence 
define the the same class in  $\bHP^{0}(\cA)$.   The proof is complete. 
\end{proof}

Granted Lemma~\ref{lem:tauD.indep.q} and Lemma~\ref{lem:tauD.indep.q.LCA} the Connes-Chern character descends to a class in $ \op{\mathbf{HP}}^{0}(\cA)$ as follows. 

\begin{definition}
The Connes-Chern character $\bCh(D)_{\sigma}\in \op{\mathbf{HP}}^{0}(\cA)$ defined as follows:
\begin{itemize}
    \item[-] When $D$ is invertible this the common class in  $\op{\mathbf{HP}}^{0}(\cA)$ of any of the cyclic cocycles  
    $\tau_{2q}^{D, \sigma}$ or $\overline\tau_{2q}^{D, \sigma}$, $q\geq \frac{1}{2}(p-1)$. 

    \item[-] When $D$ is not invertible this the common class in  $\op{\mathbf{HP}}^{0}(\cA)$ of the cyclic cocycles  
   $\overline\tau_{2q}^{D, \sigma}$, $q\geq \frac{1}{2}(p-1)$. 
\end{itemize}
\end{definition}

\begin{proposition}\label{prop:CCC-index-formula2}
    Let $(\cA,\cH,D)_{\sigma}$ be a $p$-summable smooth twisted spectral triple. Then
\begin{enumerate}
 \item  The class $\bCh(D)_{\sigma}$  agrees with the Connes-Chern character 
$\Ch(D)_{\sigma}$  under the morphism $\op{\mathbf{HP}}^{0}(\cA)\rightarrow \op{HP}^{0}(\cA)$ 
 induced by the inclusion of $\mathbf{C}^{\bt}_{\sharp}(\cA)$ into $C^{\bt}_{\sharp}(\cA)$.

\item For any Hermitian finitely generated projective module $\cE$ over $\cA$ 
    and $\sigma$-connection $\nabla^{\cE}$ on $\cE$, we have
\begin{equation*}
\ind D_{\nabla^{\cE}}=\acou{\bCh(D)_{\sigma}}{\bCh(\cE)},
\end{equation*}where $\bCh(\cE)$ is the Chern character of $\cE$ seen as a class in 
$\op{\mathbf{HP}}_{0}(\cA)$. 
\end{enumerate}
\end{proposition}
\begin{proof}
The first part is immediate since by their very definitions, $\bCh(D)_{\sigma}$ and $\Ch(D)_{\sigma}$ are represented by the same cyclic cocycles. 
As for the 2nd part, consider an idempotent $e$ in some $M_{N}(\cA)$, $N\geq 1$ such that $\cE\simeq e\cA^{N}$. Then, for any $q\geq \frac{1}{2}(p-1)$,
\begin{equation*}
    \acou{\bCh(D)_{\sigma}}{\bCh(\cE)}= \acou{\overline{\tau}_{2q}^{D,\sigma}}{\Ch(e)}=\acou{\Ch(D)_{\sigma}}{\Ch(\cE)}. 
\end{equation*}Combining this with Proposition~\ref{thm:CCC-index-formula} gives the result. 
\end{proof}

\begin{remark}
 In what follows by the Connes-Chern character of a smooth ($p$-summable) twisted spectral triple $(\cA,\cH,D)_{\sigma}$ we shall 
 mean the cohomology class $\bCh(D)_{\sigma}\in  \op{\mathbf{HP}}^{0}(\cA)$.    
\end{remark}

\section{Invariance of the Connes-Chern Character}\label{sec:Invariance-CC-character}
In preparation for the next section, we prove in this section the invariance of the Connes-Chern character under 
equivalences and conformal deformations of twisted spectral triples. 

\subsection{Equivalence of twisted spectral triples} The equivalence of two twisted spectral triples over the same 
algebra is defined as follows. 

\begin{definition}
    Let $(\cA,\cH_{1},D_{1})_{\sigma_{1}}$ and $(\cA,\cH_{2},D_{2})_{\sigma_{2}}$ be twisted spectral triples over the 
    same algebra. For $i=1,2$ let us denote by $\pi_{i}$ the representation of $\cA$ into $\cH_{i}$. Then we say that $(\cA,\cH_{1},D_{1})_{\sigma_{1}}$ and 
    $(\cA,\cH_{2},D_{2})_{\sigma_{2}}$ are 
    equivalent when there is a unitary operator $U:\cH_{1}\rightarrow \cH_{2}$ such that 
    \begin{gather}
      D_{1}=  U^{*}D_{2}U, \label{eq:UniEquivOp}\\
      \pi_{1}(a)=U^{*}\pi_{2}(a)U \quad \text{and}\quad  
      \pi_{1}(\sigma_{1}(a))=U^{*}\pi_{2}(\sigma_{2}(a))U \quad \text{for all $a\in \cA$}.
      \label{eq:UniEquiv}
    \end{gather}
\end{definition}

\begin{remark}
    We may also define the equivalence of a pair of spectral triples $(\cA_{1},\cH_{1},D_{1})_{\sigma_{1}}$ and 
    $(\cA_{2},\cH_{2},D_{2})_{\sigma_{2}}$ with $\cA_{1}\neq \cA_{2}$ by requiring the existence of a $*$-algebra 
    isomorphism $\psi:\cA_{1}\rightarrow \cA_{2}$ and replacing the condition~(\ref{eq:UniEquiv}) by
    \begin{equation*}
         \pi_{1}(a)=U^{*}\pi_{2}\left(\psi(a)\right)U \quad \text{and}\quad  
      \pi_{1}(\sigma_{1}(a))=U^{*}\pi_{2}\left(\sigma_{2}\left(\psi(a)\right)\right)U \quad \text{for all $a\in \cA_{1}$}.
    \end{equation*}
\end{remark}

\begin{proposition}
\label{prop:equiv.TST.Same.CC}
    Let $(\cA,\cH_{1},D_{1})_{\sigma_{1}}$ and $(\cA,\cH_{2},D_{2})_{\sigma_{2}}$ be equivalent twisted spectral 
    triples that are $p$-summable for some $p\geq 1$. In addition, let $q$ be an integer~$\geq \frac{1}{2}(p-1)$. Then 
    \begin{enumerate}
        \item  The cyclic cocycles $\overline{\tau}_{2q}^{D_{1},\sigma_{1}}$ and 
        $\overline{\tau}_{2q}^{D_{2},\sigma_{2}}$ agree.

        \item The same result holds for the cocycles $\tau_{2q}^{D_{1},\sigma_{1}}$ and $\tau_{2q}^{D_{2},\sigma_{2}}$ 
        when $D_{1}$ and $D_{2}$ are invertible. 
    
        \item  The twisted spectral triples $(\cA,\cH_{1},D_{1})_{\sigma_{1}}$ and $(\cA,\cH_{2},D_{2})_{\sigma_{2}}$ 
        have the same Connes-Chern character. 
    \end{enumerate}
\end{proposition}
\begin{proof}
      The last part is an immediate consequence of the first two parts, so we only need to prove these two parts. In addition, 
      we note that 
      the invertible doubles of $(\cA,\cH_{1},D_{1})_{\sigma_{1}}$ and $(\cA,\cH_{2},D_{2})_{\sigma_{2}}$ are 
      equivalent, where the equivalence is implemented by the unitary operator $U\oplus U$ acting on 
      $\tilde{\cH}=\cH\oplus \cH$. Therefore, it is enough to assume that $D_{1}$ and $D_{2}$ are invertible and prove 
      the 2nd part.
      
      Under the aforementioned assumption and using~(\ref{eq:UniEquivOp})--(\ref{eq:UniEquiv}) 
      we see that,  for all $a \in \cA$, 
      \begin{equation*}
          D_{1}^{-1}[D_{1},a]_{\sigma_{1}}=U^{*}D_{2}^{-1}[D_{2},a]_{\sigma_{2}}U.
      \end{equation*}
       Therefore, for all $a^{j}\in \cA$, we have 
      \begin{align*}
         \tau_{2q}^{D_{1},\sigma_{1}}(a^{0},\ldots,a^{2q})& =  \frac{1}{2}(-1)^q \frac{q!}{(2q)!} 
     \Str \left\{ U^{*}D_{2}^{-1}[D_{2}, a^0]_{\sigma_{2}}U\cdots U^{*}D_{2}^{-1}[D_{2}, a^{2q}]_{\sigma_{2}}U\right\}\\ 
     & =  \tau_{2q}^{D_{2},\sigma_{2}}(a^{0},\ldots,a^{2q}).
      \end{align*}This proves the result.
  \end{proof}   

\subsection{Conformal deformations of twisted spectral triples} For future purpose it will be useful to 
extend to the setting of twisted spectral triples the conformal deformations of ordinary spectral triples. Let $(\cA,\cH,D)_{\sigma}$ be a twisted spectral triple and 
$k$ a positive element of $\cA$. We let $\hat{\sigma}:\cA\rightarrow \cA$ be the automorphism of $\cA$ given by
\begin{equation*}
    \hat{\sigma}(a)=k\sigma(kak^{-1})k^{-1} \qquad \forall a \in \cA.
\end{equation*}

\begin{proposition} $\left( \cA,\cH, kDk\right)_{\hat{\sigma}}$ is a twisted spectral triple. 
\end{proposition}
\begin{proof}
    We only need to check the boundedness of the twisted commutators $[kDk,a]_{\hat{\sigma}}$, $a\in \cA$. To see this we 
    note that 
    \begin{equation}
    \label{eq:compare.twisted.commutator}
     [kDk,a]_{\hat{\sigma}} =  kDka-\hat{\sigma}(a)kDk = k\left( D(kak^{-1})-(k^{-1}\hat{\sigma}(a)k)D\right)k  = k[D,a]_{\sigma}k.
    \end{equation}
As the twisted commutator $[D,a]_{\sigma}$ is bounded, it follows that $[kDk,a]_{\hat{\sigma}}$ is bounded as well. The proof is 
complete.  
\end{proof}

The following shows that the Connes-Chern character is invariant under conformal deformations. 

\begin{proposition}
\label{prop:conf.Dirac.same.Ch}
    Assume that $(\cA,\cH,D)_{\sigma}$ is $p$-summable for some $p\geq 1$. Then, for any positive element $k\in \cA$, 
    we have
    \begin{equation*}
        \Ch(kDk)_{\hat{\sigma}}= \Ch(D)_{\sigma}\in \op{HP}^{0}(\cA).
    \end{equation*}
\end{proposition}
\begin{remark}
The above result is proved in~\cite{CM:TGNTQF} in the special case $\sigma=\op{id}$ and $D$ is invertible. 
\end{remark}
\begin{proof}[Proof of Proposition~\ref{prop:conf.Dirac.same.Ch}]
 Set $D_{k}=kD{k}$. We shall first prove the result when $D$ is invertible, as the proof is simpler in that 
 case. Given an integer~$q>\frac{1}{2}(p-1)$ 
 let $a^{j}\in \cA$, $j=0,\ldots,2q$. Using~(\ref{eq:tau_2k}) and~(\ref{eq:compare.twisted.commutator}) we get
 \begin{align}
     \tau_{2q}^{D_{k},\hat{\sigma}}(a^{0},\ldots,a^{2q})= & c_{q}
     \Str \left\{ D_{k}^{-1}[D_{k}, a^0]_{\hat{\sigma}}\cdots D_{k}^{-1}[D_{k}, a^{2q}]_{\hat{\sigma}}\right\}\\
     = &  c_{q} \Str \left\{ (k^{-1}D^{-1}[D, ka^0k^{-1}]_{\sigma}k)\cdots (k^{-1}D^{-1}[D, ka^{2q}k^{-1}]_{\sigma}k)\right\}\\
     = &  c_{q} \Str \left\{ D^{-1}[D, ka^0k^{-1}]_{\sigma}\cdots D^{-1}[D, ka^{2q}k^{-1}]_{\sigma}\right\}\\
     = & \tau_{2q}^{D,\sigma}\left(ka^{0}k^{-1},\cdots, ka^{2q}k^{-1}\right). 
      \label{eq:tau_2q.D_k.D}
 \end{align}
 As cyclic cohomology is invariant under the action of inner automorphisms (see~\cite[Prop.\ III.1.8]{Co:NCG} 
 and~\cite[Prop.\ 4.1.3]{Lo:CH}), 
 we deduce that the cyclic cocycles $\tau_{2q}^{D,\sigma}$ and $\tau_{2q}^{D_{k},\hat{\sigma}}$ are cohomologous in 
 $H_{\lambda}^{2q}(\cA)$. Therefore, they 
 define the same class in $\op{HP}^{0}(\cA)$, and so the twisted spectral 
triples $(\cA,\cH,D)_{\sigma}$ and $(\cA,\cH,D_{k})_{\hat{\sigma}}$ have the same Connes-Chern character.
  
Let us now prove the result when $D$ is not invertible. To this end consider the respective 
unital invertible doubles $(\tilde{\cA},\tilde{\cH},\tilde{D})_{\sigma}$ and 
$(\tilde{\cA},\tilde{\cH},\tilde{D}_{k})_{\hat{\sigma}}$ of $(\cA,\cH,D)_{\sigma}$ and its conformal deformation 
$(\cA,\cH,D_{k})_{\hat{\sigma}}$, where by a slight abuse of notation we have denoted by $\sigma$ and $\hat{\sigma}$ 
the extensions to $\tilde{\cA}$ of the automorphisms $\sigma$ and $\hat{\sigma}$. As it turns out, 
$(\tilde{\cA},\tilde{\cH},\tilde{D}_{k})_{\hat{\sigma}}$ is not a conformal deformation of  
$(\tilde{\cA},\tilde{\cH},\tilde{D})_{\sigma}$ so that the proof in the invertible case does not extend to the 
invertible doubles. Nevertheless, as we shall see, $(\tilde{\cA},\tilde{\cH},\tilde{D}_{k})_{\hat{\sigma}}$ is a pseudo-inner twisting in the sense of~\cite{PW:AIM15} of 
a twisted spectral triple which is homotopy equivalent to $(\tilde{\cA},\tilde{\cH},\tilde{D})_{\sigma}$.

To wit consider the selfadjoint unbounded operator on $\tilde{\cH}=\cH\oplus \cH$ given by 
\begin{equation*}
    \tilde{D}_{1}=
    \begin{pmatrix}
        D & k^{-2} \\
        k^{-2} & -D
    \end{pmatrix}, \qquad \dom(\tilde{D}_{1})=\dom(D)\oplus \dom(D).
\end{equation*}
As $\tilde{D}_{1}$ agrees with $\tilde{D}$ up to a bounded operator, we see that 
$(\tilde{\cA},\tilde{\cH},\tilde{D}_{1})_{\sigma}$ is a $p$-summable twisted spectral triple. Note also that 
\begin{equation}
\label{eq:D_k&D_1}
    \tilde{D}_{k}= 
    \begin{pmatrix}
        kDk & 1 \\
        1 & -kDk
    \end{pmatrix}= \omega \tilde{D}_{1}\omega, \qquad \text{where } \omega:=Â«
    \begin{pmatrix}
        k & 0 \\
        0 & k
    \end{pmatrix}.
\end{equation}In particular, we see that $\tilde{D}_{1}$ is invertible. An explicit homotopy between 
$(\tilde{\cA},\tilde{\cH},\tilde{D})_{\sigma}$ and $(\tilde{\cA},\tilde{\cH},\tilde{D}_{1})_{\sigma}$ is given by the 
family of twisted spectral triples $(\tilde{\cA},\tilde{\cH},\tilde{D}_{t})_{\sigma}$, $t\in [0,1]$, with 
\begin{equation}
\label{eq:homotopy.D.D_1}
     \tilde{D}_{t}=
    \begin{pmatrix}
        D & k^{-2t} \\
        k^{-2t} & -D
    \end{pmatrix} =\tilde{D}+V_{t},  \qquad \text{where } V_{t}=
    \begin{pmatrix}
        0 & k^{-2t}-1 \\
        k^{-2t}-1 & 0
    \end{pmatrix}.
\end{equation}
We observe that $(V_{t})_{t\in [0,1]}$ is a $C^{1}$-family in $\cL(\tilde{\cH})$ and, in the same way as for 
$\tilde{D}_{1}$, it can be shown that the operator $\tilde{D}_{t}$ is invertible for all $t\in [0,1]$. Therefore, we may 
use the homotopy invariance of the Connes-Chern character in the form of~\cite[Appendix C]{PW:KJM16} to deduce that, for all 
$q\geq \frac{1}{2}(p+1)$, the cyclic cocycles $\tau^{\tilde{D},\sigma}_{2q}$ and $\tau^{\tilde{D}_1,\sigma}_{2q}$ are 
cohomologous in $H_{\lambda}^{2q}(\tilde{\cA})$. Incidentally, if denote by $\overline{\tau}^{\tilde{D},\sigma}_{2q}$ 
and $\overline{\tau}^{\tilde{D}_1,\sigma}_{2q}$ their respective restrictions to $\cA$, then we see that $\overline{\tau}^{\tilde{D},\sigma}_{2q}$ 
and $\overline{\tau}^{\tilde{D}_1,\sigma}_{2q}$ are cohomologous cyclic cocycles in $H_{\lambda}^{2q}(\cA)$. 

Bearing this in mind, let $a\in \cA$. Using~(\ref{eq:D_k&D_1}) and noting that $\omega \tilde{\pi}(a)\omega^{-1}=\tilde{\pi}(kak^{-1})$, we see that
$ \tilde{D}_{k}\tilde{\pi}(a)=   \omega \tilde{D}_{1} (\omega \tilde{\pi}(a)\omega^{-1})\omega=\omega 
 \tilde{D}_{1}\tilde{\pi}(kak^{-1})\omega$. 
Likewise, 
\begin{equation*}
    \tilde{\pi}(\hat{\sigma}(a))\tilde{D}_{k}=\omega(\omega^{-1}\tilde{\pi}(\hat{\sigma}(a))\omega)\tilde{D}_{1}\omega= 
    \omega\tilde{\pi}(k^{-1}\hat{\sigma}(a)k)\tilde{D}_{1}\omega= 
    \omega \tilde{\pi}(\sigma(kak^{-1})) \tilde{D}_{1}\omega.
\end{equation*}Thus,
\begin{equation}
\label{eq:arguingas}
   \tilde{D}_{k}^{-1}[\tilde{D}_{k},\tilde{\pi}(a)]_{\hat{\sigma}}= 
   \left(\omega^{-1}\tilde{D}_{1}^{-1}\omega^{-1}\right) \left(\omega [\tilde{D}_{1},\tilde{\pi}(kak^{-1})]_{\sigma}\omega\right)
   =\omega^{-1}\tilde{D}_{1}^{-1} [\tilde{D}_{1},\tilde{\pi}(kak^{-1})]_{\sigma}\omega. 
\end{equation}
Given an integer~$q>\frac{1}{2}(p-1)$, let $a^{j}\in \cA$, $j=0,\ldots,2q$. Using~(\ref{eq:tau_2k}) and~(\ref{eq:arguingas}) and arguing as 
in~(\ref{eq:tau_2q.D_k.D}) we obtain
\begin{align}
    \overline{\tau}_{2q}^{\tilde{D}_{k},\hat{\sigma}}(a^{0},\ldots,a^{2q}) 
    = & c_{q}\Str \left\{ 
    \tilde{D}_{k}[\tilde{D}_{k},\tilde{\pi}(a^{0})]_{\hat{\sigma}}\cdots   
    \tilde{D}_{k}[\tilde{D}_{k},\tilde{\pi}(a^{2q})]_{\hat{\sigma}}\right\} \\
    =  & c_{q}\Str \left\{ 
   (\omega^{-1} \tilde{D}_{1}[\tilde{D}_{1},\tilde{\pi}(ka^{0}k^{-1})]_{\sigma}\omega )\cdots    (\omega^{-1} 
   \tilde{D}_{1}[\tilde{D}_{1},\tilde{\pi}(ka^{2q}k^{-1})]_{\sigma}\omega )\right\} \\
    =  & c_{q}\Str \left\{ 
   \tilde{D}_{1}[\tilde{D}_{1},\tilde{\pi}(ka^{0}k^{-1})]_{\sigma}\cdots     
   \tilde{D}_{1}[\tilde{D}_{1},\tilde{\pi}(ka^{2q}k^{-1})]_{\sigma}\right\} \\   
    = &  \overline{\tau}_{2q}^{\tilde{D}_{1},\sigma}(ka^{0}k^{-1},\ldots,ka^{2q}k^{-1}).
    \label{eq:tau_2q.D_k.D.not.invertible}
\end{align}
Therefore, in the same way as in the invertible case, we deduce that the cocycles $ 
\overline{\tau}_{2q}^{\tilde{D}_{k},\hat{\sigma}}$ and $ \overline{\tau}_{2q}^{\tilde{D}_{1},\sigma}$ are cohomologous 
in $H_{\lambda}^{2q}(\cA)$. It then follows that, for $q\geq \frac{1}{2}(p+1)$, the cocycles $ 
\overline{\tau}_{2q}^{\tilde{D},\sigma}$ and $ \overline{\tau}_{2q}^{\tilde{D}_{k},\hat{\sigma}}$ are cohomologous in 
$H_{\lambda}^{2q}(\cA)$, and hence define the same class in $\op{HP}^{0}(\cA)$. This shows that the twisted spectral 
triples $(\cA,\cH,D)_{\sigma}$ and $(\cA,\cH,D_{k})_{\hat{\sigma}}$ have same Connes-Chern character. The proof is 
complete. 
\end{proof}

\subsection{Smooth twisted spectral triples} We shall now explain how to extend to smooth twisted spectral triples the 
previous results of this section. First, we have the following result.

\begin{proposition}
    Let $(\cA,\cH_{1},D_{1})_{\sigma_{1}}$ and $(\cA,\cH_{2},D_{2})_{\sigma_{2}}$ be equivalent smooth twisted spectral 
    triples that are $p$-summable for some $p\geq 1$. Then $\bCh(D_{1})_{\sigma_{1}}=\bCh(D_{2})_{\sigma_{2}}$ in 
    $\op{\mathbf{HP}}^{0}(\cA)$. 
\end{proposition}
\begin{proof}
    This is an immediate consequence of the first two parts of Proposition~\ref{prop:equiv.TST.Same.CC}, since they imply that 
    $\bCh(D_{1})_{\sigma_{1}}$ and $\bCh(D_{2})_{\sigma_{2}}$ are represented by the same cocycles. 
\end{proof}

\begin{proposition}
\label{prop:conf.Dirac.same.Ch.LCA}
    Assume that $(\cA,\cH,D)_{\sigma}$ is $p$-summable for some $p\geq 1$. Then, for any positive element $k\in \cA$, 
    we have
    \begin{equation*}
        \bCh(kDk)_{\hat{\sigma}}= \bCh(D)_{\sigma}\in \op{\mathbf{HP}}^{0}(\cA).
    \end{equation*}    
\end{proposition}
\begin{proof}
 We shall continue using the notation of the proof of Proposition~\ref{prop:conf.Dirac.same.Ch}. This proof shows that, for any $q\geq 
 \frac{1}{2}(p+1)$, the cocycles $\overline{\tau}^{D,\sigma}$ and $\overline{\tau}^{D_{k},\hat\sigma}$ define the same 
 class in $H_{\lambda}^{2q}(\cA)$ by establishing that they both are cohomologous to the cocycle 
 $\overline{\tau}^{\tilde{D}_{1},\sigma}$. In order to prove the result we only need to show that 
 $\overline{\tau}^{D,\sigma}$ and $\overline{\tau}^{D_{k},\hat\sigma}$ both are cohomologous to  
 $\overline{\tau}^{\tilde{D}_{1},\sigma}$  in $\mathbf{H}_{\lambda}^{2q}(\cA)$.  
 
 The equality between the classes of $\overline{\tau}^{D,\sigma}$ and  
 $\overline{\tau}^{\tilde{D}_{1},{\sigma}}$ in $H_{\lambda}^{2q}(\cA)$ is a consequence of the homotopy invariance of the Connes-Chern character and the 
 fact that the operators $\tilde{D}$ and $\tilde{D}_{1}$ can be connected by a operator homotopy of the 
 form~(\ref{eq:homotopy.D.D_1}). Using Lemma~\ref{lem:path.saDirac.cyclic.cocycle} we then see that the cocycles ${\tau}^{\tilde{D},\tilde{\sigma}}$ and  
 ${\tau}^{\tilde{D}_{1}, {\sigma}}$ are cohomologous in $\mathbf{H}_{\lambda}^{2q}(\tilde{\cA})$. Therefore, their 
 restrictions to $\cA$, i.e., the cocycles $\overline{\tau}^{D,\sigma}$ and  
 $\overline{\tau}^{\tilde{D}_{1},\sigma}$, define the same class in $\mathbf{H}_{\lambda}^{2q}(\cA)$. 
 
 In addition, Eq.~(\ref{eq:tau_2q.D_k.D.not.invertible}) shows that the cocycle $\overline{\tau}^{D_{k},\hat\sigma}$ is obtained from 
 $\overline{\tau}^{\tilde{D}_{1},{\sigma}}$ by composing with the inner automorphism defined by $k$. 
 The proof of the invariance of cyclic cohomology by the action of inner automorphisms in~\cite{Lo:CH} holds 
 \emph{verbatim} for the cyclic cohomology of continuous cochains. It then follows that the cocycles
 $\overline{\tau}^{D_{k},\hat\sigma}$ and 
 $\overline{\tau}^{\tilde{D}_{1},{\sigma}}$ define the same class in $\mathbf{H}_{\lambda}^{2q}(\cA)$, and so
 $\overline{\tau}^{D,\sigma}$ and $\overline{\tau}^{D_{k},\hat\sigma}$ are cohomologous in 
 $\mathbf{H}_{\lambda}^{2q}(\cA)$.  The proof is complete. 
\end{proof}

\section{The Conformal Connes-Chern Character}\label{sec:Conformal-CC-character}
In this section, after recalling the construction of the conformal Dirac spectral triple of~\cite{CM:TGNTQF} associated with any given conformal structure, we show that its 
Connes-Chern character (defined in continuous cyclic cohomology) is  actually a conformal invariant.

\subsection{The conformal Dirac spectral triple}
\label{sec:ConformalDiracST}
Throughout this section and the rest of the paper we let $M$ be a compact (closed) spin oriented manifold of even dimension $n$. We also let $\sC$ be a 
conformal structure on $M$, i.e., a conformal class of Riemannian metrics on $M$. We then denote by $G$ the identity component of the group of (smooth) 
orientation-preserving diffeomorphisms of $M$ preserving the  conformal structure $\sC$ and the spin structure of $M$. 
Let $g$ be a representative metric in the conformal class $\sC$, and consider the associated Dirac operator 
$\sD_{g}:C^{\infty}(M,\sS)\rightarrow C^{\infty}(M,\sS)$ on the sections of the spinor bundle
$\sS=\sS^{+}\oplus \sS^{-}$.  In addition, we denote by $L^{2}_{g}(M,\sS)$ the corresponding Hilbert space of $L^{2}$-spinors. 

In the setup of noncommutative geometry, the role of the quotient space $M/G$ is played by the (discrete) crossed-product algebra 
$C^{\infty}(M)\rtimes G$. The underlying vector space of $C^{\infty}(M)\rtimes G$ is $\C G \otimes C^{\infty}(M)$, where 
$\C G$ is the group algebra of $G$ over $\C$. Given $\phi \in G$ and $f \in C^{\infty}(M)$, it is convenient to denote the 
tensor products $1\otimes \phi$  and $f\otimes 1$ by $u_{\phi}$ and $f$, respectively. Using this notation, any element 
of $C^{\infty}(M)\rtimes G$ has a unique representation as a finite sum $\sum f_{\phi}u_{\phi}$, $f_{\phi}\in 
C^{\infty}(M)$, so that we have the direct-sum decomposition,
\begin{equation}
C^{\infty}(M)\rtimes G =\bigoplus_{\phi \in G}C^{\infty}(M)u_{\phi}. 
\label{eq:ConformalST.direct-sum-decomposition-cAG}
\end{equation}

The action of $G$ on $M$ gives to an action on $C^{\infty}(M)$ given by
\begin{equation}
    \phi \cdot f=f\circ \phi^{-1}, \qquad \phi\in G, \ f\in C^{\infty}(M).
    \label{eq:ConformalCCC.action-on-functions}
\end{equation}
The product and involution of $C^{\infty}(M)\rtimes G$ are given by
\begin{gather}
    (f^{0}u_{\phi_{0}})(f^{1}u_{\phi_{1}})=f^{0}(\phi_{0}\cdot f^{1})u_{\phi_{0}\phi_{1}}, \qquad \phi_{j}\in G,\ 
    f^{j}\in C^{\infty}(M), 
    \label{eq:ConformalST.product-cAG}\\
    \left( fu_{\phi}\right)^{*}=u_{\phi}^{-1}\overline{f}=(\phi^{-1} \cdot \overline{f})u_{\phi^{-1}}, \qquad \phi\in 
    G,\ f\in C^{\infty}(M).
       \label{eq:ConformalST.involution-cAG}
\end{gather}
In particular, we have the relations,
\begin{equation*}
    u_{\phi}f=(\phi \cdot f)u_{\phi} \qquad \text{and} \qquad u_{\phi}^{*}=u_{\phi}^{-1}=u_{\phi^{-1}}. 
\end{equation*}

Let $\phi\in G$. As $\phi$ is a diffeomorphism preserving the conformal class $\sC$, there is a unique function 
$k_{\phi}\in C^{\infty}(M)$, $k_{\phi}>0$, such that 
\begin{equation}
    \phi_{*}g=k_{\phi}^{2}g.
    \label{eq:TwistedST.conformal-factor}
\end{equation}Moreover, $\phi$ uniquely lifts to a unitary vector bundle isomorphism $\phi^{\sS}:\sS \rightarrow 
\phi_{*}\sS$, i.e., a unitary section of $\Hom(\sS, \phi_{*}\sS)$ (see~\cite{BG:SODVM}). 
We then let $V_{\phi}:L^{2}_{g}(M,\sS)\rightarrow L^{2}_{g}(M,\sS)$ be the bounded operator given by
\begin{equation}
    V_{\phi}u(x) = \phi^{\sS}\left( u\circ \phi^{-1}(x)\right) \qquad \forall u \in L^{2}_{g}(M,\sS)\ \forall x \in M.
    \label{eq:TwistedST.Vphi}
\end{equation}
The map $\phi \rightarrow V_{\phi}$ is a representation of $G$ in $L^{2}_{g}(M,\sS)$, but this is not a unitary 
representation. In order to get a unitary representation we need to take into account the Jacobian 
$|\phi'(x)|=k_{\phi}(x)^{n}$ of 
$\phi \in G$. This is achieved by using the unitary operator $U_{\phi}:L^{2}_{g}(M,\sS)\rightarrow 
L^{2}_{g}(M,\sS)$ given by 
\begin{equation}
    U_{\phi}=k_{\phi}^{\frac{n}{2}}V_{\phi}, \qquad \phi \in G.
    \label{eq:TwistedST.Uphi}
\end{equation}Then $\phi \rightarrow U_{\phi}$ is a unitary representation of $G$ in $L^{2}_{g}(M,\sS)$. Combining this 
with the action of $C^{\infty}(M)$ by multiplication operators provides with a $*$-representation of the 
crossed-product algebra  $C^{\infty}(M)\rtimes G$. In addition, we let $\sigma_{g}$ be the automorphism of $C^{\infty}(M)\rtimes G$ given by
\begin{equation}
    \sigma_{g}(fu_{\phi}):=k_{\phi}fu_{\phi} \qquad \forall f \in C^{\infty}(M) \ \forall \phi \in G.
    \label{eq:TwistedST.automorphism-conformal-DiracST}
\end{equation}

\begin{proposition}[\cite{CM:TGNTQF}]\label{prop:TwistedST.automorphism-conformal-DiracST} 
    The triple $( C^{\infty}(M)\rtimes G, L^{2}_{g}(M,\sS), \,\sD_{g})_{\sigma_{g}}$ is a twisted 
    spectral triple.
\end{proposition}

\begin{remark}
    The bulk of the proof is to show the boundedness of the twisted commutators $[\sD_{g},U_{\phi}]_{\sigma_{g}}$, 
    $\phi \in G$. We remark that 
    \begin{equation*}
        U_{\phi}\sD_{g}U_{\phi}^{*}=k_{\phi}^{\frac{n}{2}}(V_{\phi}\sD_{g}V_{\phi}^{-1})k_{\phi}^{-\frac{n}{2}}=  
       k_{\phi}^{\frac{n}{2}} 
       \sD_{\phi_{*}g}k_{\phi}^{-\frac{n}{2}}=k_{\phi}^{\frac{n}{2}}\sD_{k_{\phi}^{2}g}k_{\phi}^{-\frac{n}{2}}. 
    \end{equation*}Combining this with the conformal invariance of the Dirac operator (see, e.g., \cite{Hi:RITCM}) we obtain
    \begin{equation*}
        U_{\phi}\sD_{g}U_{\phi}^{*}=k_{\phi}^{\frac{n}{2}}\left(  k_{\phi}^{-\left(\frac{n+1}{2}\right)}\sD_{g}k_{\phi}^{\frac{n-1}{2}} \right)k_{\phi}^{-\frac{n}{2}}= 
                k_{\phi}^{-\frac{1}{2}}\sD_{g}k_{\phi}^{-\frac{1}{2}}.
    \end{equation*}
    Using this we see that the twisted commutator $[\sD_{g},U_{\phi}]_{\sigma_{g}}=\sD_{g}U_{\phi}-k_{\phi}U_{\phi}\sD_{g} $ is equal to
    \begin{equation*}
       \left( 
        \sD_{g}k_{\phi}^{\frac{1}{2}}-k_{\phi}(U_{\phi}\sD_{g}U_{\phi}^{*})k_{\phi}^{\frac{1}{2}}\right)k_{\phi}^{-\frac{1}{2}}U_{\phi} 
        = \left( 
        \sD_{g}k_{\phi}^{\frac{1}{2}}-k_{\phi}^{\frac{1}{2}}\sD_{g}\right)k_{\phi}^{-\frac{1}{2}}U_{\phi} = 
        [\sD_{g},k_{\phi}^{\frac{1}{2}}]k_{\phi}^{-\frac{1}{2}}U_{\phi}.
            \end{equation*}This shows that $[\sD_{g},U_{\phi}]_{\sigma_{g}}$ is bounded. 
\end{remark}

\begin{remark}
We shall refer to $( C^{\infty}(M)\rtimes G, L^{2}_{g}(M,\sS), \,\sD_{g})_{\sigma_{g}}$ as the 
\emph{conformal Dirac spectral triple} associated with the representative metric $g$. 
\end{remark}

\begin{remark}
    The automorphism $\sigma_{g}$ is ribbon in the sense mentioned in Remark~\ref{rmk:Twisted-ST.ribbon-condition}. This is seen by using 
the automorphism $\tau_{g}$ of $C^{\infty}(M)\rtimes G$ defined by 
\begin{equation*}
 \tau_{g}(fu_{\phi})=   \sqrt{k_{\phi}}fu_{\phi} \qquad \forall f \in C^{\infty}(M) \ \forall \phi \in G,
\end{equation*}where $k_{\phi}$ is the conformal factor in the sense~(\ref{eq:TwistedST.conformal-factor}). 
\end{remark}

The crossed-product $C^{\infty}(M)\rtimes G$ is endowed with a locally convex space topology as follows. The direct-sum 
decomposition~(\ref{eq:ConformalST.direct-sum-decomposition-cAG}) means that, as a vector space, $C^{\infty}(M)\rtimes G$ is the direct union of the subspaces, 
\begin{equation*}
    C^{\infty}(M)\rtimes F=\bigoplus_{\phi \in F}C^{\infty}(M)u_{\phi}, \qquad \text{$F\subset G$ finite}. 
\end{equation*}
Given a finite subset $F\subset G$, we have a canonical vector space identification $C^{\infty}(M)\rtimes F\simeq 
C^{\infty}(M)^{F}$. Note that $C^{\infty}(M)^{F}$ is a Fr\'echet space with respect to the product topology. Pulling 
back this topology to $C^{\infty}(M)\rtimes F$ we obtain a canonical Fr\'echet space topology on this space. We then 
endow  $C^{\infty}(M)\rtimes G$ with the coarsest locally convex topology that makes continuous the inclusions of the 
subspaces $C^{\infty}(M)\rtimes F$ into $C^{\infty}(M)\rtimes G$ as $F$ ranges over finite subsets of $G$. Thus a basis 
of neighborhoods of the origin in $C^{\infty}(M)\rtimes G$ consists of convex balanced subsets $\cU$ such that, for all 
every finite subset $F\subset G$, the intersection $\cU\cap (C^{\infty}(M)\rtimes F)$ is a neighborhood of $0$ in 
$C^{\infty}(M)\rtimes F$. In particular, given a topological vector space $X$, a linear map $T:C^{\infty}(M)\rtimes 
G\rightarrow X$ is continuous if and only if, for all $\phi\in G$, the linear map $f\rightarrow T(fu_{\phi})$ is 
continuous from $C^{\infty}(M)$ to $X$. 

We also observe that, given $\phi_{0}$ and $\phi_{1}$ in $G$, the multiplication~(\ref{eq:ConformalST.product-cAG}) induces a jointly continuous 
bilinear map $ \left( C^{\infty}(M)u_{\phi_{0}}\right)\times \left(C^{\infty}(M)u_{\phi_{1}}\right) 
    \rightarrow C^{\infty}(M)u_{\phi}$. 
In addition, given any $\phi\in G$, the involution~(\ref{eq:ConformalST.involution-cAG}) induces a continuous anti-linear map $ C^{\infty}(M)u_{\phi}\longrightarrow 
    C^{\infty}(M)u_{\phi^{-1}}$. It then follows that the multiplication of the crossed-product algebra 
    $C^{\infty}(M)\rtimes G$ is a jointly continuous bilinear and its involution is continuous. Therefore, we see that 
    $C^{\infty}(M)\rtimes G$ is a locally convex $*$-algebra with respect the topology described above. 

  \begin{lemma}
     For any metric $g\in \sC$, the twisted spectral triple $( C^{\infty}(M)\rtimes G, 
 L^{2}_{g}(M,\sS),\sD_{g})_{\sigma_{g}}$ is smooth in the sense of Definition~\ref{def:TST.smooth}. 
 \end{lemma}
 \begin{proof}
   Let $\phi\in G$. The map $f\rightarrow fU_{\phi}$ is continuous from $C^{\infty}(M)$ to $\cL\left( 
   L^{2}_{g}(M,\sS)\right)$. As $\sigma_{g}(fu_{\phi})=k_{\phi}fu_{\phi}$ and 
   $\sigma_{g}^{-1}(fu_{\phi})=k_{\phi}^{-1}fu_{\phi}$, the maps $\phi \rightarrow \sigma_{g}^{\pm 1}(fu_{\phi})$ are 
   continuous from $C^{\infty}(M)$ to $C^{\infty}(M)\rtimes G$. In addition, note that, for any $f\in C^{\infty}(M)$, 
   \begin{equation*}
       [\sD_{g},fU_{\phi}]_{\sigma_{g}}=[\sD_{g},f]U_{\phi}+f[\sD_{g},U_{\phi}]_{\sigma_{g}}=ic(df)U_{\phi}+f[\sD_{g},U_{\phi}]_{\sigma_{g}}, 
   \end{equation*}where $c(df)$ is the Clifford representation of the differential $df$. Thus, the map 
   $f\rightarrow [\sD_{g},fU_{\phi}]_{\sigma_{g}}$ is continuous from $C^{\infty}(M)$ to $\cL\left( 
   L^{2}_{g}(M,\sS)\right)$. It follows from all this that the conditions (1)--(3) of 
   Definition~\ref{def:TST.smooth} are satisfied, and so 
   $( C^{\infty}(M)\rtimes G,  L^{2}_{g}(M,\sS),\sD_{g})_{\sigma_{g}}$ is a smooth twisted spectral triple . The lemma is thus proved. 
 \end{proof}

\subsection{Conformal invariance of the Connes-Chern character}
The construction of the conformal Dirac spectral triple $( C^{\infty}(M)\rtimes G, L^{2}_{g}(M,\sS),\sD_{g})_{\sigma_{g}}$, 
depends on the choice of a representative metric $g$ in the conformal class $\sC$. 
The following proposition describes the dependence on this choice.

\begin{proposition}\label{prop:Auto.Conf.Change}
     Let $\hat{g}$ be another metric in the conformal class $\sC$, i.e., $\hat{g}=k^{2}g$ for some function $k\in C^{\infty}(M)$, 
     $k>0$. Let $\hat{\sigma}$ be the automorphism of $C^{\infty}(M)\rtimes G$ defined by
     \begin{equation}
     \label{eq:sigma.hat.conf}
         \hat{\sigma}(a)=k^{-\frac{1}{2}}\sigma_{g}(k^{-\frac{1}{2}}ak^{\frac{1}{2}})k^{\frac{1}{2}}= 
         k^{-1}\sigma_{g}(a)k, \qquad 
         a\in C^{\infty}(M)\rtimes G. 
     \end{equation}
 Then the conformal Dirac spectral $( C^{\infty}(M)\rtimes G, 
 L^{2}_{\hat{g}}(M,\sS),\sD_{\hat{g}})_{\sigma_{\hat{g}}}$ associated with $\hat{g}$ is equivalent to the conformal deformation 
 $(C^{\infty}(M)\rtimes G, L^{2}_{g}(M,\sS),k^{-\frac{1}{2}}\sD_{g}k^{-\frac{1}{2}})_{\hat{\sigma}}$. 
  \end{proposition}
 \begin{proof}
     Let $U:L^{2}_{\hat{g}}(M,\sS)\rightarrow L^{2}_{g}(M,\sS)$ be the operator given by the multiplication by 
     $k^{\frac{n}{2}}$. Let $u\in L^{2}_{\hat{g}}(M,\sS)$. As $|dx|_{\hat{g}}=k(x)^{n}|dx|_{g}$, we have
     \begin{equation*}
         \|Uu\|_{L^{2}_{g}(M,\sS)}^{2}=\int_{M} \acoup{k(x)^{\frac{n}{2}}u(x)}{k(x)^{\frac{n}{2}}u(x)} |dx|_{g} 
         =\int_{M}\acoup{u(x)}{u(x)}|dx|_{\hat{g}}= \|u\|_{L^{2}_{\hat{g}}(M,\sS)}^{2}.
     \end{equation*}This shows that $U$ is a unitary operator. Moreover, using the conformal invariance of the Dirac 
     operator~\cite{Hi:HS} we see that
     \begin{equation}
     \label{eq:conf.Inva.Dirac}
         \sD_{\hat{g}}=k^{-\frac{1}{2}(n+1)}\sD_{g}k^{\frac{1}{2}(n-1)}=U^{*}k^{-\frac{1}{2}}\sD_{g}k^{-\frac{1}{2}}U.
     \end{equation}

     Let $\phi\in G$. We have two representations of $\phi$. One is the unitary operator $V_{\phi}$  of 
     $L^{2}_{g}(M,\sS)$ given by~(\ref{eq:TwistedST.Vphi}) using the representative metric $g$. We have another representation of $\phi$ 
     as a unitary operator $\widehat{V}_{\phi}$ of  $L^{2}_{\hat{g}}(M,\sS)$ given by the same formula using the metric $\hat{g}$. That is, 
     \begin{equation*}
         \widehat{V}_{\phi}u=e^{n\hat{h}_{\phi}}\phi_{*}u, \qquad u \in L^{2}_{\hat{g}}(M,\sS),
     \end{equation*}where $\hat{h}_{\phi}\in C^{\infty}(M,\R)$ and satisfies that $\phi_{*}\hat{g}=e^{2h_{\phi}}\hat{g}$. Set $h=\log 
     k$. Then
     \begin{equation*}
         \phi_{*}\hat{g}=\phi_{*}(k^{2}g)=(k\circ \phi^{-1})^{2}\phi_{*}g=e^{2h}e^{2h_{\phi}}g=e^{2(h\circ 
         \phi^{-1}+h_{\phi}-h)}\hat{g},
     \end{equation*}which shows that $\hat{h}_{\phi}=h\circ 
         \phi^{-1}+h_{\phi}-h$. Let $u \in L^{2}_{\hat{g}}(M,\sS)$. Then
     \begin{equation*}
         \widehat{V}_{\phi}u=e^{-nh}e^{nh_{\phi}}e^{-nh\circ 
         \phi}\phi_{*}u=k^{-\frac{n}{2}}e^{nh_{\phi}}\phi_{*}(k^{\frac{n}{2}}u)=U^{*}V_{\phi}Uu.
     \end{equation*}Likewise, 
     \begin{equation*}
         \sigma_{\hat{g}}(\widehat{V}_{\phi})u=e^{-(n+1)\hat{h}_{\phi}}\phi_{*}u=k^{-(n+1)}e^{(n+1)h_{\phi}}\phi_{*}(k^{n+1}u)= 
         Uk^{-1}\sigma_{g}(V_{\phi})kUu=U^{-1}\hat{\sigma}(V_{\phi})Uu. 
     \end{equation*}We then deduce that, for all $f\in C^{\infty}(M)$ and $\phi\in G$,
     \begin{equation*}
         f\widehat{V}_{\phi}=U^{*}(fV_{\phi})U  \qquad \text{and} \qquad 
         \sigma_{\hat{g}}(f\widehat{V}_{\phi})=U^{*}\hat{\sigma}(fV_{\phi})U.
     \end{equation*}
     Combining this with~(\ref{eq:conf.Inva.Dirac}) shows that the twisted spectral triples $( C^{\infty}(M)\rtimes G, 
 L^{2}_{\hat{g}}(M,\sS),\sD_{\hat{g}})_{\sigma_{\hat{g}}}$ and 
 $(C^{\infty}(M)\rtimes G, L^{2}_{g}(M,\sS),k^{-\frac{1}{2}}\sD_{g}k^{-\frac{1}{2}})_{\hat{\sigma}}$ are equivalent. The proof is complete. 
 \end{proof}
 
 Given any metric $g\in \sC$, the $j$-th eingenvalue of $|\sD_{g}|$ grows like 
 $j^{\frac{1}{n}}$ as $j\rightarrow \infty$. Therefore, the associated twisted spectral triple $( C^{\infty}(M)\rtimes G, 
 L^{2}_{g}(M,\sS),\sD_{g})_{\sigma_{g}}$ is $p$-summable for all $p>n$. Combining this with Proposition~\ref{prop:equiv.TST.Same.CC} shows that the Connes-Chern character of $( C^{\infty}(M)\rtimes G, 
 L^{2}_{g}(M,\sS),\sD_{g})_{\sigma_{g}}$ is well defined as a class in $\op{\mathbf{HP}}^{0}( C^{\infty}(M)\rtimes G)$ for every 
 metric $g$ in the conformal class $\sC$.  
 
We are now in a position to state the main result of this section.
 
\begin{theorem}\label{thm:CC.conf.inv}
    The Connes-Chern character $\op{\mathbf{Ch}}(\sD_{g})_{\sigma_{g}}\in \op{\mathbf{HP}}^{0}\left(C^{\infty}(M)\rtimes G\right)$ is 
    independent of the choice of the metric $g\in \sC$, i.e, it is an invariant of the conformal structure $\sC$. 
\end{theorem}
\begin{proof}
Let $\hat{g}$ be another metric in the conformal class $\sC$, so that $\hat{g}=k^{2}g$ with $k\in C^{\infty}(M)$, 
     $k>0$. Combining Proposition~\ref{prop:Auto.Conf.Change} with Proposition~\ref{prop:equiv.TST.Same.CC} and 
 Proposition~\ref{prop:conf.Dirac.same.Ch.LCA} we get
 \begin{equation*}
   \bCh(\sD_{\hat{g}})_{\sigma_{\hat{g}}}=\bCh\left(k^{-\frac{1}{2}}\sD_{g}k^{-\frac{1}{2}}\right)_{\hat{\sigma}}=\bCh(\sD_{g})_{\sigma_{g}} \qquad  
   \text{in $\op{\mathbf{HP}}^{0}\left(C^{\infty}(M)\rtimes G\right)$},
 \end{equation*}where $\hat{\sigma}$ is defined as in~(\ref{eq:sigma.hat.conf}). This proves the result.  
\end{proof}

 This leads us to the following definition. 
 
 \begin{definition}
  The Connes-Chern character of the conformal class $\sC$, denoted by   $\bCh(\sC)$, 
  is the common class in $\op{\mathbf{HP}}^{0}\left(C^{\infty}(M)\rtimes G\right)$ of the Connes-Chern characters of the 
  conformal Dirac spectral triples $( C^{\infty}(M)\rtimes G, 
  L^{2}_{g}(M,\sS),\sD_{g})_{\sigma_{g}}$ as the metric $g$ ranges over the conformal class $\sC$. 
 \end{definition}
 
Combining this with Proposition~\ref{prop:CCC-index-formula2} we obtain the following index formula.

\begin{proposition}\label{prop:ConformalCCC.index-formula}
 Let $\cE$ be a finitely generated projective module over $C^{\infty}(M)\rtimes G$. Then, for any metric $g\in \sC$ and any $\sigma_{g}$-connection on $\cE$, it 
 holds that
 \begin{equation*}
     \ind D_{\nabla^{\cE}}=\acou{\bCh(\sC)}{\bCh(\cE)}.
 \end{equation*}
\end{proposition}

\section{Local Index Formula in Conformal Geometry}\label{sec:LIF-Conformal-Geometry}
In this section, we shall compute the conformal Connes-Chern character $\bCh(\sC)$ when the conformal structure $\sC$ 
is not flat. Together with 
Proposition~\ref{prop:ConformalCCC.index-formula} this will provide us with the local index formula in 
conformal-diffeomorphism invariant geometry. We recall that the conformal structure $\sC$ is flat when it is equivalent to the conformal 
structure of the round sphere $\mathbb{S}^{n}$. In case $M$ is simply connected and has dimension~$\geq 4$ this is 
equivalent to the vanishing of the Weyl curvature tensor of $M$ (see~\cite{Ku:CFSL}).

As pointed out in Remark~\ref{rem:Moscivici.ansatz}, the ansatz of Moscovici~\cite{Mo:LIFTST} is not known to hold for conformal 
deformations of ordinary spectral triples sastisfying the local index formula in noncommutative geometry 
of~\cite{CM:GAFA95}. As a result, unless the conformal structure $\sC$ is flat and $G$ is a maximal parabolic subgroup of 
$\op{PO}(n+1,1)$,  we cannot claim that for a general metric in $\sC$ the corresponding conformal Dirac spectral triple 
$( C^{\infty}(M)\rtimes G, L^{2}_{g}(M,\sS), \,\sD_{g})_{\sigma_{g}}$ satisfies Moscovici's ansatz. 

Having said this, the conformal invariance of the Connes-Chern character $\op{\mathbf{Ch}}(\sD_{g})_{\sigma_{g}}$ provided by Theorem~\ref{thm:CC.conf.inv} allows us to 
choose \emph{any} metric in 
the conformal class $\sC$ to compute $\bCh(\sC)$. In particular, as we shall see below (and as also observed by Moscovici~\cite{Mo:LIFTST}), 
the computation is greatly simplified by choosing a $G$-invariant metric in the conformal class $\sC$. When the conformal structure $\sC$ is non-flat, the existence 
of such a metric is ensured by the following result.

\begin{proposition}[Ferrand-Obata~\cite{Ba:ECGCI, Fe:ACTRM, Sc:OCCRAG}] If the conformal structure $\sC$ is 
    not flat, then the group of smooth diffeomorphisms of $M$ preserving $\sC$ is a compact Lie group and there 
    is a metric in $\sC$ that is invariant by this group.
\end{proposition}


The relevance of using a $G$-invariant metric $g\in \sC$ stems from the observation that in this case  $\phi_{*}g=g$ for all 
$\phi\in G$, and so, for every $\phi\in G$, the conformal factor $k_{\phi}$ is always the constant function $1$. This 
implies that, for any $\phi\in G$, the unitary operator $U_{\phi}$ agrees with the pushforward by $\phi$. This also 
implies that the automorphism $\sigma_{g}$ given by~(\ref{eq:TwistedST.automorphism-conformal-DiracST}) is trivial, so that the conformal Dirac spectral triple 
$(C^{\infty}(M)\rtimes G, L^{2}_{g}(M,\sS), \,\sD_{g})_{\sigma_{g}}$ is actually an \emph{ordinary} spectral triple. Therefore, we 
arrive at the following statement. 

\begin{proposition}
\label{prop:representative.conformal.Chern.char}
   The conformal Connes-Chern character $\bCh(\sC)$ agrees with the ordinary 
  Connes-Chern character of the equivariant Dirac spectral triple $(C^{\infty}(M)\rtimes G,L^{2}_{g}(M,\sS),\sD_{g})$ 
  associated with any $G$-invariant metric $g$ in the 
  conformal class $\sC$.
\end{proposition}

The computation of the Connes-Chern character of an equivariant Dirac spectral triple $(C^{\infty}(M)\rtimes G,L^{2}_{g}(M,\sS),\sD_{g})$ 
  associated with any $G$-invariant metric $g\in\sC$ is carried out in~\cite{PW:JNCG16}. In order to present the results of~\cite{PW:JNCG16}  
  we need to introduce some notation. 

In what follows we let $g$ be a $G$-invariant metric in the conformal class $\sC$. Given $\phi \in G$ we denote by $M^{\phi}$ its 
fixed-point set. As $\phi$ preserves the orientation and the metric $g$, we have the following disjoint-union decomposition,
\begin{equation*}
    M^{\phi}=\bigsqcup_{n-a\in 2\N_{0}}M_{a}^{\phi}, \qquad \text{where } M_{a}^{\phi}=\left\{x\in M^{\phi};\ 
    \op{rk}(\phi'(x)-\op{id})=a\right\}, 
\end{equation*}
and each component $M_{a}^{\phi}$ is a submanifold of $M$ of dimension $a$. For $a=0,2,\ldots,n$ we let $\cN_{a}^{\phi}=(TM^{\phi})^{\perp}$ be the normal 
bundle of $M^{\phi}_{a}$; this is a smooth vector bundle over $M_{a}^{\phi}$. We denote by $\phi^{\cN}$ the isometric vector bundle isomorphism  
induced on  $\cN^{\phi}_{a}$ by $\phi'$. We note that the eigenvalues of $\phi^{\cN}$ are either $-1$ (which has even 
multiplicity) or complex conjugates $e^{\pm i 
\theta}$, $\theta \in (0,\pi)$, with same multiplicity. In addition, we shall orient $M^{\phi}_{a}$ like 
in~\cite[Prop.~6.14]{BGV:HKDO}, so that the vector bundle isomorphism $\phi^{\sS}:\sS\rightarrow \phi_*\sS$ gives rise 
to a section of $\Lambda^{n-a}(\cN^{\phi}_{a})^{*}$ 
which is positive with respect to the orientation of $\cN^{\phi}$ defined by the orientations of $M$ and $M^{\phi}_{a}$.

For $a=0,2,\ldots,n$, we let $\Omega(M_{a}^{\phi})=(\Omega^{\bt}(M_{a}^{\phi}),\wedge, d)$ be the differential graded 
algebra (DGA) of differential forms on $M_{a}^{\phi}$. We then define $\Omega(M^{\phi})$ as the DGA obtained as the 
direct sum of the DGAs  $\Omega(M_{a}^{\phi})$, $a=0,2,\ldots,n$. We shall  refer to elements of 
$\Omega^{\bt}(M^{\phi})$ as differential forms on $M^{\phi}$. Given $\omega\in \Omega^{\bt}(M^{\phi})$ we shall denote by $\restr{\omega}{M_{a}^{\phi}}$ its 
component in $\Omega^{\bt}(M_{a}^{\phi})$. Note that
\begin{equation*}
    \restr{(\omega^{1}\wedge \omega^{2})}{M_{a}^{\phi}}= \left(\restr{\omega^{1}}{M_{a}^{\phi}}\right) \wedge  
    \left(\restr{\omega^{2}}{M_{a}^{\phi}}\right) , \qquad \omega^{j}\in \Omega(M^{\phi}), \quad a=0,2,\ldots,n.
\end{equation*}
For $a=0,2,\ldots,n$, we let $\iota_{M_{a}^{\phi}}^{*}:\Omega(M)\rightarrow \Omega(M_{a}^{\phi})$ be the DGA map given 
by pulling-pack by the smooth inclusion $\iota_{M^{\phi}_{a}}:M_{a}^{\phi}\hookrightarrow M$. Taking direct sums of 
these DGA maps then provides us with a DGA map $\iota_{M^{\phi}}^{*}:\Omega(M)\rightarrow \Omega(M^{\phi})$. 
We also define the integration $\int_{M^{\phi}}:\Omega(M^{\phi})\rightarrow \C$ by 
\begin{equation}
    \int_{M^{\phi}}\omega:=\sum_{n-a\in 2\N_{0}}\int_{M_{a}^{\phi}}\left(\restr{\omega}{M_{a}^{\phi}}\right)^{(a)}, \qquad \omega \in 
    \Omega(M^{\phi}),
    \label{eq:Conformal.integration-Mphi}
\end{equation}where $(\restr{\omega}{M_{a}^{\phi}})^{(a)}\in \Omega^{a}(M^{\phi}_{a})$ is the top degree component of 
$\restr{\omega}{M_{a}^{\phi}}$.

In addition, as the Levi-Civita connection $\nabla^{TM}$ is preserved by $\phi$, for $a=0,2,\ldots, n$, it preserves the orthogonal splitting 
$\left.TM\right|_{M_{a}^{\phi}}=TM^{\phi}_{a}\oplus \cN^{\phi}_{a}$, and so it induces a connection 
$\nabla^{\cN^{\phi}_{a}}$ on $\cN^{\phi}_{a}$ in such a way that 
\begin{equation*}
   \left. \nabla^{TM}\right|_{|TM^{\phi}_{a}}=\nabla^{TM^{\phi}_{a}}\oplus \nabla^{\cN^{\phi}_{a}},
\end{equation*}where $\nabla^{TM^{\phi}_{a}}$ is the Levi-Civita connection of $TM^{\phi}_{a}$. 
Let $R^{TM^{\phi}_{a}}$ and $R^{\cN^{\phi}_{a}}$ be the respective curvatures of $\nabla^{TM^{\phi}_{a}}$ and 
$\nabla^{\cN^{\phi}_{a}}$. We associate with these curvatures the following characteristic forms,
\begin{equation}
    \hat{A}\left(R^{TM^{\phi}_{a}}\right):= {\det}^{\frac{1}{2}} \left(\frac{R^{TM^{\phi}_{a}}/2}{\sinh(R^{TM^{\phi}_{a}}/2)}\right) 
    \quad \text{and} \quad  
    \cV_{\phi}\left(R^{\cN^{\phi}_{a}}\right):={\det}^{-\frac{1}{2}}\left(1- \phi^{\cN}e^{-R^{\cN^{\phi}_{a}}}\right),
    \label{eq:Conformal.characteristic-forms}
\end{equation}where ${\det}^{-\frac{1}{2}}(1- \phi^{\cN}e^{-R^{\cN^{\phi}_{a}}})$ is defined in the same way as 
in~\cite[Section 6.3]{BGV:HKDO}. In particular, $\hat{A}(R^{TM^{\phi}_{a}})$ is the $\hat{A}$-form of the curvature 
$R^{TM^{\phi}_{a}}$ and the zeroth degree component of the characteristic form $ \cN_{\phi}(R^{\cN^{\phi}_{a}})$ is the 
Lefschetz number ${\det}^{-\frac{1}{2}}(1- \phi^{\cN})$. We then let $\Upsilon^{\phi}(R)\in \Omega(M^{\phi})$ be the 
differential form on $M^{\phi}$ defined by
\begin{equation}
  \restr{\Upsilon^{\phi}(R)}{M^{\phi}_{a}}=(-i)^{\frac{n}{2}}(2\pi)^{-\frac{a}{2}}\hat{A}\left(R^{TM^{\phi}_{a}}\right)\wedge 
 \cV_{\phi}\left(R^{\cN^{\phi}_{a}}\right), \qquad a=0,2,\ldots,n.
 \label{eq:Conformal.Upsilon-phi}
\end{equation}
Note that $\Upsilon^{\phi}(R)$ is a closed differential form on $M^{\phi}$. 

\begin{theorem}[{\cite[Theorem~7.8]{PW:JNCG16}}]
    Let $g$ be a $G$-invariant metric. Then the Connes-Chern character $\bCh(D)$ of the equivariant Dirac spectral 
    triple  $(C^{\infty}(M)\rtimes G, L^{2}_{g}(M,\sS), \sD_{g})$ is represented in 
    $\op{\mathbf{HP}}^{0}(C^{\infty}(M)\rtimes G)$ by the cocycle $\varphi^{\CM}=(\varphi_{2q})_{q\geq 0}$ defined by 
\begin{equation}
 \varphi_{2q}(f^{0}u_{\phi_{0}}, \cdots, f^{2q}u_{\phi_{2q}})= 
 \frac{1}{(2q)!}\int_{M^{\phi}}  \Upsilon^{\phi}(R)\wedge \iota_{M^{\phi}}^{*}\left(f^{0}d\hat{f}^{1} \wedge \cdots \wedge 
             d\hat{f}^{2q}\right), 
             \label{eq:Conformal.CM-cocycle}
\end{equation}where we have set $\phi:=\phi_{0}\circ \cdots \circ \phi_{2q}$
 and $\hat{f}^{j}:=f^{j}\circ \phi_{j-1}^{-1}\circ \cdots \circ \phi_{0}^{-1}$, $j=1,\ldots,2q$.
\end{theorem}
 
Combining this with Proposition~\ref{prop:representative.conformal.Chern.char} we obtain the following index formula in 
conformal-diffeomorphism invariant geometry.

\begin{theorem}\label{thm:LIF-Conformal-Geometry}
 Assume that the conformal structure $\sC$ is non-flat. 
\begin{enumerate}
    \item  For any $G$-invariant metric $g\in \sC$, the conformal Connes-Chern character $\bCh(\sC)$  
is represented in $\op{\mathbf{HP}}^{0}(C^{\infty}(M)\rtimes G)$ by the cocycle $\varphi$ associated with $g$ 
by~(\ref{eq:Conformal.CM-cocycle}). 

    \item  Let $\cE$ be a finitely generated projective module over $C^{\infty}(M)\rtimes G$. Then, for any metric $g\in \sC$ and any $\sigma_{g}$-connection $\nabla^{\cE}$ on $\cE$, 
we have
 \begin{equation*}
     \ind D_{\nabla^{\cE}}=\acou{\varphi}{\bCh(\cE)},
 \end{equation*}where $\bCh(\cE)$ is the Chern character of $\cE$ and 
 $\varphi$ is the cocycle~(\ref{eq:Conformal.CM-cocycle}) associated with any $G$-invariant metric in $\sC$.  
\end{enumerate}    
\end{theorem}

\begin{remark}[See also~\cite{Mo:LIFTST}] 
 For $q=\frac{1}{2}n$ the right-hand side~of~(\ref{eq:Conformal.CM-cocycle}) reduces to an integral 
over $M_{n}^{\phi}$. This submanifold is a disjoint union of connected components of $M$ and on each of those 
components $\phi$ is the identity. Therefore, if $M$ is connected, then $M^{\phi}$ is empty unless $\phi=\op{id}$. In 
this case we then have
\begin{equation}
\label{eq:transverse.fundamental.class}
    \varphi_{n}(f^{0}u_{\phi_{0}}, \cdots, f^{n}u_{\phi_{n}})= \left\{ 
    \begin{array}{ll} {\displaystyle \frac{(2i\pi)^{-\frac{n}{2}}}{n!}     \int_{M}  
            f^{0}d\hat{f}^{1} \wedge \cdots \wedge d\hat{f}^{n} } & \text{if $\phi_{0}\circ \cdots \circ 
            \phi_{n}=\op{id}$},\\
0     & \text{otherwise}.  
    \end{array}\right.
\end{equation}That is, $\varphi_{n}$ agrees with the transverse fundamental class cocycle of Connes~\cite{Co:Kyoto83}.  
\end{remark}

\begin{remark}
    The computation in~\cite{PW:JNCG16} of the Connes-Chern character in $\op{\mathbf{HP}}(C^{\infty}(M)\rtimes G)$ of an equivariant Dirac spectral triple 
    $(C^{\infty}(M)\rtimes G,L^{2}_{g}(M,\sS),\sD_{g})$ is divided into two main 
    steps. The first step consists in showing that the Connes-Chern character is represented in $\op{\mathbf{HP}}^0(C^{\infty}(M)\rtimes G)$ 
    (and not just in $\op{HP}^0(C^{\infty}(M)\rtimes G)$) by the CM cocycle. This requires extending 
    the local index formula of~\cite{CM:CMP93, CM:GAFA95} to the setting of smooth spectral triples. 
    The second step is the 
    explicit computation of that CM cocycle as an immediate byproduct of a new proof of the equivariant local index theorem of Patodi~\cite{Pa:BAMS}, Donnelly-Patodi~\cite{DP:T} and Gilkey~\cite{Gi:LNPAM} (see also~\cite{Az:RMJM,CH:ECCIDO}). 
\end{remark}

\section{Equivariant Cohomology and Mixed Equivariant Homology}\label{sec:group}
In this section, we recall the main background on group homology and equivariant cohomology that is needed in this paper. We also review the mixed equivariant homology of~\cite{Po:CRAS5}. For a more comprehensive accounts on group homology and equivariant cohomology we refer to the monographs~\cite{Br:CG, GS:Springer99}. 

Throughout this section we let $\Gamma$ be a group. By a $\Gamma$-space we shall mean a vector space equipped with a left action of $\Gamma$, i.e., a left 
$\C\Gamma$-module.

\subsection{Group homology and group cohomology}
The standard cyclic $\Gamma$-space  (or $\C\Gamma$-module) of $\Gamma$ is $C(\Gamma)=(C_\bt(\Gamma), d,s,t)$, where $C_m(\Gamma):=\C\Gamma^{m+1}$, $m \geq 0$, and the structural operators $d:C_\bt(\Gamma)\rightarrow C_{\bt-1}(\Gamma)$, $s:C_\bt(\Gamma)\rightarrow C_{\bt+1}(\Gamma)$, $t:C_\bt(\Gamma)\rightarrow C_{\bt}(\Gamma)$ are given by
\begin{gather}
 d(\psi_0,\ldots, \psi_m)= (\psi_0,\ldots, \psi_{m-1}), \qquad s(\psi_0,\ldots, \psi_m)= (\psi_m,\psi_0,\ldots, \psi_m),
 \label{eq:equivariant.cyclic-G-space-group-end-face} \\
 t(\psi_0,\ldots, \psi_m)= (\psi_{m},\psi_0,\ldots, \psi_{m-1}), \quad \psi_j\in \Gamma. 
\end{gather}
Each space $C_m(\Gamma)$ is equipped with the $\Gamma$-action such that
 $\psi \cdot (\psi_0,\ldots, \psi_m)= (\psi\psi_0,\ldots,\psi \psi_m)$ for all $\psi, \psi_0, \ldots, \psi_m$ in $\Gamma$. The operators $(d,s,t)$ are $\Gamma$-equivariant with respect to this action so that we get a cyclic $\Gamma$-space. 
 At the simplicial level we obtain the standard simplicial $\Gamma$-space of $\Gamma$ whose geometric realization is the universal $\Gamma$-bundle $E\Gamma$ (a.k.a.\ Milnor construction). This gives rise to a chain complex $(C_\bt(\Gamma), \partial)$, where the differential $\partial: C_\bt(\Gamma)\rightarrow C_{\bt-1}(\Gamma)$ is given by
\begin{equation}
\partial(\psi_0,\ldots, \psi_m) = \sum_{0\leq j \leq m} (-1)^j (\psi_0,\ldots,\hat{\psi}_{j},\ldots,\psi_{m}), \qquad \psi_{j}\in \Gamma.
\label{eq:equivariant.group-boundary}
\end{equation}

In what follows, given $\Gamma$-spaces $\sE_1$ and $\sE_2$ we equip the tensor product $\sE_1\otimes \sE_2$ (over $\C$) with the diagonal action of $\Gamma$. We then denote by $\sE_1\otimes_\Gamma \sE_2$ their tensor product over $\Gamma$, i.e., the quotient  of $\sE_1\otimes \sE_2$ by the action of $\Gamma$.  This ensures us that $(\psi\xi_1)\otimes_\Gamma (\psi\xi_2)= \xi_1\otimes_\Gamma \xi_2$ for all $\xi_j \in \Gamma$ and $\psi \in \Gamma$. Bearing this in mind, for any $\Gamma$-space, we form the chain complex $C(\Gamma, \sE):=(C_\bt(\Gamma, \sE), \partial)$, where $C_m(\Gamma, \sE)=C_m(\Gamma)\otimes_\Gamma \sE$. The homology of this complex is the \emph{homology of $\Gamma$ with coefficients in $\sE$} and is denoted by $H_\bt(\Gamma, \sE)$. In the case of the trivial $\Gamma$-space $\sE=\C$ we recover the homology of the classifying space $B\Gamma=E\Gamma\slash \Gamma$ (with coefficients in $\C$). 

There is a natural duality between $C_m(\Gamma,\sE)$ and the space $C^m(\Gamma,\sE)$ that consists of all $\Gamma$-equivariant maps $u:\Gamma^{m+1}\rightarrow \sE$. We thus obtain a cochain complex $(C^\bt(\Gamma, \sE),\partial)$, where $\partial: C^\bt(\Gamma, \sE)\rightarrow C^{\bt+1}(\Gamma, \sE)$ is given by
\begin{equation*}
( \partial u)(\psi_0,\ldots, \psi_{m+1}) = \sum_{0\leq j \leq m+1} (-1)^j u(\psi_0,\ldots,\hat{\psi}_{j},\ldots,\psi_{m+1}), \qquad \psi_{j}\in \Gamma. 
\end{equation*}
The cohomology of this cochain complex is the \emph{cohomology of $\Gamma$ with coefficients in $\sE$} and is denoted by $H^\bt(\Gamma, \sE)$. When $\sE=\C$ we recover the cohomology of $B\Gamma$. 
 
 We have a natural cup product $\smile: C^p(\Gamma, \sE_1)\times C^q(\Gamma, \sE_2)\rightarrow C^{p+q}(\Gamma, \sE_1\otimes \sE_2)$ given by 
\begin{equation*}
 (u\smile v) (\psi_0,\ldots, \psi_{p+q}) = u(\psi_{q},\ldots, \psi_{p+q})\otimes  v(\psi_{0},\ldots, \psi_{q}), \qquad \psi_j \in \Gamma.  
\end{equation*}
 This cup product is associative and compatible with the group coboundary $\partial$. Therefore, it induces a (graded associative) cup product on group cohomology. In particular, this turns $H^\bt(\Gamma, \C)$ into a graded ring. By duality we get a cap product $\frown:  C^p(\Gamma, \sE_1)\times C_{p+q}(\Gamma, \sE_2)\rightarrow C_q(\Gamma, \sE_1\otimes \sE_2)$ given by
\begin{equation}
 u \frown \left[(\psi_0,\ldots, \psi_{p+q})\otimes_\Gamma \xi\right] = (\psi_p,\ldots, \psi_{p+q})\otimes_\Gamma \left[ u(\psi_0, \ldots, \psi_p) \otimes \xi\right], \quad \psi_j\in \Gamma, \ \xi\in \sE_2. 
 \label{eq:equivariant.cap-product}
\end{equation}
This cap product is compatible with the group boundary and coboundary, as well as with the cup product. Therefore, it descends to a cap product between group cohomology and group homology which is compatible with the cup product. In particular, when $\sE_1=\C$ we obtain a graded (left) action of the cohomology ring $H^\bt(\Gamma,\C)$ on the group homology $H_\bt(\Gamma, \sE_2)$. 

\subsection{Equivariant cohomology}
From now on we assume that $\Gamma$ acts on a manifold $M$ by (smooth) diffeomorphisms. The equivariant cohomology of $M$ is the cohomology of the homotopy quotient $E\Gamma\times_\Gamma M$ (a.k.a.\ Borel construction).  This cohomology can be defined by using the bicomplex of Bott~\cite{Bott:LNM652} as follows. This bicomplex is  $C_\Gamma(M):=(C^{\bt,\bt}_\Gamma(M), \partial, d)$, where $C_\Gamma^{p,q}(M)=C^p(\Gamma, \Omega^q(M))$ and we regard the de Rham differential as an operator $d: C^{\bt,\bt}_\Gamma(M) \rightarrow C^{\bt,\bt+1}_\Gamma(M)$. It gives rise to the total complex $\Tot(C_\Gamma(M))=(\Tot^\bt(C_\Gamma(M)),d^\dagger)$, where 
$\Tot^m(C_\Gamma(M))=\bigoplus_{p+q=m} C^p(\Gamma, \Omega^q(M))$ and $d^\dagger=\partial +(-1)^p d$ on $C^p(\Gamma, \Omega^q(M))$. The \emph{equivariant cohomology} of the $\Gamma$-manifold $M$ is the cohomology of the cochain complex $\Tot(C_\Gamma(M))$. It is denoted by $H^\bt_\Gamma(M)$. 

We also define the \emph{even/odd equivariant cohomology} $H^{\ev/\odd}_\Gamma(M)$ as the cohomology of the cochain complex $(C^{\ev/\odd}_\Gamma(M), d^\dagger)$, where $C^{\ev/\odd}_\Gamma(M)=\prod_{\text{$p+q$ even/odd}} C^p(\Gamma, \Omega^q(M))$. This is a natural receptacle for the construction of equivariant characteristic classes. In particular, given any equivariant vector bundle $E$ over $M$, we have a uniquely defined equivariant Chern character $\Ch_\Gamma(E)\in H^\ev_\Gamma(M)$. Following Getzler~\cite{Ge:AM}, given any connection $\nabla^E$ on $E$, we can represent this equivariant Chern character by an explicit equivariant cochain $\Ch_\Gamma (\nabla^E)\in C^\ev_\Gamma(M)$ as follows. 

As $E$ is a  $\Gamma$-equivariant bundle, any $\psi\in \Gamma$ lifts to a vector bundle isomorphism 
$\psi^{E}$ on $E$. In particular, this enables us to pushforward $\nabla_{0}^{E}$ to the connection $\psi_{*}\nabla^{E}:=\left( \psi_{*}\otimes \psi^{E}_{*}\right)\circ \nabla^{E}\circ \left(\psi^{E}_{*}\right)^{-1}$, where $\psi_{*}$ and $\psi^{E}_{*}$ are the respective actions of $\psi$ on $C^{\infty}(M,T^{*}M)$ and 
$C^{\infty}(M,E)$. The equivariant Chern form is $\Ch_\Gamma(\nabla^E)=(\Ch^p_\Gamma(\nabla^E))_{p\geq 0}$, where the components $\Ch^p_\Gamma(\nabla^E)\in C^p(\Gamma,\Omega^\bt(M))$ are given by
\begin{gather}
 \Ch_\Gamma^0(\nabla^E)(\psi_0)= \Ch\left( (\psi_0)_*\nabla^E\right), 
 \label{eq:equivariant.Chern-form0}\\
  \Ch_\Gamma^p(\nabla^E)(\psi_0,\ldots, \psi_p)= (-1)^p \CS\left( (\psi_0)_*\nabla^E, \ldots, (\psi_p)_*\nabla^E\right) , \quad \psi_j\in \Gamma.
   \label{eq:equivariant.Chern-form+}
\end{gather}
Here $ \Ch\left( (\psi_0)_*\nabla^E\right)$ is the Chern form of the connection $(\psi_0)_*\nabla^E$ and $\CS\left( (\psi_0)_*\nabla^E, \ldots, (\psi_p)_*\nabla^E\right)$ is the Chern-Simons form of the connections $(\psi_0)_*\nabla^E, \ldots, (\psi_p)_*\nabla^E$ as defined in~\cite{Ge:AM}. Note that $\Ch_\Gamma^p(\nabla^E)$ is an element of  $C^p(\Gamma, \Omega^\ev(M))$ (resp., $C^p(\Gamma, \Omega^\odd(M))$) when $p$ is even (resp., odd).

When $\Gamma$ is a compact connected Lie group, the equivariant cohomology $H^\bt_\Gamma(M)$ can be described equivalently in terms of the Cartan model of equivariant differentiel forms  and the Weil model (see~\cite{Ca:Paris50, GS:Springer99, MQ:Topo86}).

\subsection{Mixed Equivariant Homology}
We also can define the equivariant homology of the $\Gamma$-manifold $M$ by using a chain version of Bott's bicomplex. For our purpose it is more convenient to use the mixed equivariant homology of~\cite{Po:CRAS5}. Following~\cite{Po:CRAS5} the equivariant mixed bicomplex of $M$  is obtained as the tensor product over $\Gamma$ of the mixed complexes $(C_\bt(\Gamma),\partial,0)$ and $(\Omega^\bt(M),0,d)$. Thus, this is  $C(\Gamma, M):=(C_{\bt, \bt}(\Gamma,M), \partial, 0,0,d)$, where $C_{p,q}(\Gamma, M)=C_p(\Gamma, \Omega^q(M))$, $p,q\geq 0$. This gives rise to the ``total" mixed complex $\Tot(C(\Gamma, M))= ( \Tot_\bt(C(\Gamma, M)),\partial, (-1)^pd)$, where $\Tot_m(C(\Gamma, M))=  \bigoplus_{p+q=m} C_p(\Gamma, \Omega^q(M))$. The cyclic homology of this mixed complex is called the \emph{mixed equivariant homology} of the $\Gamma$-manifold $M$ and is denoted by $H^\Gamma_\bt(M)^\sharp$. A mixed equivariant chain in $\Tot_m(C(\Gamma, M))^\natural$ is of the form $\omega =(\omega_{p,q})_{m-(p+q)\in 2\N_0}$ where $\omega_{p,q}\in \C\Gamma^{p+1}\otimes_{\Gamma} \Omega^q(M)$. This gives a cycle iff 
$\partial \omega_{p,q}=0$ when $p+q=m$ and $\partial \omega_{p+1,q}+(-1)^pd \omega_{p,q-1}=0$ when $m-(p+q)\in 2\N$. 

The periodic cyclic homology of the equivariant mixed complex $\Tot(C(\Gamma, M))$ is called the \emph{even/odd mixed equivariant homology} of $M$ and is denoted by $H^\Gamma_{\ev/\odd}(M)^\sharp$. Note that the periodic cyclic complex is $(C_{\ev/\odd}(\Gamma, M), \partial+(-1)^pd)$, where $C_{\ev/\odd}(\Gamma, M)=\prod_{\text{$p+q$ even/odd}} C_p(\Gamma, \Omega^q(M))$. 

The mixed equivariant homology is the natural receptacle for the cap product between equivariant cohomology and group homology. More precisely, the cap product~(\ref{eq:equivariant.cap-product}) gives a cap product $\frown: C^p(\Gamma, \Omega^q(M)) \times C_m(\Gamma, \C) \rightarrow C_{m-p}(\Gamma, \Omega^q(M))$. This is a differential bilinear map which is compatible with the group differentials and the de Rham differential. Therefore, it gives rise to a (graded) cap product, 
\begin{equation}
 \frown: H^{\ev/\odd}_\Gamma(M) \times H_{\bt}(\Gamma, \C) \longrightarrow H^\Gamma_{\ev/\odd}(M)^\sharp. 
 \label{eq:equivariant.cap-product-mixed}
\end{equation}
In particular, the cap products of equivariant cohomology classes with group homology classes naturally give rise to classes in $H^\Gamma_{\ev/\odd}(M)^\sharp$. 

\section{Cyclic Homology and Group Actions on Manifolds}\label{sec:cyclic-homology-crossed-product}
There is a large amount of work on the cyclic homology of crossed-product algebras, especially in the case of group actions on manifolds or varieties (see, e.g., \cite{BC:CCDG, Br:AIF87, Br:Preprint87, BN:KT94, BDN:AIM17, Co:Kyoto83, Co:NCG, Cr:KT99, FT:LNM87, GJ:Crelle93, NPPT:Crelle06, Ni:InvM90, Po:CRAS4, Po:CRAS5}). In his seminal work on cyclic cohomology and foliations Connes~\cite{Co:Kyoto83} constructed an explicit cochain map and quasi-isomorphism from equivariant cohomology into the homogeneous component of the periodic cyclic cohomology. However, since then it was not until the recent notes~\cite{Po:CRAS4, Po:CRAS5} that further explicit quasi-isomorphisms were exhibited.  In this section, we briefly describe the results of~\cite{Po:CRAS4, Po:CRAS5} in the case of group actions on manifolds. 

Throughout this section we let $\Gamma$ be a group acting by smooth diffeomorphisms on a closed manifold $M^n$. We let $\cA$ be the Fr\'echet algebra $C^\infty(M)$. The action of $\Gamma$ gives rise to an action on $\cA$ given by~(\ref{eq:ConformalCCC.action-on-functions}). We also denote by $\cA_\Gamma$ the crossed-product algebra $C^\infty(M)\rtimes \Gamma$, which we endow with the locally convex algebra topology described in Section~\ref{sec:Conformal-CC-character}.

\subsection{Preamble 1: Paracyclic and cylindrical modules} In the following two preambles we let $\K$ be an arbitrary unital ring.  
In the context of group actions on algebras, we are naturally lead to go beyond the scope of the mixed complexes and cyclic modules described in Section~\ref{sec:CyclicCohomChernChar}. According to Getzler-Jones~\cite{GJ:Crelle93}, a \emph{parachain complex} is given by the datum of $(C_\bt,b,B)$, where $C_m$, $m\geq 0$, are $\K$-modules and $b:C_\bt \rightarrow C_{\bt-1}$ and $B:C_\bt\rightarrow C_{\bt+1}$ are $\K$-module maps such that $b^2=B^2=0$ and $bB+Bb=1-T$, where $T:C_\bt\rightarrow C_\bt$ is some invertible $\K$-module map. Given a parachain complex $C=(C_\bt,b,B)$  we also can 
form a cyclic complex $C^\natural=(C^\natural_\bt, b+SB)$ as in the case of mixed complexes. 
This need not be a chain complex (unless $C$ is a mixed complex), but we obtain a para-$S$-module (\emph{cf}.\ Preamble~2). 

Following~\cite{GJ:Crelle93}, a \emph{paracyclic $\K$-module} is given by the datum of $(C_\bt, d,s,t)$, where $C_m$, $m\geq 0$, are $\K$-modules and the $\K$-module maps $d:C_\bt\rightarrow C_{\bt-1}$, $s:C_\bt\rightarrow C_{\bt+1}$ and $t:C_\bt \rightarrow C_\bt$ define a simplicial module structure in the same way as cyclic modules, but the cyclicity of the operator $t$ is replaced by the relation $t=ds$ and the requirement that $t$ is invertible. Any paracyclic $\K$-module gives rise to a parachain complex $(C_\bt, b,B)$, where $b$ is defined in the same way as with cyclic $\K$-modules and $B=(1-\tau)s'N$ with $s'=sb's$ and $b'=\sum_{j=0}^{m-1} (-1)^j d_j=b-d$ on $C_m$. In that case $T=1-(bB+Bb)=t^{m+1}$ on $C_m$. (When $C$ is a cyclic module we obtain a mixed complex which is isomorphic to its usual mixed complex.) 

A \emph{parachain bicomplex} is given by the datum of $(C_{\bt,\bt}, \wb, \wB, b,B)$, where $C_{p,q}$, $p,q\geq 0$, are $\K$-modules, $(C_{\bt,q},\wb, \wB)$ and $(C_{p,\bt},b,B)$ are parachain complexes for all $p,q\geq 0$, and the horizontal differentials $(\wb, \wB)$ both commute with each of the vertical differentials $(b,B)$.  
We have a \emph{cylindrical complex} when the operators $\wT=1-(\wb\wB+\wB\wb)$ and $T=1-(bB+Bb)$ are inverses of each other 
(\emph{cf}~\cite{GJ:Crelle93}). 
As observed in~\cite{GJ:Crelle93}, any cylindrical complex gives rise to a total mixed complex $\Tot(C)=(\Tot_\bt(C),b^\dagger, B^\dagger)$, where $\Tot_m(C)=\bigoplus_{p+q=m} C_{p,q}$ and $b^\dagger=\wb +(-1)^pb$ and $B^\dagger =\wB+(-1)^p \wT B$ on $C_{p,q}$.

A \emph{bi-paracyclic $\K$-module} is given by the datum of $(C_{\bt,\bt}, \wbd, \ws, \wt, d,s,t)$, where $C_{p,q}$, $p,q\geq 0$, are $\K$-modules, $(C_{\bt,q},\wbd, \ws, \wt)$ and $(C_{p,\bt},d,s,t)$ are paracyclic $\K$-modules for all $p,q\geq 0$, and all the horizontal operators $(\wbd, \ws, \wt)$ commute with each of the vertical operators $(d,s,t)$. 
We have a \emph{cylindrical $\K$-module} when $\wt^{p+1}t^{q+1}=1$ on $C_{p,q}$. 
Any cyclindrical $\K$-module $C=(C_{\bt,\bt}, \wbd, \ws, \wt, d,s,t)$ gives rise to a cyclindrical complex
$(C_{\bt,\bt}, \wb, \wB, b,B)$, where $(\wb, \wB)$ (resp., $(b,B)$) are defined as above by using $(\wbd, \ws, \wt)$ (resp., $(d,s,t)$). This gives rise to a total mixed complex $\Tot(C)$. We also obtain a diagonal cyclic $\K$-module $\Diag(C)=(\Diag_\bt(C), \wbd d, \ws s, \wt t)$, with $\Diag_m(C)=C_{m,m}$, $m\geq 0$. 

\subsection{Preamble~2: Triangular $S$-modules}
The $S$-modules of Jones-Kassel~\cite{JK:KT89,Ka:Crelle90} encapsulate various approaches to cyclic homology. More generally, according to~\cite{Po:CRAS4} a \emph{para-$S$-module} is given by the datum of $(C_\bt, b, S)$, where $C_m$, $m\geq 0$, are $\K$-modules and $d:C_\bt \rightarrow C_{\bt-1}$ and $S:C_\bt \rightarrow C_{\bt -2}$ are $\K$-module maps commuting with each other such that $d^2=(1-T)S$ where $T:C_\bt\rightarrow C_\bt$ is some $\K$-module map commuting with both $d$ and $S$. When $d^2=0$ we obtain an $S$-module. For instance, if $C=(C_\bt, b,B)$ is a parachain complex, then we can define its cyclic complex of the para-$S$-module $C^\natural=(C^\natural_\bt, b+SB, S)$.

According to~\cite{Po:CRAS4}, a (left) \emph{triangular para-$S$-module} is given by the datum of $(C_{\bt,\bt},d, b,B, S)$, where $C_{p,q}$, $p,q\geq 0$, are $\K$-modules, 
 $(C_{\bt,q}, d, S)$ is a para-$S$-module and $(C_{p,\bt}, b,B)$ is a parachain complex for all $p,q\geq 0$,  
 the horizontal operators $(d,S)$ both commute with each of the vertical differentials $(b,B)$. We say that we have a \emph{(left) triangular $S$-module} when  $d^2+S(bB+Bb)=0$. Any cylindrical complex gives rise to a triangular $S$-module (\emph{cf}.~\cite{Po:CRAS4}). Triangular (para-)$S$-modules provide us with a natural framework for defining the tensor product of (para-)$S$-modules with mixed and parachain complexes (see Subsection~\ref{subsec:HCrossed.infinite-order} below). 
  
 Any triangular para-$S$-module $C=(C_{\bt,\bt},\wbd, b,B, S)$ gives rise to a \emph{total para-$S$-module} $\Tot(C)=(\Tot_\bt(C), d^\dagger, S)$, where $\Tot_m(C)= \bigoplus_{p+q=m} C_{p,q}$ and $d^\dagger = d + (-1)^p (b+SB)$ on $C_{p,q}$. When $C$ is a triangular $S$-module the condition $d^2+S(bB+Bb)=0$ precisely ensures us that $(d^\dagger)^2=0$, and so we actually obtain an $S$-module.  Furthermore, the filtration by columns of $\Tot_\bt(C)$ is a filtration of chain complexes, and so it gives rise to a spectral sequence converging to $H_{\bt}(\Tot(C))$. The spectral sequence of Getzler-Jones~\cite{GJ:Crelle93} is an instance of such a spectral sequence (\emph{cf}.~\cite{Po:CRAS4}). 

\subsection{The cylindrical space $\bC^\phi(\Gamma, \cA)$} As we shall see, a central role in the computation of the cyclic homology of $\cA_\Gamma$ is played by some cylindrical spaces $\bC^\phi(\Gamma,\cA)$ that are defined as follows. 

Suppose that $\phi$ is a \emph{central} element of $\Gamma$. We can use $\phi$ to twist the cyclic structure of the cyclic $\Gamma$-space $C(\Gamma)$. We obtain the paracyclic $\Gamma$-space $C^\phi(\Gamma)=(C_\bt(\Gamma),d,s_\phi, t_\phi)$, where $d$ is as in~(\ref{eq:equivariant.cyclic-G-space-group-end-face}) 
and the operators $s_\phi:C_\bt(\Gamma)\rightarrow C_{\bt+1}(\Gamma)$ and $t_\phi:C_\bt(\Gamma)\rightarrow C_{\bt}(\Gamma)$ are given by
\begin{align*}
 s_\phi(\psi_0,\ldots,\psi_m)&= (\phi^{-1}\psi_m,\psi_0,\ldots, \psi_m),\\  
 t_\phi(\psi_0,\ldots,\psi_m)&= (\phi^{-1}\psi_m,\psi_0,\ldots, \psi_{m-1}), \quad \psi_j\in \Gamma. 
\end{align*}
The simplicial space structures of $C^\phi(\Gamma)$ and $C(\Gamma)$ agree, and so the $b$-differential of $C^\phi(\Gamma)$ is the group differential $\partial$ given by~(\ref{eq:equivariant.group-boundary}). We also observe that $t_\phi^{m+1}(\psi_0,\ldots,\psi_m)=(\phi^{-1}\psi_0,\ldots,\phi^{-1}\psi_m)$, i.e., $t_\phi^{m+1}$ agrees with the action of $\phi^{-1}$ on $C_{m}(\Gamma)$.

We also can twist the cyclic structure of the cyclic $\Gamma$-space $\bC(\cA)$. We get the paracyclic $\Gamma$-space $\bC^\phi(\cA)=(\bC_\bt(\cA), d_\phi, s,t_\phi)$, 
where $s$ is given by~(\ref{eq:cyclic.extra-degeneracy-chain}) and the operators $d_\phi: \bC_\bt(\cA)\rightarrow \bC_{\bt-1}(\cA)$ and $t_\phi: \bC_\bt(\cA)\rightarrow \bC_\bt(\cA)$ are given by
\begin{align*}
 d_\phi(f^0\otimes \cdots \otimes f^m)&=[(\phi^{-1}\cdot f^m)f^0]\otimes f^1\otimes \cdots \otimes f^{m-1},\\
 t_\phi(f^0\otimes \cdots \otimes f^m)&=(\phi^{-1}\cdot f^m)\otimes f^0\otimes \cdots \otimes f^{m-1}, \quad f^j\in \cA. 
\end{align*}
Note also that $t^{m+1}_\phi$ agrees with the action of $\phi^{-1}$ on $\bC_m(\cA_\Gamma)$. 

The cylindrical space $\bC^\phi(\Gamma, \cA)$ is the tensor product over $\Gamma$ of the paracyclic $\Gamma$-spaces $C^\phi(\Gamma)$ and $\bC^\phi(\cA)$. Namely, $\bC^\phi(\Gamma, \cA)=(\bC_{\bt, \bt}(\Gamma, \cA), \wbd, \ws_\phi, \wt_\phi, d_\phi, s, t_\phi)$, where $\bC_{p, q}(\Gamma, \cA)= C_p(\Gamma)\otimes_\Gamma \bC_q(\cA)$, and the horizontal (resp., vertical) structural operators $(\wbd, \ws_\phi, \wt_\phi)$ (resp., $(d_\phi, s, t_\phi)$) are given by the structural operators of $C^\phi(\Gamma)$ (resp., $\bC^\phi(\cA)$). We obtain a cyclindrical space because in degree $m$ the operator $\wt_\phi^{m+1}t_\phi^{m+1}$ arises from the diagonal action of $\phi^{-1}$ on $C_m(\Gamma)\otimes \bC_m(\cA)$, and hence it agrees with the identity map on $C_m(\Gamma)\otimes_\Gamma \bC_m(\cA)$. This gives rise to a diagonal cyclic space $\Diag( \bC^\phi(\Gamma, \cA))$ and a total mixed complex 
$\Tot( \bC^\phi(\Gamma, \cA))$. By the Eilenberg-Zilber theorem for bi-paracyclic spaces~\cite{KR:CMB04, Po:CRAS3, ZW:CMB14}, at the level of their cyclic complexes we have an explicit chain homotopy equivalence 
$\Tot_\bt\left(\bC^\phi(\Gamma_\phi,\cA)\right)^\natural \xrightleftharpoons[\AW^\natural]{\shuffle^\natural} \Diag_\bt\left(\bC^\phi(\Gamma_\phi,\cA)\right)^\natural$, where $\shuffle^\natural$ and $\AW^\natural$ are $S$-maps whose zeroth degree components are the usual shuffle and Alexander-Whitney maps. At the periodic level we also have a chain homotopy equivalence $\Tot_\bt\left(\bC^\phi(\Gamma_\phi,\cA)\right)^\sharp \xrightleftharpoons[\AW^\sharp]{\shuffle^\sharp} \Diag_\bt\left(\bC^\phi(\Gamma_\phi,\cA)\right)^\sharp$, where $\shuffle^\sharp$ and $\AW^\sharp$ are the respective extensions of $\shuffle^\natural$ and $\AW^\natural$ to periodic chains. 

\subsection{Splitting along conjugacy classes}  Given any $\phi\in \Gamma$, we denote by $[\phi]$ its conjugacy class. In addition, we let $[\Gamma]$ be the set of conjugacy classes of $\Gamma$. 
Given any $\phi \in \Gamma$ and $m\geq 0$,  we let $\bC_m(\cA_\Gamma)_{[\phi]}$ be the closure in $\bC(\cA_\Gamma)$ of the subspace spanned by elementary tensor products $f^{0}u_{\phi_{0}}\otimes \cdots \otimes f^{m}u_{\phi_{m}}$, with $f^0,...,f^m$ in $\cA$ and $\phi_0,\ldots, \phi_m$ in $\Gamma$ such that $\phi_0 \cdots \phi_m\in [\phi]$. The subspace $\bC_\bt(\cA_\Gamma)_{[\phi]}$ is preserved by the structural operators $(d,s,t)$ in~(\ref{eq:cyclic.end-face-chain})--(\ref{eq:cyclic.cyclic-operator-chain}). 
We thus obtain a cyclic subspace $\bC(\cA_\Gamma)_{[\phi]}= (\bC_\bt(\cA_\Gamma)_{[\phi]}, d,s,t)$. We then have a direct-sum decomposition of cyclic spaces, 
\begin{equation}
     \bC_\bt(\cA_{\Gamma})=\bigoplus_{[\phi]\in [\Gamma]} \bC_\bt(\cA_{\Gamma})_{[\phi]},
     \label{eq:HCrossed.splitting-conjugacy-cyclic-space}
\end{equation}
where the summation goes over all conjugacy classes. 

In what follows, given $\phi \in \Gamma$, we denote by $\bC^\natural(\cA_\Gamma)_{[\phi]}$ and $\bC^\sharp(\cA_\Gamma)_{[\phi]}$ the respective cyclic and periodic complexes of the cyclic space $\bC_\bt(\cA_\Gamma)_{[\phi]}$. We also denote by $\bHC_\bt(\cA_\Gamma)_{[\phi]}$ and  $\bHP_\bt(\cA_\Gamma)_{[\phi]}$ their respective homologies. At the cyclic level the decomposition~(\ref{eq:HCrossed.splitting-conjugacy-cyclic-space}) gives rise to direct-sum decompositions, 
\begin{equation*}
  \bC_\bt^\natural(\cA_{\Gamma})=\bigoplus_{[\phi]\in [\Gamma]} \bC^\natural_\bt(\cA_\Gamma)_{[\phi]} \quad \text{and} \quad  \bHC_\bt(\cA_\Gamma)=\bigoplus_{[\phi]\in [\Gamma]}  \bHC_\bt(\cA_\Gamma)_{[\phi]}. 
\end{equation*}
At the periodic level $\bC^\sharp(\cA_\Gamma)_{[\phi]}$  and  $\bHP_\bt(\cA_\Gamma)_{[\phi]}$ are direct summands of  $\bC^\sharp(\cA_\Gamma)$  and  $\bHP_\bt(\cA_\Gamma)$, respectively. Therefore, we have direct-summand inclusions, 
\begin{equation*}
\bigoplus_{[\phi]\in [\Gamma]} \bC^\sharp_\bt(\cA_\Gamma)_{[\phi]} \subset   \bC_\bt^\sharp(\cA_{\Gamma}) \quad \text{and} \quad 
\bigoplus_{[\phi]\in [\Gamma]}  \bHP_\bt(\cA_\Gamma)_{[\phi]} \subset 
 \bHP_\bt(\cA_\Gamma). 
\end{equation*}
These inclusions are onto when there is a finite number of conjugacy classes. 

Given $\phi \in \Gamma$, let us denote by $\Gamma_\phi$ its centralizer in $\Gamma$. As $\phi$ is a central element of $\Gamma_\phi$, we may form the cylindrical complex $\bC^\phi(\Gamma_\phi,\cA)$ as above. We have a natural embedding of cyclic spaces
$\mu_\phi: \Diag_\bt(\bC^\phi(\Gamma_\phi,\cA))\rightarrow \bC_\bt(\cA_\Gamma)_{[\phi]}$ given by 
  \begin{multline*}
 \mu_{\phi}\left((\psi_{0},\ldots,\psi_{m})\otimes_{\Gamma_\phi}(f^{0}\otimes \cdots \otimes f^{m})\right)= \\ 
    [(\psi_{m}^{-1}\phi)\cdot f^{0}]u_{\phi\psi_{m}^{-1} \psi_{0}}\otimes 
     ( \psi_{0}^{-1}\cdot f^{1})u_{\psi_{0}^{-1}\psi_{1}}\otimes \cdots \otimes ( \psi_{m-1}^{-1}\cdot f^{m})u_{\psi_{m-1}^{-1}\psi_{m}}. 
 \end{multline*}
 This  cyclic space embedding gives rise to quasi-isomorphisms at the level the cyclic and periodic complexes.  We thus obtain 
 explicit quasi-isomorphisms,   
 \begin{equation}
\Tot_\bt\left(\bC^\phi(\Gamma_\phi,\cA)\right)^\natural\xrightleftharpoons[\AW^\natural]{\shuffle^\natural} \Diag_\bt\left(\bC^\phi(\Gamma_\phi,\cA)\right)^\natural \xrightarrow{\mu_\phi} \bC_\bt(\cA_\Gamma)_{[\phi]}^\natural. 
\label{eq:quasi-isomorphism-CphiGA-CAGphi}
\end{equation}
There are analogous quasi-isomorphisms between the corresponding periodic cyclic complexes. 

All this shows that the computation of the cyclic homology and periodic cyclic homology of $\cA_{\Gamma}$ reduces to that of the mixed complexes  $\Tot_\bt(\bC^\phi(\Gamma_\phi,\cA))$, $\phi \in \Gamma$. 

\subsection{A twisted Connes-Hochschild-Kostant-Rosenberg Theorem} 
Let $\phi\in \Gamma$ and denote by $M^\phi$ its fixed-point set in $M$. We shall say that the action of $\phi$ on $M$ is \emph{clean} when, 
 for every $x_0\in M^\phi$,  the following two conditions are satisfied:
\begin{enumerate}
\item[(i)] The fixed-point set $M^\phi$ is a submanifold of $M$ near $x_0$. 

\item[(ii)] We have $T_{x_0}M^\phi = \ker (\phi'(x_0)-1)$ and $T_{x_0}M = T_{x_0}M^\phi \oplus \op{ran} (\phi'(x_0)-1)$.  
\end{enumerate}
We have a clean action whenever $\phi$ preserves an affine connection or even a Riemannian metric. In particular, we always have a clean action when $\phi$ has finite order. 

From now on we assume that the action of $\phi$ on $M$ is clean. For $a=0,1,\ldots, \dim M$, set $M^\phi_a:=\{x\in M^\phi; \ \dim \ker (\phi'(x)-1)=a\}$. Each subset $M^\phi_a$ is a submanifold of $M$, and so we have a disjoint union stratification  $M^\phi=\bigsqcup_a M^\phi_a$. This enables us to define the de Rham complex $\Omega(M^\phi)= (\Omega^\bt(M^\phi),d)$ as the direct sum of the de Rham complexes $\Omega(M^\phi_a)$. Note also that each component $M_a^\phi$ is preserved by the action of the centralizer $\Gamma_\phi$. Let $(\bC^\phi(\cA), b_\phi, B_\phi)$ be the parachain complex of the paracyclic $\Gamma$-space. Regarding $\Omega(M^\phi)$ as a mixed complex $(\Omega^\bt(M^\phi),0,d)$ we then have a map of parachain complexes $\alpha^{\phi}:\bC^\phi_\bt(\cA) \rightarrow \Omega^{\bt}(M^\phi)$ given by 
\begin{equation}
\alpha^\phi(f^0\otimes \cdots \otimes f^m)=\frac{1}{m} \sum_a (f^0df^1\wedge \cdots \wedge df^m)_{|M^\phi_a}, \qquad f^j \in \cA.
\label{eq:HKR-map}
\end{equation}

\begin{proposition}[{{\rm Brylinski~\cite{Br:AIF87}, Brylinski-Nistor~\cite{BN:KT94}}}]\label{prop:twisted-CHKR}
 Let $\phi \in \Gamma$ acts cleanly on $M$. Then the chain map $\alpha^{\phi}:(\bC^\phi_\bt(\cA),b_\phi) \rightarrow (\Omega^{\bt}(M^\phi),0)$ is a quasi-isomorphism. 
\end{proposition}

\subsection{The Cyclic space $\bC(\cA_\Gamma)_{[\phi]}$ when $\phi$ has finite order} Let $\phi\in \Gamma$ have finite order $r$. Let $C^\flat(\Gamma_\phi)$ be the mixed complex $(C_\bt(\Gamma_\phi),\partial,0)$. We have a $\Gamma_\phi$-equivariant map of parachain complexes 
$(\varepsilon \nu_\phi): C^\phi_\bt(\Gamma_\phi) \rightarrow C^\flat_\bt(\Gamma_\phi)$, where $\nu_\phi: C^\phi_\bt(\Gamma_\phi) \rightarrow C_\bt(\Gamma_\phi)$ and $\varepsilon: C_\bt(\Gamma_\phi) \rightarrow C^\flat_\bt(\Gamma_\phi)$ are projections given by 
\begin{gather}
 \nu_\phi(\psi_0,\ldots, \psi_m)= \frac{1}{r^{m+1}}\sum_{0\leq j \leq m} \sum_{0\leq \ell_j \leq r-1}(\phi^{\ell_0}\psi_0,\ldots, \phi^{\ell_m}\psi_m),\\
 \varepsilon(\psi_0,\ldots, \psi_m) = \frac{1}{(m+1)!} \sum_{\sigma \in \fS_m} (\psi_{\sigma^{-1}(0)}, \ldots, \psi_{\sigma^{-1}(m)}), \quad \psi_j\in \Gamma_\phi,
 \label{eq:HCrossed.antisymmetrization}
\end{gather}
where $\fS_m$ is the group of permutations of $\{0,\ldots, m\}$. This gives rise to explicit chain homotopy equivalences between the cyclic complexes of $C^\phi(\Gamma_\phi)$ and $C^\flat(\Gamma_\phi)$, as well as between their periodic cyclic complexes (\cite{Bu:CMH85, Ma:BCP86, Po:CRAS4}). 

As $\phi$ has finite order, its action on $M$ is clean. We define the equivariant bicomplex $C_{\Gamma_\phi}(M^\phi)$ and the equivariant mixed bicomplex $C(\Gamma_\phi,M^\phi)$ as the direct sums of the bicomplexes $C_{\Gamma_\phi}(M^\phi_a)$ and $C(\Gamma_\phi,M^\phi_a)$, respectively. This enables us to define the equivariant cohomology $H^\bt_{\Gamma_\phi}(M^\phi)$ and mixed equivariant homology $H_\bt^{\Gamma_\phi}(M^\phi)^\natural$ of $M^\phi$. We observe that the 
mixed bicomplex $C(\Gamma_\phi,M^\phi)$ is the tensor product over $\Gamma_\phi$ of the mixed complexes $C^\flat(\Gamma_\phi)$ and $(\Omega(M^\phi), 0,d)$. Therefore, we obtain a mixed complex map $(\varepsilon \nu_\phi)\otimes \alpha^{\phi}: \Tot_\bt(\bC^\phi(\Gamma_\phi, \cA))\rightarrow \Tot_\bt(C(\Gamma_\phi,M^\phi))$, which is shown to be a quasi-isomorphism~\cite{Po:CRAS5}. Combining this with the quasi-isomorphisms~(\ref{eq:quasi-isomorphism-CphiGA-CAGphi}) we then obtain the following result. 

\begin{theorem}[\cite{Po:CRAS5}] \label{thm:HCcross-finite-order}
Let $\phi\in \Gamma$ have finite order. 
 \begin{enumerate}
\item  The following chain maps are quasi-isomorphisms,
\begin{equation}
 \Tot_\bt\left( C(\Gamma_\phi,M^\phi)\right)^\natural  \xleftarrow[]{(\varepsilon \nu_\phi)\otimes \alpha^\phi}
 \Tot_\bt\left(\bC^\phi(\Gamma_\phi,\cA)\right)^\natural 
 \xrightarrow{\mu_\phi\circ\shuffle^\natural} \bC_\bt(\cA_\Gamma)_{[\phi]}^\natural.
 \label{eq:manifolds.quasi-isom-finite-order}
\end{equation}

\item There are similar quasi-isomorphisms between the respective periodic cyclic complexes. 

\item This gives rise to isomorphisms, 
\begin{equation*}
\bHC_\bt(\cA_\Gamma)_{[\phi]}\simeq H^{\Gamma_\phi}_\bt(M^\phi)^\natural,   \qquad \bHP_\bt(\cA_\Gamma)_{[\phi]}\simeq H^{\Gamma_\phi}_{\ev/\odd}(M^\phi)^\sharp
\end{equation*}
\end{enumerate}
\end{theorem}

\begin{remark}
Brylinski-Nistor~\cite{BN:KT94} expressed  $\bHC_\bt(\cA_\Gamma)_{[\phi]}$ and $\bHP_\bt(\cA_\Gamma)_{[\phi]}$ in terms of the equivariant homology of $M^\phi$. We obtain explicit quasi-isomorphisms with the equivariant homology complex by combining the quasi-isomorphisms~(\ref{eq:manifolds.quasi-isom-finite-order}) with the Poincar\'e duality for the de Rham complex $\Omega(M^\phi)$. In particular, this enables us to recover the results of~\cite{BN:KT94}.
\end{remark}

\begin{remark}
 When $\phi=1$ the cyclic space embedding $\mu_{1}$ actually is an isomorphism and has an explicit inverse $\mu_1^{-1}: \bC_\bt(\cA_\Gamma)_{[1]} \rightarrow
 \Diag_\bt(\bC^\phi(\Gamma_\phi,\cA))$. Therefore, in this case we obtain an explicit quasi-isomorphism, 
 \begin{equation*}
\mu_1^{-1} \circ  \AW^\natural \circ (\varepsilon\otimes \alpha): \bC_\bt(\cA_\Gamma)_{[1]} \longrightarrow \Tot_\bt\left( C(\Gamma,M)\right)^\natural.
\end{equation*}
There is a similar quasi-isomorphism at the periodic cyclic level. By duality this gives rise to a map from the $\Gamma$-equivariant cohomology of $M$ into the cyclic cohomology of $\cA_\Gamma$. This provides us with an alternative to the cochain map of Connes~\cite{Co:Kyoto83, Co:NCG}. 
 \end{remark}
 
\begin{remark}
 When $\Gamma_\phi$ is finite, we have an explicit quasi-isomorphism of mixed complexes from  $\Tot(C(\Gamma_\phi,M^\phi))$ to the $\Gamma_\phi$-invariant de Rham mixed complex $(\Omega^\bt(M^\phi)^{\Gamma_\phi},0,d)$. This allows us to express the cyclic homology and periodic cyclic homology of $\bC(\cA_\Gamma)_{[\phi]}$ in terms of the invariant de Rham cohomology $H^\bt(M^\phi)^{\Gamma_\phi}$. When $\Gamma$ is finite, by using the 
 splitting~(\ref{eq:HCrossed.splitting-conjugacy-cyclic-space}) 
 we obtain an explicit quasi-isomorphisms of mixed complexes between $\bC(\cA_\Gamma)$ and the orbifold de Rham mixed complex
  $(\Omega^\bt(M/\Gamma), 0,d)$, where $\Omega^\bt(M/\Gamma) := \bigoplus_a  \Omega^\bt(M^\phi)^{\Gamma_\phi}$. This allows us to recover the identification by  Baum-Connes~\cite{BC:CCDG} of $\bHP_\bt(\cA_\Gamma)$ with the orbifold cohomology $H^{\ev/\odd}(M/\Gamma)$.  
\end{remark}

Let $\eta^\phi: H^{\Gamma_\phi}_{\ev/\odd}(M^\phi)^\sharp\rightarrow \bHP_\bt(\cA_\Gamma)_{[\phi]}$ be the isomorphism given by Theorem~\ref{thm:HCcross-finite-order}. Composing it with the cap product~(\ref{eq:equivariant.cap-product-mixed}) arrives at the following statement.

\begin{corollary}[\cite{Po:CRAS5}]\label{cor:finite-order.cap-HG-HP}
 Let $\phi \in \Gamma$ have finite order. Then we have a graded bilinear graded map,
 \begin{equation*}
 \eta^\phi(- \frown-) : H^{\ev/\odd}_{\Gamma_\phi}(M^\phi) \times H_{\ev/\odd}(\Gamma_\phi,\C) \longrightarrow \bHP_\bt(\cA_\Gamma)_{[\phi]}. 
\end{equation*}
 In particular, equivariant characteristic classes naturally give rise to classes in $\bHP_\bt(\cA_\Gamma)_{[\phi]}$. 
\end{corollary}

The definition of the isomorphism $\eta^\phi$ involve a cylindrical version of the shuffle map. Although explicit, this map involves numerous terms. As it turns out, we actually obtain a very simple and concise formula when we pair $\eta^\phi$ with cochains arising from equivariant currents as follows.  

Let $\Omega^\Gamma(M)=(\Omega^\Gamma_\bt(M), d)$ be the cochain complex of equivariant currents, where $\Omega^\Gamma_m(M)$, $m\geq 0$, consists of maps $C:\Gamma\rightarrow \Omega_m(M)$ that are $\Gamma$-equivariant in the sense that
\begin{equation*}
C(\psi^{-1} \phi_0 \psi)=\psi_*[C(\phi_0)] \qquad \text{for all  $\phi_0,\psi \in \Gamma$}.
\end{equation*}
Any equivariant current $C\in \Omega^\Gamma_m(M)$ defines a cochain $\varphi_C\in C^m(\cA_\Gamma)$ given by
\begin{equation}
\varphi_C(f^0u_{\psi_0}, \ldots, f^mu_{\psi_m})= \frac{1}{m!} \acou{C(\psi)}{f^0 d\hat{f}^1 \wedge \cdots \wedge d\hat{f}^m},\quad f^j\in \cA, \ \psi_j\in \Gamma,
\label{eq:HCrossed.vphiC}
\end{equation}
where we have set $\psi=\psi_0 \cdots \psi_m$ and $\hat{f}^j= f^j\circ (\psi_{0} \cdots\psi_{j-1})^{-1}$. 
This provides us with a map of mixed complexes  from 
$(\Omega_\bt^\Gamma(M), d,0)$ to $(C^\bt(\cA_\Gamma), B, b)$. Therefore, we obtain cochain maps between their respective cyclic and periodic cochain  complexes. Note that the periodic cyclic complex of $\Omega^\Gamma(M)$ is just $(\Omega^\Gamma_{\ev/\odd}(M), d)$, with $\Omega^\Gamma_{\ev/\odd}(M)= 
\bigoplus_{\text{$m$ even/odd}}\Omega_m^\Gamma(M)$. 

In what follows, given any equivariant chain $\omega \in C^{\ev/\odd}(\Gamma_\phi,M^\phi)$, we denote by  $\omega_0$ its component in $C_{0,\bt}(\Gamma_\phi, M^\phi)\simeq \Omega^\bt(M^\phi)$. In addition, when $\omega$ is a cycle we shall denote by $[\omega]$ its class in $H^{\ev/\odd}_{\Gamma_\phi}(M^\phi)^\sharp$. 

\begin{proposition}[\cite{Po:CRAS5}]\label{prop:manifolds.pairing-equivarian-currents}
Let $\phi \in \Gamma$ have finite order. Then, for any closed equivariant current $C\in \Omega_{\ev/\odd}^\Gamma(M)$ and any mixed equivariant cycle $\omega \in C_{\ev/\odd}(\Gamma_\phi, M^\phi)$, we have \[\acou{\varphi_C}{\eta^\phi([\omega])}= \acou{C(\phi)}{\tilde{\omega}_0},\] where 
$\tilde{\omega}_0\in \Omega^\ev(M)$ is such that $\iota_{M^\phi}^*\tilde{\omega}_0=\omega_0$. 
\end{proposition}

\subsection{The Cyclic space $\bC(\cA_\Gamma)_{[\phi]}$ when $\phi$ has infinite order}\label{subsec:HCrossed.infinite-order}
Let $\phi \in \Gamma$ have infinite order. Set $\overline{\Gamma}_\phi=\Gamma_\phi / \brak\phi$, where $\brak\phi$ is the subgroup generated by $\phi$. 
The canonical projection $\wpi:\Gamma_{\phi} \rightarrow \overline{\Gamma}_{\phi}$ gives rise to a $(\Gamma_{\phi}, \overline{\Gamma}_{\phi})$-equivariant paracyclic $\Gamma_{\phi}$-space map $\wpi: C^\phi_\bt(\Gamma_{\phi})\rightarrow C_\bt(\overline{\Gamma}_{\phi})$. 
Composing this map with the projection $\pi^\natural: C_\bt(\overline{\Gamma}_{\phi})^\natural\rightarrow C_\bt(\overline{\Gamma}_{\phi})$ and the antisymmetrization map $\varepsilon:C_\bt(\overline{\Gamma}_{\phi}) \rightarrow C_\bt(\overline{\Gamma}_{\phi})$ given by~(\ref{eq:HCrossed.antisymmetrization}), we then obtain a $(\Gamma_{\phi}, \overline{\Gamma}_{\phi})$-equivariant chain map $\wepsn: C^\phi_\bt(\Gamma_{\phi})^\natural \rightarrow C_\bt(\overline{\Gamma}_{\phi})$.  In addition, let $u_\phi \in C^2(\overline{\Gamma}_\phi,\Z)$ be a $2$-cocycle representing the Euler class $e_\phi \in H^2(\overline{\Gamma}_\phi,\Z)$ of the Abelian extension $1\rightarrow \brak\phi \rightarrow \overline{\Gamma}_\phi \rightarrow \Gamma_\phi \rightarrow 1$ (\emph{cf}.\ \cite{Br:CG, Ro:IHA}).  The cap product with $u_\phi$ then gives rise to a chain map $u_\phi \frown -: C_{\bt}(\overline{\Gamma}_\phi)\rightarrow C_{\bt-2}(\overline{\Gamma}_\phi)$. Then there is an explicit  chain homotopy $h_\phi: C_\bt^\phi(\Gamma_\phi)^\natural \rightarrow C_{\bt -1}(\overline{\Gamma}_\phi)$ such that $\wepsn S- (u_\phi \frown -) \wepsn = \partial h_\phi + h_\phi(\partial +B_\phi S)$, where $B_\phi$ is the $B$-differential of $C^\phi(\Gamma_\phi)$  (see \cite{Ji:KT95, Po:CRAS4}). 

Assume further that $\phi$ acts cleanly on $M$. As $\phi$ acts trivially on the fixed-point set $M^\phi$, the action of $\Gamma_\phi$ on $\Omega^\bt(M^\phi)$ descends to an action of $\OG_\phi$. Moreover, as $u_\phi \frown -$ is a chain map of degree~$-2$, we get an 
$S$-module $C^\sigma(\OG_\phi)=(C_\bt(\OG_\phi), \partial, u_\phi\frown -)$. The triangular $S$-module $C^\sigma(\OG_\phi,\Omega(M^\phi))$ is obtained as the tensor product over $\OG_\phi$ of $C^\sigma(\OG_\phi)$ with the de Rham mixed complex $(\Omega^\bt(M^\phi),0,d)$. Namely, $C^\sigma(\OG_\phi,\Omega(M^\phi))= (C^\sigma_{\bt,\bt}(\OG_\phi,\Omega(M^\phi)), \partial, 0,d, u_\phi \frown -)$, where $C^\sigma_{p,q}(\OG_\phi,\Omega(M^\phi))= C_p(\OG_\phi) \otimes_{\OG_\phi} \Omega^q(M^\phi)$.  Its total $S$-module is $(\Tot_\bt(C^\sigma(\OG_\phi, \Omega(M^\phi))), d^\dagger, u_\phi \frown -)$, where $\Tot_m(C^\sigma(\OG_\phi, \Omega(M^\phi)))= \bigoplus_{p+q=m} C_p(\OG_\phi) \otimes_{\OG_\phi} \Omega^q(M^\phi)$ and $d^\dagger = \partial + (-1)^p d(u_\phi \frown -)$ on $C_p(\OG_\phi) \otimes_{\OG_\phi} \Omega^q(M^\phi)$. 

We observe that 
\begin{align*}
 \Tot_m\left( \bC^\phi(\Gamma_\phi,\cA)\right)^\natural & =  \Tot_m\left( \bC^\phi(\Gamma_\phi,\cA)\right) \oplus  \Tot_{m-2}\left( \bC^\phi(\Gamma_\phi,\cA)\right)^\natural \cdots \\
 & = \bigoplus_{m-(p+q)\in 2\N_0} C_p(\Gamma_\phi) \otimes_{\Gamma_\phi} \bC_q(\cA) \\
 & = \bigoplus_{p+q=m} \left(C_p(\Gamma_\phi)\oplus C_{p-2}(\Gamma_\phi) \oplus \cdots\right) \otimes_{\Gamma_\phi} \bC_q(\cA)\\
 & = \bigoplus_{p+q=m} C^\phi_p(\Gamma_\phi)^\natural  \otimes_{\Gamma_\phi} \bC_q(\cA). 
\end{align*}
Using this we get a chain map $\theta(1\otimes \alpha^\phi):\Tot_\bt(\bC^\phi(\Gamma_\phi,\cA))^\natural \rightarrow 
\Tot_\bt(C^\sigma(\overline{\Gamma}_\phi,\Omega(M^\phi)))$,  where $\alpha^\phi$ is given by~(\ref{eq:HKR-map}) and $\theta: C^\phi_\bt(\Gamma_\phi)\otimes \Omega^\bt(M^\phi)\rightarrow C_\bt(\OG_\phi)\otimes_{\OG_\phi} \Omega^\bt(M^\phi)$ is  given by
\begin{equation*}
\theta=\wepsn \otimes 1 + (-1)^{p-1} (h_\phi \otimes d)\quad  \text{on}\  C^\phi_p(\Gamma_\phi)\otimes_{\OG_\phi} \Omega^q(M^\phi)).
\end{equation*}
We then have the following result.

\begin{theorem}[\cite{Po:CRAS5}]\label{thm:manifolds.infinite-order}
 Let $\phi \in \Gamma$ have infinite order and act cleanly on $M$.  
 \begin{enumerate}
\item The following chain maps are quasi-isomorphisms,
  \begin{equation}
 \Tot_\bt (C^\sigma(\overline{\Gamma}_\phi ,\Omega(M^\phi)))  \xleftarrow{\theta(1\otimes \alpha)}
 \Tot_\bt\left(\bC^\phi(\Gamma_\phi,\cA)\right)^\natural 
 \xrightarrow{\mu_\phi \circ\shuffle^\natural } C_\bt(\cA_\Gamma)_{[\phi]}^\natural.
 \label{eq:manifolds.infinite-quasi-isom}
\end{equation} 

\item This gives rise to an isomorphism, 
\begin{equation}
 \bHC_\bt(\cA_{\Gamma})_{[\phi]}\simeq H_\bt\left( \Tot(C^\sigma(\overline{\Gamma}_\phi ,\Omega(M^\phi)))\right).
 \label{eq:HC.infinite-order-isomorphism}
\end{equation}

\item Under this isomorphism, the periodicity operator of $\bHC_\bt(\cA_{\Gamma})_{[\phi]}$ is given by the cap product, 
\begin{equation*}
e_\phi \frown - : H_\bt\left(\Tot(C^\sigma(\overline{\Gamma}_\phi ,\Omega(M^\phi)))\right) \longrightarrow H_{\bt-2}\left(\Tot(C^\sigma(\overline{\Gamma}_\phi ,\Omega(M^\phi)))\right),
\end{equation*}
 where $e_\phi \in H^2(\OG_\phi, \Z)$ is the Euler class of the extension $1\rightarrow \brak\phi \rightarrow \overline{\Gamma}_\phi \rightarrow \Gamma_\phi \rightarrow 1$. 
In particular, $\bHP_\bt(\cA_\Gamma)_{[\phi]}=0$ whenever $e_\phi$ is nilpotent in $H^\bt(\OG_\phi, \C)$.  
\end{enumerate}
\end{theorem}

\begin{remark}
 The quasi-isomorphisms~(\ref{eq:manifolds.infinite-quasi-isom}) and the filtration by columns of $ \Tot_\bt (C^\sigma(\overline{\Gamma}_\phi ,\Omega(M^\phi)))$ give rise to a spectral sequence,  
 \begin{equation*}
E^2_{p,q}= H_p(\OG_\phi, \Omega^q(M^\phi)) \Longrightarrow \bHC_{p+q}(\cA_\Gamma)_{[\phi]},
\end{equation*}
  where the $d^2$-differential is  $(-1)^pd(u_\phi\frown -): H_p(\OG_\phi, \Omega^q(M^\phi))\rightarrow H_{p-2}(\OG_\phi, \Omega^{q+1}(M^\phi))$. Crainic~\cite{Cr:KT99} obtained such a spectral sequence. He inferred from this that  $\bHC_m(\cA_\Gamma)_{[\phi]} \simeq \oplus_{p+q=m} H_p(\OG_\phi, \Omega^q(M^\phi))$ (see~\cite[Corollary~4.15]{Cr:KT99}). What we really have is the 
 isomorphism~(\ref{eq:HC.infinite-order-isomorphism}).
 \end{remark}

\begin{remark}
 When $\phi \in \Gamma$ has infinite order and does not act cleanly on $M$, we still have a spectral sequence 
  \begin{equation*}
E^2_{p,q}= H_p(\OG_\phi,H_q^\phi(\cA)) \Longrightarrow \bHC_{p+q}(\cA_\Gamma)_{[\phi]},
\end{equation*}
where $H^\phi_\bt(\cA)$ is the homology of the twisted Hochschild complex $(\bC^\phi_\bt(\cA), b_\phi)$ (see~\cite{Cr:KT99, Po:CRAS5}). In addition, it can be shown that $\bHC_{\bt}(\cA_\Gamma)_{[\phi]}$ is a module over the cohomology ring $H^\bt(\OG_\phi,\C)$ and its periodicity operator is given by the action of the Euler class $e_\phi$ (see~\cite{Cr:KT99, Ni:InvM90, Po:CRAS5}). In particular, $\bHP_\bt(\cA_\Gamma)_{[\phi]}=0$ whenever $e_\phi$ is nilpotent in $H^\bt(\OG_\phi, \C)$. In~\cite{Po:CRAS5} this action arises from the construction in~\cite{Po:CRAS4} of an explicit coproduct at the level of chains.   
\end{remark}

\begin{remark}
The nilpotence of the Euler class $e_\phi$  in $H^\bt(\Gamma, \Q)$ is closely related to the Bass and idempotent conjectures (see, e.g., \cite{Em:InvM98,Ji:KT95}). In particular, $e_\phi$ is rationally nilpotent for every infinite order element $\phi \in \Gamma$ whenever $\Gamma$ belongs to one of the following classes of groups: free products of Abelian groups, hyperbolic groups, and arithmetic groups. 
\end{remark}

\section{Conformal Invariants}\label{sec:Conformal-Invariants}
In this section, we shall use the results of the previous section to compute the conformal invariants produced by Theorem~\ref{thm:CC.conf.inv}. 
This will provide us with a new family of geometric conformal invariants associated with equivariant classes and conformal diffeomorphisms. 

Throughout this section we let $M^{n}$ be an even dimensional spin oriented compact manifold equipped with a non-flat 
conformal structure $\sC$. In addition, we let $G$ be the (identity component) of the group of 
orientation-preserving diffeomorphisms of $M$ preserving the conformal structure $\sC$ and the spin structure of $M$. 

\subsection{Conformal invariants from mixed equivariant cycles} Let $\phi \in G$ have finite order. We denote by $G_\phi$ its centralizer in $G$. As in Section~\ref{sec:cyclic-homology-crossed-product}, 
given any mixed equivariant chain $\omega \in C_{\ev/\odd}(G_\phi,M^\phi)$ we denote by $\omega_0$ its component in 
$C_0(G_\phi,\C)\otimes_{G_\phi} \Omega^{\ev/\odd}(M^\phi)\simeq \Omega^{\ev/\odd}(M^\phi)$. 

By Theorem~\ref{thm:HCcross-finite-order} we have an isomorphism $\eta^\phi: H^{G_\phi}_{\ev/\odd}(M^\phi)^\sharp \rightarrow \bHP_\bt(\cA_G)_{[\phi]}$. 
As  $\bHP_\bt(\cA_G)_{[\phi]}$ is a direct summand of $\bHP_\bt(\cA_G)$ we obtain a linear embedding 
$\eta^\phi:  H^{G_\phi}_{\ev/\odd}(M^\phi)^{\sharp} \hookrightarrow  \bHP_\bt(\cA_G)$. Given any metric $g \in \sC$ and 
any even mixed equivariant cycle $\omega \in C_{\ev}^{G_\phi}(G_\phi, M^\phi)$, we set 
 \begin{equation*}
  I_{g}^{\phi}(\omega):=\acou{\bCh(\sD_{g})_{\sigma_{g}}}{\eta^\phi[\omega]},
\end{equation*}where $[\omega]$ is the class of $\omega$ in $H_{\ev/\odd}^{G_\phi}(M^\phi)^\sharp$, and $\bCh(\sD_{g})_{\sigma_{g}}\in \bHP^{0}(\cA_{G})$ is the Connes-Chern character of the conformal twisted spectral triple 
$(C^{\infty}(M),L^{2}_{g}(M,\sS),\sD_{g})_{\sigma_{g}}$ associated with $g$. 

\begin{theorem}\label{thm:Confinv.main}
    Let $\phi \in G$ have finite order. In addition, let $\omega \in C_{\ev}^{G_\phi}(G_\phi, M^\phi)$ be a mixed equivariant cycle. 
    \begin{enumerate}
        \item  The scalar $I_{g}^{\phi}(\omega)$ is an invariant of the conformal class $\sC$. Moreover, it depends only on the class of 
        $\omega$ in the even mixed equivariant homology $H^{G_{\phi}}_{\ev}(M^{\phi})^{\sharp}$. 
    
        \item  For any $G$-invariant metric $g\in \sC$, we have 
        \begin{equation}
            I_{g}^{\phi}(\omega)=  \int_{M^{\phi}}\Upsilon^{\phi}(R) \wedge \omega_{0},
            \label{eq:Confinv.Igphi-omega}
        \end{equation}where $\Upsilon^{\phi}(R) $ is defined by~(\ref{eq:Conformal.Upsilon-phi}) and    the integral has the same meaning as 
        in~(\ref{eq:Conformal.integration-Mphi}).
    \end{enumerate}
\end{theorem}
\begin{proof}
    By Theorem~\ref{thm:CC.conf.inv} the Connes-Chern character $\bCh(\sD_{g})_{\sigma_{g}}\in \bHP^{0}(\cA_{G})$ is an invariant of the 
    conformal class $\sC$. Therefore, $I_{g}^{\phi}(\omega)=\acou{\bCh(\sD_{g})_{\sigma_{g}}}{\eta^\phi[\omega]}$ 
    depends only on the conformal class $\sC$ and the class of $\omega$ in $H^{G_{\phi}}_{\ev}(M^{\phi})^{\sharp}$. 
    This proves the first part. 
     
    As for the second part, let $g$ be a $G$-invariant metric in $\sC$. By Theorem~\ref{thm:LIF-Conformal-Geometry} the Connes-Chern 
   character $\bCh(\sD_{g})_{\sigma_{g}} \in \bHP^0(\cA_G)$  is represented by the CM 
    cocycle $\varphi^{\CM}\in \bC^0(\cA_G)^\sharp$ associated with $g$ by~(\ref{eq:Conformal.CM-cocycle}). 
    Thus,
    \begin{equation}
        I_{g}^{\phi}(\omega)=\acou{[\varphi^{\CM}]}{\eta^\phi[\omega]}. 
        \label{eq:Confinv.Ig-omega-vphiCM}
    \end{equation}

Recall that, given any ${\phi_0} \in G$, the differential form $\Upsilon^{{\phi_0}}(R) \in \Omega^\ev(M^{\phi_0})$ is defined by
\begin{equation*}
  \restr{\Upsilon^{{\phi_0}}(R)}{M^{\phi_0}_{a}}=(-i)^{\frac{n}{2}}(2\pi)^{-\frac{a}{2}}\hat{A}\left(R^{TM^{{\phi_0}}_{a}}\right)\wedge 
 \cV_{{\phi_0}}\left(R^{\cN^{{\phi_0}}_{a}}\right), \qquad a=0,2,\ldots,n,
\end{equation*}
where the differential forms $\hat{A}(R^{TM^{{\phi_0}}_{a}})$ and  $\cV_{{\phi_0}}(R^{\cN^{{\phi_0}}_{a}})$ are given by~(\ref{eq:Conformal.characteristic-forms}). Note that 
$\hat{A}(R^{TM^{{\phi_0}}_{a}})$ and  $\cV_{{\phi_0}}(R^{\cN^{{\phi_0}}_{a}})$ are characteristic differential forms for the vector 
bundles $TM^{{\phi_0}}_{a}$ and $\cN_{a}^{{\phi_0}}$, respectively.  In particular, these differential forms are closed, and hence, $\Upsilon^{{\phi_0}}(R)$ is closed differential form on $M^{\phi_0}$. 

Let $\psi\in G$ and set $\phi_1=\psi^{-1} \phi_0 \psi$. For $a=0,2,\ldots,n$, the diffeomorphism $\psi$ induces an isometric diffeomorphism from $M^{\phi_0}_a$ onto $M^{\phi_1}_a$, and,  hence,  $\psi_*\hat{A}(R^{TM_a^{\phi_0}})= \hat{A}(R^{TM_a^{\phi_1}})$. As the action of $\psi$ on $TM$ is isometric, it also induces a vector bundle isomorphism $\psi_*: \cN^{\phi_0}_a \rightarrow \cN^{\phi_1}_a$ such that $\psi_*\nabla^{\cN^{\phi_0}_a}=\nabla^{\cN^{\phi_1}_a}$ and $(\psi_*)^{-1} \phi_0^{\cN} \psi_* =\phi_1^{\cN}$ on $\cN_a^{\phi_1}$. Thus, 
\begin{equation*}
\psi_*\cV_{{\phi_0}}(R^{\cN^{{\phi_0}}_{a}}) = \psi_* 
{\det}^{-\frac{1}{2}}\left(1- \phi_0^{\cN}e^{-R^{\cN^{\phi_0}_{a}}}\right)=  
{\det}^{-\frac{1}{2}}\left(1- \phi_1^{\cN}e^{-R^{\cN^{\phi_1}_{a}}}\right)= \cV_{{\phi_1}}(R^{\cN^{{\phi_1}}_{a}}). 
\end{equation*}
It then follows that 
\begin{equation}
\psi_* \Upsilon^{{\phi_0}}(R) = \Upsilon^{{\psi^{-1}\phi_0 \psi}}(R) \qquad \forall \psi \in G. 
\label{eq:Confinv.equivariance-Upsilon}
\end{equation}

Let $\Omega^G(M)=(\Omega^G_\bt(M),d)$ be the $G$-equivariant de Rham complex of $M$ as defined in Section~\ref{sec:cyclic-homology-crossed-product}. Let $C:G\rightarrow \Omega_\ev(M)$ be the map defined by 
\begin{equation*}
 \acou{C(\phi_0)}{\xi} = \int_{M^{\phi_0}} \Upsilon^{{\phi_0}}(R) \wedge \iota_{M^{\phi_0}}^* \xi, \qquad \xi\in \Omega^\ev(M), \ {\phi_0} \in G. 
\end{equation*}
For each ${\phi_0} \in G$, this defines a current $C({\phi_0})\in \Omega_\ev(M)$ which is the image of the Poincar\'e dual of $\Upsilon^{{\phi_0}}(R)$ under the canonical embedding $(\iota_{M^{\phi_0}})_*: \Omega_\bt(M^{\phi_0}) \rightarrow \Omega_\bt(M)$. In particular, we see that $C({\phi_0})$ is a closed de Rham current. 
In addition, let $\psi\in G$ and $\phi_1= \psi^{-1} \psi_0 \psi$. Then the $G$-equivariance~(\ref{eq:Confinv.equivariance-Upsilon}) of $C$ implies that  
\begin{equation*}
 \acou{C({\phi_0})}{\psi_*\xi}= \int_{M^{\phi_0}} \psi^*\left( \Upsilon^{{\phi_0}}(R) \wedge \iota_{M^{\phi_0}}^* \xi\right) = \int_{M^{\phi_1}} 
  \Upsilon^{{\phi_1}}(R) \wedge \iota_{M^{\phi_1}}^* \xi =  \acou{C({\phi_1})}{\xi}.  
\end{equation*}
This shows that $C:G\rightarrow \Omega_\ev(M)$ is a closed equivariant current in $\Omega^G_\ev(M)$. 

We observe that it follows from~(\ref{eq:HCrossed.vphiC}) and~(\ref{eq:Conformal.CM-cocycle}) that $\varphi^\CM=\varphi_C$. Combining this with Proposition~\ref{prop:manifolds.pairing-equivarian-currents} and~(\ref{eq:Confinv.Ig-omega-vphiCM}) we then get
\begin{equation*}
 I_g^\phi(\omega)=\acou{\varphi_C}{\eta^\phi[\omega]}=\acou{C(\phi)}{\tilde{\omega}_0}, 
\end{equation*}
where $\tilde{\omega}_0 \in \Omega^\ev(M)$ is such that $\iota_{M^\phi}^*\tilde{\omega}_0=\omega_0$. In view of the definition of $C(\phi)$ we obtain
\begin{equation*}
 I_g^\phi(\omega)= \int_{M^\phi}  \Upsilon^{{\phi}}(R) \wedge \iota_{M^\phi}^*\tilde{\omega}_0=  \int_{M^\phi}  \Upsilon^{{\phi}}(R) \wedge \omega_0.
\end{equation*}
This proves~(\ref{eq:Confinv.Igphi-omega}) and completes the proof. 
\end{proof}

Combining Theorem~\ref{thm:Confinv.main} with Corollary~\ref{cor:finite-order.cap-HG-HP} we immediately get the following corollary. 

\begin{corollary}
Let $\phi \in G$ have finite order. In addition, let $\omega$ be a closed equivariant cocycle in $C^\ev_{G_\phi}(M^\phi)$ (resp., $C^\odd_{G_\phi}(M^\phi)$) and $\gamma$ a group cycle in $C_p(G_\phi,\C)$ with $p$ even (resp., odd). 
\begin{enumerate}
        \item  The scalar $I_{g}^{\phi}(\omega \frown \gamma)$ is an invariant of the conformal class $\sC$ which depends only on the class of 
        $\omega$ in the equivariant cohomology $H_{G_{\phi}}^{\ev/\odd}(M^{\phi})$ and on the class of $\gamma$ in the group homology $H_p(G_\phi, \C)$. 
    
        \item  For any $G$-invariant metric $g\in \sC$, we have 
        \begin{equation}
            I_{g}^{\phi}(\omega \frown \gamma)=  \int_{M^{\phi}}\Upsilon^{\phi}(R) \wedge (\omega \frown \gamma)_{0},
            \label{eq:Confinv.Igphi-omega-gamma}
        \end{equation}
\end{enumerate}
\end{corollary}

\subsection{Conformal invariants and equivariant characteristic classes}
Given $a\in \{0,2,\ldots, n\}$, let $E$ be a $G_\phi$-equivariant vector bundle over $M^\phi_a$ equipped with a connection $\nabla^E$. Let $\Ch_{G_\phi}(\nabla^E) \in C^\ev_{G_\phi}(M^\phi_a)$ be the equivariant Chern form associated with $\nabla^E$ (\emph{cf}.~Section~\ref{sec:group}). As mentioned in Section~\ref{sec:group}, it represents the equivariant Chern character of $E$ in $H^\ev_{G_\phi}(M^\phi_a)$. By construction 
$\Ch_{G_\phi}(\nabla^E)=(\Ch_{G_\phi}^p(\nabla^E))_{p\geq 0}$, where $\Ch_{G_\phi}^p(\nabla^E)\in C^p(G_\phi, \Omega^\bt(M^\phi))$ is given by~(\ref{eq:equivariant.Chern-form0}) for $p=0$ and by~(\ref{eq:equivariant.Chern-form+}) for $p\geq 1$. Thus, 
\begin{equation*}
\left(\Ch_{G_\phi}(\nabla^E) \frown 1\right)_0= \Ch_{G_\phi}(\nabla^E)(1)= \Ch(\nabla^E). 
\end{equation*}
Combining this~(\ref{eq:Confinv.Igphi-omega-gamma}) we then arrive at the following statement. 

\begin{proposition}
Let $\phi \in G$ have finite order. In addition, let $E$ be a $G_\phi$-equivariant vector bundle over $M^\phi_a$ for some $a\in \{0,2,\ldots, n\}$. Then, for every 
$G$-invariant metric $g\in \sC$ and every connection $\nabla^E$ on $E$, we have 
\begin{equation}
 I_g^\phi\left(\Ch_{G_\phi}(\nabla^E)\frown 1\right) =  \int_{M^{\phi}_a}\Upsilon^{\phi}(R) \wedge \Ch(\nabla^E). 
 \label{eq:Confinv.Ig-ChGEcap1-Mphia}
\end{equation}
\end{proposition}

The right-hand side of~(\ref{eq:Confinv.Ig-ChGEcap1-Mphia}) has a natural interpretation as a generalized Lefschtez number as follows. Let $E$ be a $G$-equivariant vector bundle over $M$ equipped with  a connection $\nabla^E$. For $a=0,2,\ldots, n$, we denote by $E^\phi_a$ the restriction of $E$ to $M^\phi_a$. This is a $G_\phi$-equivariant vector bundle over $M_a^\phi$. The connection $\nabla^E$ induces a $G_\phi$-invariant connection $\nabla^{E^\phi_a}$ on each vector bundle $E^\phi_a$. 
We also set $E^\phi=\bigsqcup_a E_a^\phi$. This is a $G_\phi$-equivariant (topological) vector bundle over $M^\phi$. We equip $E^\phi$ with the connection $\nabla^E=\bigoplus_a \nabla^{E^\phi_a}$. Recall that $C^\ev_{G_\phi} (M^\phi)= \bigoplus_a C^\ev_{G_\phi}(M^\phi_a)$ and  $H^\ev_{G_\phi}(M^\phi) = \bigoplus_a H^\ev_{G_\phi}(M^\phi_a)$. We then define the $G_\phi$-equivariant Chern character of $E^\phi$ by
\begin{equation*}
\Ch_{G_\phi}(E^\phi):= \sum_a \Ch_{G_\phi}(E^\phi)\in H^\ev_{G_\phi}(M^\phi). 
\end{equation*}
It is represented by the equivariant Chern form,
\begin{equation*}
\Ch_{G_\phi}(\nabla^{E^\phi}):= \sum_a \Ch_{G_\phi}(\nabla^{E_a^\phi}) \in C^\ev_{G_\phi} (M^\phi). 
\end{equation*}
We have $I_g(\Ch_{G_\phi}(\nabla^{E^\phi})\frown 1)= \sum_a I_g(\Ch_{G_\phi}(\nabla^{E^\phi_a})\frown 1)$. Therefore, by using~(\ref{eq:Confinv.Ig-ChGEcap1-Mphia}) we get
\begin{align}
I_g \left( \Ch_{G_\phi}(\nabla^{E^\phi})\frown 1\right) & = \sum_a \int_{M_a^\phi} \Upsilon^{\phi}(R) \wedge \Ch(\nabla^{E^\phi_a}) \nonumber\\
&  = \int_{M^\phi} \Upsilon^{\phi}(R) \wedge \iota_{M^\phi}^*\Ch(\nabla^E).  
\label{eq:Confinv.Ig-ChG-Ephicap1}
\end{align}

Let us further assume that  $E$ is equipped with a $G$-invariant Hermitian metric and 
$\nabla^{E}$ is a $G$-invariant Hermitian connection on $E$. In particular, any diffeomorphism $\phi\in G$ lifts to a unitary vector bundle 
isomorphism $\phi^{E}:E\rightarrow \phi_{*}E$. Given any $G$-invariant metric $g\in \sC$, we then get a unitary 
representation $\phi \rightarrow U_{\phi}^{E}$ of $G$ into $L^{2}_{g}(M,\sS\otimes E)$ given by
\begin{equation*}
    (U_{\phi}^{E}\xi)(x)=(\phi^{\sS}\otimes \phi^{E})[\xi(\phi^{-1}(x))] \qquad \text{for all $\xi\in 
    L^{2}_{g}(M,\sS\otimes E)$}.
\end{equation*}Note also that $U_{\phi}$ preserves the $\Z_{2}$-grading,
\begin{equation}
   L^{2}_{g}(M,\sS\otimes E)=L^{2}_{g}(M,\sS^{+}\otimes E)\oplus L^{2}_{g}(M,\sS^{-}\otimes E). 
   \label{eq:Confinv.Z2-grading-spinors}
\end{equation}

Let $\sD_{\nabla^{E}}:C^{\infty}(M,\sS\otimes E)\rightarrow C^{\infty}(M,\sS\otimes E)$ be the Dirac operator 
associated with the Hermitian connection $\nabla^{E}$. Namely, 
\begin{equation*}
    \sD_{\nabla^{E}}=\sD_{g}\otimes 1_{E}+(c\otimes 1_{E})\circ \nabla^{E},
\end{equation*}where $\sD_{g}:C^{\infty}(M,\sS)\rightarrow C^{\infty}(M,\sS)$ is the Dirac operator of $M$ and 
$c:T^{*}M\times \sS\rightarrow \sS$ is the Clifford action of $T^{*}M$ on $\sS$. With respect to the splitting~(\ref{eq:Confinv.Z2-grading-spinors}) 
the operator $\sD_{\nabla^{E}}$ takes the form,
\begin{equation*}
    \sD_{\nabla^{E}}=
    \begin{pmatrix}
        0 & \sD_{\nabla^{E}}^{-} \\
        \sD_{\nabla^{E}}^{+} & 0
    \end{pmatrix}. 
\end{equation*}
As $\sD_{\nabla^{E}}$ is $G$-invariant, the nullspaces $\ker \sD^{\pm}_{\nabla^{E}}$ are preserved by the action of $G$ 
associated with the representation $\phi\rightarrow U_{\phi}^{E}$. The \emph{equivariant index} 
$\ind \sD_{\nabla^{E}}:G\rightarrow \R$ (a.k.a.~generalized Lefschetz number) is then defined by
\begin{equation*}
    \ind \sD_{\nabla^{E}}(\phi)= \Tr \left[U^{E}_{\phi|\ker \sD^{+}_{\nabla^{E}}}\right]- \Tr\left[U^{E}_{\phi|\ker \sD^{-}_{\nabla^{E}}}\right] \qquad 
    \forall \phi\in G. 
\end{equation*}

When $\phi=\op{id}_{M}$ we recover the Fredholm index of $\sD_{\nabla^{E}}$. When all the fixed-points 
of $\phi$ are isolated this agrees with the Lefschetz number of $\phi$ with respect to the Dirac complex with coefficients in $E$ in the sense of Atiyah-Bott~\cite{AB:AM67, AB:AM68}. The equivariant index theorem of Atiyah-Segal-Singer~\cite{AS:IEO2, AS:IEO3} 
produces a geometric formula for the equivariant index. Namely,
\begin{equation*}
    \ind \sD_{\nabla^{E}}(\phi)=   \int_{M^{\phi}}\Upsilon^{\phi}(R) \wedge \iota^{*}_{M^{\phi}}\Ch(\nabla^{E}).
\end{equation*} Combining this formula with~(\ref{eq:Confinv.Ig-ChG-Ephicap1}) we then arrive at the following result.  

\begin{theorem}\label{thm:Confinv.characteristic-classes}
    Let $\phi \in G$ have finite order. In addition, let $E$ be a $G$-equivariant Hermitian vector bundle over $M$.  Then, for every $G$-invariant metric  $g \in \sC$ and every $G$-invariant Hermitian connection $\nabla^E$ on $E$, we have
     \begin{align*}
       I_g \left( \Ch_{G_\phi}(\nabla^{E^\phi})\frown 1\right)  = & 
       \int_{M^{\phi}}\Upsilon^{\phi}(R) \wedge \iota^{*}_{M^{\phi}}\Ch(\nabla^{E})\\ 
       = &\ind \sD_{\nabla^{E}}(\phi).
    \end{align*}In particular, when the fixed-points 
of $\phi$ are isolated, the conformal invariant 
    $I_g ( \Ch_{G_\phi}(\nabla^{E^\phi})\frown 1)$ is given by the Lefschetz number of 
    $\phi$ with respect to the Dirac complex with coefficients in $E$. 
\end{theorem}

\subsection{Conformal indices} 
The conformal invariants $I^\phi_g(\omega)$ arising from Theorem~\ref{thm:Confinv.main} shed a new light on the conformal indices of  Branson-{\O}rsted~\cite{BO:CM86, BO:DGA91}. To put things into context, let $P_g:C^\infty(M,E)\rightarrow C^\infty(M,E)$ be a natural differential operator of (even) order $m \leq n+2$ with positive-definite leading symbol which is the power of some conformally covariant operator. Examples include the Yamabe and Paneitz operators, and more generally the conformal powers of the Laplacian of Graham-Jenne-Mason-Sparling~\cite{GJMS:JLMS92}.  Non-scalar examples include the square of the Dirac operator and the conformal powers of the Hodge Laplacian of Branson-Gover~\cite{BG:CPDE05}. Such operators then  admit a short-time heat-semigroup expansion, 
\begin{equation*}
\Tr\left[ e^{-t P_g}\right] \sim  \sum_{j\geq 0} t^{j-\hat{m}}\int_M b_j(P_g)(x)|dx| \qquad \text{as $t\rightarrow 0^+$}. 
\end{equation*}
Here $\hat{m}=n+2-m$, $|dx|$ is the Riemmanian density, and  $b_j(P_g)(x)$ is given by a universal expression in terms of the metric $g$ and its derivatives. Consider the constant term in the asymptotic expansion above, i.e.,  
\begin{equation*}
J(P_g):= \int_M b_{\hat{m}}(P_g)(x)|dx|.  
\end{equation*}
Branson-{\O}rsted~\cite{BO:CM86} (see also Parker-Rosenberg~\cite{PR:JDG87}) showed that $J(P_g)$ is an invariant of the conformal class of $g$. Branson-{\O}rsted called it the  \emph{conformal index} of $P_g$.   

More generally, let $\phi:M\rightarrow M$ be a smooth isometry which lifts to a unitary vector bundle isomorphism $\phi^E:E\rightarrow E$. Then, in the same way as in Section~\ref{sec:Conformal-CC-character}, this gives rise to a unitary operator $U_\phi:L^2(M,E)\rightarrow L^2(M,E)$. In addition, the fixed-point set of $\phi$ is a disjoint union $M^\phi= \bigsqcup M_a^\phi$, where $M_a^\phi$ is a submanifold of $M$ of dimension $a$. It can be shown (see, e.g, \cite{Gi:LNPAM, PW:JNCG16}) that we have a short-time asymptotic, 
\begin{equation*}
\Tr\left[ e^{-t P_g}U_\phi\right] \sim  \sum_{j\geq 0} \sum_{0\leq a \leq n} t^{j-\hat{m}(a)}\int_{M_a^\phi} b_j^{(a)}(P_g)(x)|dx| \qquad \text{as $t\rightarrow 0^+$}. 
\end{equation*}
Here $\hat{m}(a)=a+2-m$ and $b_j^{(a)}(P_g)(x)$ is given by a universal expression in terms of the metric $g$ and the diffeomorphism $\phi$ and their partial derivatives.  
For $a=0,\ldots, n$, define
\begin{equation*}
J_\phi^a(P_g):= \int_{M_a^\phi} b_{\hat{m}(a)}^{(a)}(P_g)(x)|dx|. 
\end{equation*}

Let us denote by $\sC$ the conformal class of $g$ and by $\sC_\phi$ the sub-class consisting of $\phi$-invariant metrics in $\sC$.  Thus, $\hat{g}\in \sC_\phi$ if and only if $\hat{g}=k^{-2}g$ with $k \in C^\infty(M)$, $k>0$ such that $k\circ \phi=k$. 

\begin{proposition}[{Branson-{\O}rsted~\cite{BO:DGA91}}] 
For $a=0,\ldots, n$, the scalar $J_\phi^a(P_g)$ is an invariant of the sub-class $\sC_\phi$.   
\end{proposition}
 
This result provides us with invariants of the sub-class $\sC_\phi$ only. As we shall now see, in the case of the square of the Dirac operator, we actually obtain invariants of the whole conformal class $\sC$ as follows. 

We shall work in the original setup of this section, where $M$ is a compact spin oriented Riemannian manifold of even dimension $n$ and $G$ is a group of diffeomorphisms preserving the spin structure and a given (non-flat) conformal structure $\sC$. Let $\omega=\omega_0+\omega_2+\cdots + \omega_n$ be an even differential form on $M$, where $\omega_{2q}\in \Omega^{2q}(M)$. For $t>0$ set 
\begin{equation*}
c_t(\omega):=c(\omega_0)+tc(\omega_2)+\cdots + t^nc(\omega_n), 
\end{equation*}
where $c:\Omega^\bt(M)\rightarrow C^\infty(M,\End (\sS))$ is the Clifford representation. The results of~\cite{PW:JNCG16} then imply that, for any $G$-invariant metric $g \in \sC$, we have
\begin{equation*}
J^\phi_g(\omega):= \lim_{t\rightarrow 0^+}\Str\left[c_t(\omega) e^{-\sD_g}U_\phi\right] = \int_{M^\phi} \Upsilon^\phi(R) \wedge \iota_{M^\phi}^* \omega.
\end{equation*}
Note that if we take $\omega=1$ on $M_a^\phi$ and $\omega=0$ on $M_{b}^\phi$ with $b\neq a$, then we obtain a ``super-symmetric'' version of the conformal index $J_\phi^a(\sD_g^2)$. In fact, an extension of the arguments of~\cite{BO:DGA91} shows that $J^\phi_g(\omega)$ also is an invariant of the subclass $\sC_\phi$. 

Bearing this in mind, Theorem~\ref{thm:Confinv.main} asserts that, given any mixed equivariant cycle $\omega \in C_{\ev}^{G_\phi}(G_\phi, M^\phi)$, for every $G$-invariant metric $g \in \sC$, we have
\begin{equation*}
 I_g^\phi(\omega)= \int_{M^\phi} \Upsilon^\phi(R) \wedge\omega_0= J^\phi_g(\tilde{\omega}_0), 
\end{equation*}
where $\tilde{\omega}_0$ is any differential form in $\Omega^\ev(M)$ such that $\iota_{M^\phi}^* \tilde{\omega}_0$ is the component of $\omega$ in $C_0(G_\phi) \otimes_{G_\phi} \Omega^\ev(M^\phi)\simeq \Omega^\ev(M^\phi)$. Therefore, we see that the (generalized) conformal indices $J^\phi_g(\omega)$ actually lead us to invariants of the \emph{full} conformal class $\sC$. 

In addition, let $\cE$ be a finitely generated product module over $\cA_\Gamma$ and $\nabla^\cE$ a connection on $\cE$. Assume further that $\omega$ is such that $\eta^\phi(\omega)=\Ch(\cE)$. Then by the very definition of $I_g^\phi(\omega)$ we have
\begin{equation*}
 J^\phi_g(\tilde{\omega}_0)=I_g^\phi(\omega)= \acou{\bCh(\sD_g)_{\sigma_g}}{\eta^\phi(\omega)}= \acou{\bCh(\sD_g)_{\sigma_g}}{\Ch(\cE)} = \ind \sD_{\nabla^\cE}. 
\end{equation*}
This shows that the generalized conformal index $ J^\phi_g(\tilde{\omega}_0)$ is a true Fredholm index. This confirms the intuition of Branson-{\O}rsted that their conformal indices could be interpreted as some sort of Fredholm indices. 

\subsection{Deser-Schwimmer conjecture} The conformal invariants arising from Theorem~\ref{thm:Confinv.main} are not the same kind of conformal invariants as those considered by Alexakis~\cite{Al:DGCI} in his long solution of the conjecture of Deser-Schwimmer~\cite{DS:GCCAAD}. This conjectures gives a description of \emph{all} the global conformal invariants. Here by a global conformal invariant it is meant an invariant of the form, 
\begin{equation}
\cI_g= \int_M P(g)(x) |dx|,
\label{eq:Confinv.global-confinv-DS}
\end{equation}
where $P(g)(x)$ is a universal polynomial in the metric $g$ and its derivatives in such a way that the above integral is an invariant of the conformal class of the metric $g$. The difference with the invariants from Theorem~\ref{thm:Confinv.main} lies in the fact that the latter takes into account the action of the conformal diffeomorphisms, i.e., the conformal gauge group. As mentioned in the introduction, this of some importance in the context of string theory and conformal gravity. 

We obtain from Theorem~\ref{thm:Confinv.main} conformal invariants of the form~(\ref{eq:Confinv.global-confinv-DS}) when $\phi =1$. In this case, given any mixed equivariant cycle $\omega \in C_\ev(G,M)$, for every $G$-invariant metric $g\in \sC$, we have
\begin{equation*}
I_g^1(\omega)=\int_M \hat{A}(R^M)\wedge \omega_0.
\end{equation*}
Note that the conformal invariance of the above integral is also an immediate consequence of the well-known conformal invariance of the $\hat{A}$-form $ \hat{A}(R^M)$ (see Avez~\cite{Av:PANAS70} and Chern-Simons~\cite{CS:AM74}).

\end{document}